\documentclass[11pt]{amsart}

\usepackage{moreverb}
\usepackage[colorlinks,citecolor=red,urlcolor=red]{hyperref}
\usepackage{amsmath, amssymb}
\usepackage{multirow}
\usepackage{accents}
\usepackage{hhline}

\usepackage{amsmath,mathtools}
\usepackage{amsfonts}
\usepackage{amssymb}

\newtheorem{theorem}{Theorem}[section]
\newtheorem{corollary}[theorem]{Corollary}
\newtheorem{lemma}[theorem]{Lemma}

\newtheorem{remark}[theorem]{Remark}
\newtheorem{algorithm}{Algorithm}[section]
\newtheorem{assumption}{Assumption}[section]
\newtheorem{definition}{Deffinition}[section]

\newcommand{\red}[1]{{ \color{red} #1}}
\newcommand{\blue}[1]{{ \color{blue} #1}}
\newcommand{\ubar}[1]{\underaccent{\bar}{#1}}

\def\R{{\mathbb R}}

\def\bfv{{\bf v}}
\def\bfz{{\bf z}}

\def\bfu{{\bf u}}
\def\bff{{\bf f}}
\def\bfftil{\widetilde{\bf f}}
\def\Lambdat{\widetilde{\Lambda}}
\def\Lambdatt{{\Lambda}}

\def\bfF{{\bf F}}

\def\bPsi{\boldsymbol \Psi}

\def\calA{\mathcal A}
\def\calAt{{\mathbb A}}

\def\calIt{{\mathbb I}}
\def\calL{\mathcal L}
\def\calI{\mathcal I}
\def\calD{\mathcal D}
\def\calR{\mathcal R}

\def\sgn{\mathop{\rm sgn}}

\def\argmax{\mathop{\rm argmax}}
\def\sup{\mathop{\rm sup}}
\def\inf{\mathop{\rm inf}}

\def\kkk{{\mathrm k}}
\def\dH{\dot{H}}

\def\calT{{\mathcal T}}
\def\sumj{\sum_{j=1}^\infty}

\title[Solution of Systems with Fractional Power of SPD matrices]
{Optimal Solvers for Linear Systems with Fractional Powers of Sparse SPD Matrices}

\author[S. Harizanov, R. Lazarov, P. Marinov, S. Margenov, Y. Vutov]
{Stanislav Harizanov \and Raytcho Lazarov \and Pencho Marinov \and Svetozar Margenov \and Yavor Vutov}
\address{Institute of Information and Communication Technologies, Bulgarian Academy of 
Sciences, Acad. G. Bonchev, bl. 25A, 1113 Sofia, Bulgaria
(sharizanov@parallel.bas.bg)}
\address{Deptartment of Mathematics, Texas A\&M University, 
College Station,
TX 77843-3368, USA (lazarov@math.tamu.edu) and Institute of Mathematics and Informatics,
Bulgarian Academy of Sciences, Acad. G. Bonchev, bl. 8, 1113 Sofia, Bulgaria}
\address{Institute of Information and Communication Technologies, Bulgarian Academy of 
Sciences, Acad. G. Bonchev, bl. 25A, 1113 Sofia, Bulgaria (pencho@parallel.bas.bg)}
\address{Institute of Information and Communication Technologies, Bulgarian Academy of 
Sciences, Acad. G. Bonchev, bl. 25A, 1113 Sofia, Bulgaria (margenov@parallel.bas.bg)}
\address{Institute of Information and Communication Technologies, Bulgarian Academy of 
Sciences, Acad. G. Bonchev, bl. 25A, 1113 Sofia, Bulgaria (yavor@parallel.bas.bg)}

\begin{document}
\date{\today}

\begin{abstract}
In this paper we consider efficient algorithms  for solving  the algebraic equation
$\calA^\alpha \bfu=\bff$, $0< \alpha <1$, where $\calA$ is a properly scaled  
symmetric and positive definite matrix
obtained from  finite difference or finite element approximations of second order 
elliptic problems in $\R^d$, $d=1,2,3$. This solution is then written as
 $\bfu =\calA^{\beta-\alpha} \bfF$ with $\bfF=\calA^{-\beta} \bff$ with $\beta$ 
positive integer. The approximate solution method we propose and study 
is based on the best uniform rational approximation of the function  $t^{\beta-\alpha}$
for $0 < t \le 1$, and the assumption that
one has at hand an efficient method (e.g. multigrid, multilevel, or other fast algorithm) 
for solving equations like  $(\calA +c \calI)\bfu= \bfF$, $c \ge 0$.
The provided numerical experiments confirm the efficiency of the proposed algorithms.
\\[1ex]
\\[1ex]
\end{abstract}

\maketitle

\section{Introduction}

\subsection{Motivation for our study}
 
Let $\Omega$ be a bounded domain in $\R^d$, $d=1,2,3$,
with polygonal boundary $\Gamma=\partial \Omega=\bar \Gamma_D \cup \bar \Gamma_N$, where 
$\Gamma_D$ has positive measure. Let $q({x}) \ge 0$ in $\Omega$, and ${\bf a}({x}) \in \R^{d \times d}$ 
be a symmetric and positive definite  (SPD) matrix uniformly bounded in $\Omega$, i.e., 
\begin{equation}\label{ax}
c \xi^T \xi \leq
{\xi}^T {\bf a}({x}) \, {\xi} \leq
C \xi^T \xi \quad \forall {\xi} \in\R^{d}, \forall {x} \in \Omega,
\end{equation}
for some positive constants $c$ and $C$. 
Next, on $V \times V$,  $V:= \{ v \in H^1(\Omega): ~v({x})=0 ~\mbox{on} ~\Gamma_D \}$
define the  bilinear form
\begin{equation}\label{eqn:weak}
A(u,v) := \int_\Omega \big ({\bf a}({x}) \nabla u({x}) \cdot  \nabla v({x}) + q(x) u({x}) v({x}) \big )d{x}.
\end{equation}
Under the assumptions on $ {\bf a}({x})$, $q$, and $\Gamma$, the bilinear form is symmetric and coercive on $V$.
Further,  introduce  $\calT:L^2:=L^2(\Omega)\rightarrow V$, where for $f \in L^2(\Omega)$ 
the function $u=\calT f  \in V$ is the unique solution to  $A(u,\phi)=(f,\phi), \ \forall \phi\in V,$
   and $ {(v,u)}$, for $u,v \in L^2(\Omega)$ is the inner product in $L^2(\Omega)$.

The goal of this paper is to study methods and algorithms for solving
the finite element approximation of the operator equation
\begin{equation}\label{eq:fracPDE}
\calL^\alpha u = f, \ \  \calL = \calT^{-1}, \ \ 
\calL^{\alpha} u(x) = \sum_{i=1}^\infty \lambda_i^{\alpha} c_i \psi_i(x), \quad
\mbox{where} \quad
u(x) = \sum_{i=1}^\infty  c_i \psi_i(x),
\end{equation}
$\{ \psi_i(x) \}_{i=1}^\infty$ are the eigenfunctions of $\calL$,
orthonormal in $L_2$-inner product and  $\{ \lambda_i \}_{i=1}^\infty$
are the corresponding eigenvalues that are real and positive.

This definition  is general, but different from the definition of the fractional powers 
of elliptic operators with homogeneous Dirichlet data defined through Riesz potentials, 
which generalizes the concept of equally weighted left and right Riemann-Liouville fractional derivative
defined in one space dimension to the multidimensional case, see, e.g. \cite{BP15}. There is ongoing research 
about the relations of these two different definitions and their possible applications 
to problems in science and engineering, see, e.g. \cite{bates2006nonlocal}. 
However, we shall focus on the current definition and withhold comments and references
on such works.

Studying and numerically solving such problems 
is  motivated by the recent development in the fractional calculus and its numerous applications
 to  Hamiltonian chaos,  anomalous transport, and super-diffusion,  \cite{zaslavsky2002chaos}, 
 anomalous diffusion in complex systems such as turbulent plasma, convective rolls,
and zonal flow system, \cite{bakunin2008turbulence}, long-range interaction in elastic deformations,
\cite{silling2000reformulation}, nonlocal electromagnetic fluid flows, \cite{mccay1981theory}, 
image processing, \cite{gilboa2008nonlocal},  nonlocal evolution equations arising in materials science, 
\cite{bates2006nonlocal}.
A recent discussion about various anomalous diffusion models,
their properties and applicability to chemistry and engineering one can find in \cite{metzler2014anomalous}.
These applications lead to various types of fractional order partial differential equations that 
involve in general non-symmetric elliptic operators see, e.g. 
\cite{KilbasSrivastavaTrujillo:2006}.

An important subclass of such problems are the fractional powers of self-adjoint elliptic 
operators described below, which are nonlocal but self adjoint. 
Assume that a finite element method has been applied to approximate
the problem \eqref{eq:fracPDE} and this resulted in a certain algebraic problem.
The aim of this paper is to address the issue of solving such systems.
The rigorous error analysis of such approximation is a difficult task that  
is outside the scope of this paper. Such error bounds are derived under 
certain assumptions that are interplay between the data regularity, 
regularity pick-up of $\calL$ and the fractional power $\alpha$.
The needed justification is provided by the work of  Bonito and Pasciak, 
\cite{BP15}, for the problem \eqref{eqn:weak} in the case when $q=0$ 
and $\Gamma_D=\partial \Omega$, see also earlier work \cite{MU93}.
Following \cite{BP15}, one introduces 
$
\dH^\alpha:= \{ v \in L^2\  : \ \sumj  \lambda_j^{2\alpha}
    |(v,\psi_j)|^2<\infty\}
      $
and shows that $\dH^\alpha$ is a Hilbert space under the inner product
  $A_\alpha(v,w):=(\calL^{\alpha/2}v,\calL^{\alpha/2}w)$, 
{for all}  $v,w\in \dH^\alpha.$
To set up a finite element approximation of $\calL^\alpha u=f$ 
we first introduce its weak form: find $u \in \dH^\alpha$ such that
\begin{equation}\label{al-weak}
A_\alpha(u,v)=(f,v), \  \forall v \in \dH^\alpha.
\end{equation}
This problem has a unique solution 
\begin{equation}\label{frac-sol}
u=\calT^\alpha  f:=\sumj \lambda_j^{-\alpha} (f,\psi_j)\psi_j.
\end{equation}

Then for a finite elements space $V_h \subset \dH^1$ of continuous piece-wise 
polynomial functions defined on a quasi-uniform mesh with mesh size $h$ one
gets an approximate solution $u_h$ to \eqref{eq:fracPDE} by setting 
\begin{equation}\label{fracFEM}
\calA_h^\alpha u_h = \pi_h f, \quad \calA_h^{-1} = \calT_h, 
\end{equation}
where $\calT_h: V_h \to V_h$ is the solution operator to the FEM of finding $u_h \in V_h$ 
s. t. 
$$
A(u_h,v)=(f,v) \equiv (\pi_h f,v), \  \forall v \in V_h
$$ 
and $\pi_h: L^2(\Omega) \to V_h$ is the orthogonal projection on $V_h$.
Here for $\calT_h^\alpha$ we use an expression similar to \eqref{frac-sol}
but involving the eigenfunctions and eigenvalues of $\calT_h$.
As shown in \cite{BP15}, if the operator $\calT$ satisfies the regularity pick up, i.e., there is $s \in (0,1]$ s.t.
$ \| u \|_{\dH^{1+s}(\Omega)} \equiv \|\calT f \|_{\dH^{1+s}(\Omega)} \le c
    \|f\|_{\dH^{-1+s}(\Omega)}$
    and $\calL$ is a bounded map of $\dH^{1+s}(\Omega) $ into
    $\dH^{-1+s}(\Omega)$, then for $\alpha >s$ one has
\begin{equation}\label{FEMerror}
\| u - u_h\|_{L^2} =
\|\calT^{\alpha}f- \calT_h^{\alpha} \pi_h f\|_{L^2}  \le 
C  h^{2s}    \|f\|_{\dH^{2\delta}} , \ \ \delta \ge 0.
\end{equation}
The paper \cite[see, Theorem 4.3]{BP15}  contains more refined results depending on 
the relationship between smoothness of the data $\delta$, 
the regularity pick up $s$ and the fractional order $\alpha$. 
In the case of full regularity, $s=1$, the
best possible rate for $f \in L^2(\Omega)$ is,  
cf. \cite[Remark  4.1]{BP15},
$$
\| u - u_h\|_{L^2} \le C h^{2\alpha} |\ln h| \|f \|_{L^2}.
$$
The bottomline of the error estimates from \cite{BP15} is that, if  $f \in L^2(\Omega)$ and $s>\alpha$, then
$  \|u - u_h\|_{L^2}$ is  essentially $O(h^{2 \alpha})$.  Therefore,   the solution $u_h \in V_h$  of 
\eqref{fracFEM} is an approximation to the solution $u$ of \eqref{eq:fracPDE}.  
This fact makes our aim of solving the algebraic problem \eqref{fracFEM} justifiable. 
Finite element approximations of the elliptic problem \eqref{eq:fracPDE} 
and also more general non-symmetric problems were recently 
considered and studied by Bonito and Pasciak in \cite{BP17}.

\subsection{Algebraic problem under consideration}\label{algebraic}

  
  Now let $N$ be the dimension of $V_h$ and consider that a standard nodal basis is used. Let  
 $\calAt  \in \R^{N \times N}$ be a matrix representation of $\calA_h =\calT_h^{-1}$ 
 (defined in \eqref{eq:fracPDE}) and $\widetilde \bfu \in \R^N$ and $\bfftil \in \R^N$ be vector
representations through the nodal values of $u_h \in V_h$, and $\pi_h f \in V_h$, respectively. 
Then   we can recast the problem \eqref{fracFEM} 
in the following algebraic form: 
\begin{equation}\label{eq:fal}
\mbox{find } ~~ \bfu \in \R^N ~~ \mbox{such that } \quad
{\calAt}^\alpha \bfu =   \bfftil.
\end{equation}
The matrix $\calAt$ is symmetric and positive definite.
The fractional power $\calAt^\alpha$, $0 < \alpha <1$, of a symmetric positive definite matrix 
$\calAt$ 
is expressed through the eigenvalues and eigenvectors $\{ (\Lambdat_i, \bPsi_i) \}_{i=1}^N $
of $\calAt$. We assume that the eigenvectors are $l_2$-orthonormal, i.e.
$\bPsi_i^T \bPsi_j = \delta_{ij}$  and $\Lambdat_1 \le \Lambdat_2 \le \dots \Lambdat_N$.
The spectral condition number $\kkk(\calAt)=\Lambdat_N/\Lambdat_1 =O(h^{-2})$ 
for quasi-uniform meshes with mesh-size $h$.
Then $\calAt=WDW^T$ and $\calAt^\alpha = WD^\alpha W^T$,
where  
$W, D \in \R^{N \times N}$ are
defined as $W=[\bPsi_1^T, \bPsi_2^T, ..., \bPsi_N^T]$ and
$D=diag(\Lambdat_1, \dots, \Lambdat_N)$.
Then $ \calAt^{-\alpha}= W D^{-\alpha} W^T$ and the solution of 
$\calAt^\alpha \bfu =\bfftil$ can be expressed as
\begin{equation}\label{eqn:exacalg}
{\bfu} = \calAt^{-\alpha} \bfftil =W D^{-\alpha} W^T \bfftil.
\end{equation}
Obviously, we have the following standard 
{equality} for any $\beta \in \R$:
$$
\| \bfu \|_{\calAt^{\beta+\alpha}}= \| \bfftil \|_{\calAt^{\beta-\alpha}} \quad \mbox{with} 
\quad \| \bfu\|^2_{\calAt^\gamma} = \bfu^T \calAt^{\gamma} \bfu, \quad \gamma \in \R .
$$
The formula \eqref{eqn:exacalg}  could be used in practical computations if the eigenvectors and eigenvalues are
explicitly known and  the matrix vector multiplication with $W$ is
equivalent to a Fast Fourier Transform when $\calAt$ is a circulant matrix.
In such cases the computational complexity is almost linear,  $O(N \log N)$.
However, this limits  the applications to problems with constant coefficients in simple domains and
to the lowest order finite element approximations.
More general is the approach using an approximation of $\calAt$ with $H$-matrices combined with
Kronecker tensor-product 
{approximation.} 
This allows computations 
with almost linear complexity of the inverse of fractional power of 
a discrete elliptic operator in a hypercube $(0,1)^d\in\R^d$,
for more details, see  \cite{GHK05}. First attempt to apply
$H$-matrices to solving fractional differential equations in
one space variable is done in \cite{ZHAO2017}.
 
This work is related also to the more difficult problem of
stable computations of the matrix square root and other functions of matrices, see, 
e.g. \cite{druskin1998extended,Higham1997,Kenney1991}, 
where the stabilization of Newton method is achieved by using suitable
Pad\'e iteration.  However, in this paper we do not deal with evaluation of $\calAt^{\alpha}$,
instead we propose an efficient method
for solving the algebraic system $\calAt^{\alpha} \bfu =\widetilde \bff$, 
where $\calAt$ is an SPD matrix generated by 
approximation of second order elliptic operators.
Our research is also connected with the work done in \cite{
ilic2009numerical},
where numerical approximation of a fractional-in-space diffusion equation with 
non-homogeneous boundary conditions is considered.
In \cite{ilic2009numerical}, the proposed solver of the arising algebraic system 
relies on Lanczos method.   First, the adaptively preconditioned thick restart 
Lanczos procedure is applied to a system with $\calAt$.  The gathered spectral 
information is then used to solve the system with $\calAt^\alpha$. 
In \cite{druskin1998extended} an extended Krylov subspace method is proposed, originating by actions of the 
SPD matrix and its inverse. It is shown that for the same approximation quality, the variant of the extended subspaces 
requires about the square root of the dimension of the standard Krylov subspaces using only positive or negative matrix 
powers. A drawback of this method is the memory required to store the full dense matrix $W$ needed to perform the 
reorthogonalization.
Essentially, this approach and the method proposed and used in \cite{HMMV2016} rely 
on polynomial approximation of $t^{-\alpha}$, and the efficiency of the methods 
depends on the condition number of $\calAt$ and deteriorates  substantially for ill-conditioned
matrices.
  
\subsection{Overview of existing methods}

The numerical solution of nonlocal problems is rather expensive. 
The following three approaches (A1 - A3) are based on transformation 
of the original problem to a local elliptic or pseudo-parabolic problem, or on integral 
representation of the solution, thus increasing the dimension of  the original computational domain.  The Poisson 
problem is considered in the related papers refereed bellow.

\begin{itemize}
\item[{\bf A1}] 

A “Neumann to Dirichlet” map is used in \cite{CNOSaldago_2016}. 
Then, the solution of fractional Laplacian problem is obtained by 
$u(x) = v(x, 0)$ 
where $v : \Omega\times\mathbb R_+ \rightarrow \mathbb R$ is a solution of the equation
$$
-div\left ( y^{1-2\alpha}\nabla v(x,y)\right ) = 0, ~~~ 
(x,y)\in \Omega\times \mathbb R_+ ,
$$
where $v(\cdot,y)$ satisfies the boundary conditions 
of (\ref {eqn:weak}) $\forall y\in \mathbb R_+$,
$
\lim_{y\rightarrow\infty} v(x,y) = 0, ~~~ x\in\Omega ,
$
as well as
$
\lim_{y\rightarrow 0^+} \left (- y^{1-2\alpha} v_y(x,y)\right ) = f(x), 
~~~ x\in\Omega.
$
The variational formulation of this equation is  well posed in the related weighted 
Sobolev space. The finite element 
approximation uses the rapid decay of the solution $v(x,y)$ in the 
$y$ direction, thus enabling truncation of the semi-infinite
cylinder to a bounded domain of modest size. The proposed multilevel 
PCG method is based on the Xu-Zikatanov identity \cite{XuZ2002}.

\item[\bf A2]

A fractional Laplacian is considered in \cite{Vabishchevich14,Vabishchevich15} assuming the boundary condition 
$$
a(x){\frac{\partial u}{\partial n}} + \mu (x) u = 0, ~~~ x\in \partial\Omega,
$$
which ensures ${\calL} = {\calL}^* \ge \delta {\calI}$, $\delta >0$. Then the solution of 
the nonlocal problem $u$ can be 
found as
$
u(x)=w(x,1), ~ w(x,0)=\delta^{-\alpha}f,
$
where $w(x,t), 0<t<1$, is the solution of pseudo-parabolic equation 
$$
(t{\calD} + \delta {\calI} ) {\frac{dw}{dt}} + \alpha {\calD}w = 0,
$$
and ${\calD} = {\calL} - \delta {\calI} \ge 0$. Stability conditions 
are obtained for the fully discrete schemes under consideration.
A further development of this approach is presented in \cite{LV17} where
the case of fractional order boundary conditions is studied.

\item[\bf A3]
The following representation of the solution operator of  (\ref {eqn:weak})  is used in \cite{BP15}:
$$
{\calL}^{-\alpha} = \frac{2\sin (\pi\alpha)}{\pi} \int_0^\infty  t^{2\alpha - 1} 
\left ({\calI} +t^2 {\calL}\right )^{-1} dt 
$$
The authors introduce an exponentially convergent  
quadrature scheme. Then, 
{\blue (the approximation of $u$}
only involves evaluations 
of $({\calI}+t_i {\calA})^{-1} f$, where $t_i\in (0,\infty)$ is 
related to the current quadrature node, and where $\calI$ and ${\calA}$ stand 
for the identity and the finite element stiffness matrix corresponding
to the Laplacian.
A further development of this approach is available in \cite{BP17}, where the theoretical analysis is
extended to  the class of regularly accretive operators.
\end{itemize}

There are various other problems leading to systems with fractional 
power of sparse symmetric and positive definite matrices. For an illustration we give 
the following examples.  Consider a non-overlapping 
domain decomposition (DD) for the 2D model Laplacian problem  on a regular 
mesh with an interface along a  single mesh line.   The elimination of all degrees of freedom 
from the interior of the two subdomains reduces the solution to a system of equations on the
interface, or the Schur complement system.  
The matrix of this system is 
spectrally equivalent  to $\calAt^{1/2}$, where $\calAt=  h^{-1} tridiag (-1,2,-1)$, 
see, e.g. \cite{Nepomn_1991}. When in the preconditioned conjugate gradient (PCG) method 
FFT is used to solve the system with $\calAt^{1/2}$, then the DD preconditioner
has almost optimal complexity. 

In \cite{HMMV2016} 
the approach discussed in this paper
was used to optimally solve 
the equation \eqref{eq:fal} with $\alpha=0.5$, when $\calAt$ 
is an SPD matrix and belongs to a class of particular weighted graph Laplacian
models used in the volume constrained 2-phase segmentation of images.
Then, after rescaling, so that the spectrum of the  rescaled matrix is in $(\Lambda_1,1]$,
the best uniform polynomial approximation of $t^{-1/2}$, $t\in [\Lambda_1,1]$,
$\Lambda_1$ well separated from 0, was used to construct an optimal  solver.

\subsection{Our approach and contributions} 
\label{contibutions}
In Section \ref{sec:problem} we introduce the mathematical 
problem and present the idea of the proposed algorithm.  Let $\Lambdatt$ be an upper bound for the 
spectrum of $\calAt$, namely, $ \Lambdat_j \le \Lambdatt$, $j=1, \dots, N$. We 
rescale the system to the form
\begin{equation}\label{algebraic}
\calA^\alpha \bfu = \bff, \quad \mbox{where} \quad \calA = \calAt/\Lambdatt
\quad \mbox{ and } \quad \bff =\bfftil/\Lambdatt^\alpha, 
\end{equation} 
so that the spectrum of $\calA$, $\Lambda_j = \Lambdat_j/\Lambdatt$, 
$j=1,\cdots, N$ is in $(0,1]$. 
We summarize the properties of the rescaled matrix $\calA$ in the following  assumption.
\begin{assumption}\label{aa1}
 ~~$\calA$ is a symmetric, positive definite matrix and 
its spectrum is in the interval $(0,1]$.
\end{assumption}
Next, we argue that instead of the system $\calA^{\alpha} \bfu= \bff $ one can solve the equivalent system 
$\calA^{\alpha  - \beta} \bfu= \calA^{-\beta} \bff $ with $ \beta \ge 1$ an integer. Then  the
 idea is to approximately evaluate $\calA^{\beta -\alpha} \bff $ by
$P_k(\calA)(Q_k(\calA))^{-1} \bff$, for $k $ integer, where $P_k(t)Q_k^{-1}(t):=r^\beta_\alpha(t) $ is the best
uniform rational approximation 
(BURA) of  $t^{\beta -\alpha}$ on the interval $(0,1]$, see for 
more details Subsection \ref{sec:rational}.  In Section 2 we discuss the methods for computing $r^\beta_\alpha(t)$,
its approximation properties and questions regarding the implementation of $P_k(\calA)(Q_k(\calA))^{-1} \bff$.

The properties of the best rational approximation have been 
an object of numerous studies. In particular, the distribution of the poles, zeros, and extreme points, 
and the asymptotic behavior of the error 
$E_\alpha(k,k;\beta)=\max_{t \in [0,1]} \left |t^{\beta-\alpha} - r^\beta_\alpha(t) \right |$
when $k \to \infty$ are given in \cite{stahl2003}.
For example, it is known that all poles lie on the negative real line and the error decays 
exponentially in $k$, namely, is 
$O(e^{- c\sqrt{k}})$, $c>0$, 
see, relation \eqref{eq:asymp-approx}.
%

In Theorem \ref{th:error} and Remark \ref{k-bound} we show that 
we can balance the 
finite element error \eqref{FEMerror} with the error of the BURA \eqref{errorU} so that 
the total error is $O(h^{2\alpha})$ when $k\approx \frac{\beta^2}{\pi^2(\beta-\alpha)} |\ln h|^2$.
Thus, the feasibility of the method  {will depend} 
on the possibility to 
address two key issues: 
(1) for a given $0<\alpha<1$  and chosen $k$ integer, 
compute the BURA $P_k(t)/Q_k(t)$ and 
(2)  implement $P_k(\calA)(Q_k(\calA))^{-1}\bff$ efficiently.

To find the BURA for $t^{\beta - \alpha}$ we apply the modified Remez algorithm, 
see, \cite{PGMASA1987,CheneyPowell1987}. The main difficulty in implementing the algorithm is its instability for large $k$,
outlined for example in \cite{Dunham1984}. 
Our experience shows that for moderate $k=5,6,7$  we can compute the BURA
using double precision and equivalent representation by Chebyshev polynomials \eqref{eq1bua}. We note that 
for $\alpha < 0.5$ we have better approximation and the algorithms for finding BURA have better  stability, but still an 
outstanding issue is the stability of the computations for $k>9$.

Due to Lemma \ref{BURA_caracter} (see also \cite[Lemma 2.1]{saff1992asymptotic})
one can represent the rational function as a sum of partial fractions, so that the
implementation of $P_k(\calA)(Q_k(\calA))^{-1}$ will involve inversion of 
$\calA - d_j I $, $d_j \le 0$ for $j=0,1, \dots, k$, see representation \eqref{eq:simplified}.
The integer parameter $k \ge 1$ is the number of partial fractions of the best uniform 
rational approximation 
 {$r^1_\alpha(t)$  of $t^{1 -\alpha}$ on the interval $(0,1]$.  The 
most general form of this method for $\beta=1$
leads to \eqref{eq:Rgeneral}.
The nonpositivity of $d_j$ ensures that the systems with 
$ \calA - d_j I $, $d_j \le 0$ can be solved efficiently having at hand 
some efficient solver for systems with $\calA$. The positivity of $c_j$ 
means that the BURA approximation $r^1_\alpha(\calA)$ is positive. 
The behavior of $c_j>0$ and the related numerical round-off stability is 
further discussed in Remark \ref{rm:aaa}.
}
%
Since we can compute BURA efficiently for $k \le 10$ and small $\alpha$ 
we have developed, studied and experimented with a new concept of multi-step 
BURA algorithm, outlined in Section \ref{sec:multi-step BURA}.

Finally, in Section \ref{sec:NumTest} we present numerical experiments
 that illustrate the efficiency of the proposed algorithms. The first 
group of tests concerns scaled (normalized) matrices corresponding to 
1D Poisson equation where the exact solution is known and the BURA 
approximations are exactly computed. This setting allows numerically 
confirming the sharpness of theoretical estimates. In particular, some 
promising approximation properties are observed when different powers of 
$\calA$ are involved in the multi-step BURA.  The experiments with 2D 
fractional Laplacian illustrate the theoretical results concerning balancing 
the rescaling effect in \eqref{errorU}. Finally, we present 3D numerical 
experiments involving jumping coefficients. The solution $\bfu$ is unknown, 
while $\bfu_r$ is computed by a preconditioned conjugate gradient (PCG) 
solver that uses algebraic multigrid as a preconditioner. A multi-step 
setting of BURA is used to confirm the robustness with respect to the PCG 
accuracy.

\section{Solution strategy}
\label{sec:problem}

\subsection{The idea and theoretical justification of the method}
\label{sec:rational}


The goal of this study is to present a new robust solver of optimal 
complexity for solving the system \eqref{eq:fal} 
for a large class of sparse SPD matrices assuming that such a solver
is available for $\alpha=1$. This assumption holds in a very general 
setting, when $\calA$  is generated by finite element or finite difference approximation of second 
order elliptic operators. Such matrices are used in the numerical
tests presented in Section \ref{sec:NumTest}.
Note that $\calA^\alpha$ is dense and  in general not known.
This means in particular that the standard iterative solution methods are not
applicable since even in the case when 
$\calA^\alpha \bfv$
is computable just
one such computation requires $O(N^2)$ arithmetic operations.

We consider the class of rational functions 
$$
\mathcal R(m,k):=\{r(t)=P_m(t)/Q_k(t)\,:\,P_m\in{\mathcal P}_m,Q_k \in {\mathcal P}_k\},
$$ 
where ${\mathcal P}_k$ is the set of all polynomials of degree $k$.  
For a given univariate function $g(t)$, $ 0 \le t \le 1$, 
the minimizer $r^\ast(t) = \frac{P^{\ast}_m(t)}{Q^{\ast}_k(t)} \in \mathcal R(m,k)$  of the problem
\begin{equation}\label{eq2_bua}
\min_{ r \in \mathcal R(m,k)}  \max_{t \in [0,1]} \left | g(t) - r(t)  \right | = \max_{t \in [0,1]} \left | g(t) - 
r^\ast(t)  \right | 
\end{equation}
is called {\it Best Uniform Rational Approximation}  (BURA) of $g(t)$.
\begin{definition}\label{def-BURA}
~~The minimizer 
$r^\beta_\alpha(t)=\frac{P^{\ast}_m(t)}{Q^{\ast}_k(t)}$  for $g(t) = t^{\beta -\alpha}$ 
is called {\it  $\beta$-Best Uniform Rational Approximation}
($\beta$-BURA) and its  error is denoted by 
$$
E_\alpha(m,k;\beta):=\max_{t \in [0,1]} \left |t^{\beta-\alpha} - r^\beta_\alpha(t) \right |.
$$
\end{definition}

Our algorithm is based on  the following Lemma:

\begin{lemma}\label{l:BURAerror} 
~~Let $\calA$ satisfy the assumption \ref{aa1} and  $ \bfF=\calA^{-\beta}\bff$ so that
$\bfu =  \calA^{\beta-\alpha} \bfF$. Let 
$r_\alpha^\beta(t)$ be the  best uniform rational approximation of $t^{\beta - \alpha}$ on $[0,1]$
and consider  $ \bfu_r=r^\beta_\alpha(\calA)\bfF$ to be an approximation to $\bfu$. Then the following 
bound for the error holds true
\begin{align}\label{bound}
\|\bfu_r-\bfu\|_{\calA^{\gamma}}
 \le E_\alpha(m,k;\beta)\|\bff\|_{\calA^{\gamma -2\beta}} \ \ \forall \gamma \in \R.
\end{align}
\end{lemma}
\begin{proof}
 Consider the representation of $\bfF$ with respect to the eigenvectors of $\calA$, 
  $
  \bfF =\sum_{i=1}^N  F_i\bPsi_i 
$  so that  
$  \| \bfF \|^2_{\calA^\gamma} = \sum_{i=1}^N  \Lambda^{\gamma}_i  F_i^2, $ 
for any $  \gamma \in \mathbb R $.
Since $r^\beta_\alpha$ is analytic in 
$(0,1]$, it has convergent Maclaurin expansion there and 
therefore 
$r^\beta_\alpha(\calA)\bPsi_i=r^\beta_\alpha(\Lambda_i)\bPsi_i$, $ i=1,\dots,N$.
Using the orthonormal
property of the eigenvectors  $\bPsi_i^T \bPsi_j := \langle \bPsi_i, \bPsi_j \rangle = \delta_{ij}$ we get easily 
\begin{align*} 
\|\bfu_r-\bfu\|^2_{\calA^{\gamma}} & =
\| r^\beta_\alpha(\calA) \bfF  -   \calA^{\beta-\alpha}\bfF \|^2_{\calA^{\gamma}} \\
&  =
\left\langle\sum_{i=1}^N(\calA^{\gamma}(r^\beta_\alpha(\calA)-\calA^{\beta-\alpha}) F_i\bPsi_i,
\sum_{i=1}^N(r^\beta_\alpha(\calA)-\calA^{\beta-\alpha}) F_i\bPsi_i\right\rangle\\ \notag
&= \sum_{i=1}^{N}
F^2_i\Lambda^{\gamma}_i\left(r^\beta_\alpha(\Lambda_i)-\Lambda^{\beta-\alpha}_i\right)^2 
 \le \max_{t \in [0,1]} | r^\beta_\alpha(t)-t^{\beta -\alpha}|^2 \sum_{i=1}^N  \Lambda^{\gamma}_i F_i^2 \\
&  \le   E_\alpha(m,k;\beta)   \| \bfF \|^2_{\calA^{\gamma}}.
\end{align*}
To complete the proof  take into account that $\bfF = \calA^{-\beta}\bff$.
\end{proof}
It is important to keep in mind that the above estimate is for the scaled system 
\eqref{algebraic}, where
$\bff=\bfftil/\Lambdatt^\alpha$  and $\calA = \calAt/\Lambdatt$. 
As a corollary we get the following bound for the solution through the original (unscaled) data:
\begin{corollary}\label{c:BURAerror}
~~The following estimate holds true for the solution of \eqref{eq:fal}:
\begin{equation}\label{errorU}
\|\bfu_r-\bfu\|_{\calAt^{\gamma}} 
\le  E_\alpha(m,k;\beta) \Lambdatt^{\beta -\alpha}  \| \bfftil \|_{\calAt^{\gamma -2\beta}}.
\end{equation}
\end{corollary}

Among various classes of best rational approximations, the diagonal sequences 
$r \in {\calR}(k,k)$ of the Walsh table of $t^\alpha$, $0 < \alpha <1$ are 
studied in greatest detail, see, e.g. \cite{stahl2003,varga1992some}. 
The existence of best uniform rational approximation, 
the distribution of the poles, zeros, and extreme points, 
and the asymptotic behavior of $E_\alpha(k,k;\beta)$ when $k \to \infty$  
are well known.  For example, Theorem 1 from \cite{stahl2003}  shows that 
%
$$ 
\lim_{k \to \infty} e^{2 \pi \sqrt{(\beta-\alpha)k} }    E_\alpha(k,k; \beta) =  4^{1+\beta -\alpha} |\sin \pi 
(\beta-\alpha) | 
$$ 
holds for any $0< \alpha <1 \le \beta$, $\beta$ integer.
%
This could be written as an approximate relation (used in practice) 
\begin{equation}\label{eq:asymp-approx}
  E_\alpha(k,k; \beta) \approx  4^{1+\beta -\alpha} |\sin \pi (\beta-\alpha) | e^{-2 \pi \sqrt{(\beta-\alpha)k} }  .  
\end{equation}
In order to convince ourselves in the feasibility of practical use of \eqref{eq:asymp-approx} in 
 Table \ref{ta:error} we present the obtained results for $E_\alpha(k,k;\beta)$ for various $k$ and $\beta$
when the BURA $ P^\ast_k(t)/Q^\ast_k(t) $  is computed using 
the Remez algorithm, see for more details Section \ref{sec:approx}. 
The results show that for relatively small $k$  we can get good approximation $E_\alpha(k,k;\beta)$.
It is quite clear from this table that $E_\alpha(k,k;\beta)$ is a couple of orders of
magnitude smaller for $\beta=3$ compared with $\beta =1$.
However, this comes at a cost. From  \eqref{errorU} we observe that: 
(1)  the errors are measured in two different ways and 
(2) the scaling factor enters into the play with a negative impact on the accuracy for larger $\beta$.
Thus, the values $\beta=2$ and $\beta=3$ have not been used in our computations, we 
are giving the approximation properties of BURA for these values just for comparison.
%
In Table \ref{ta:error} in parenthesis we show the computed 
values from the asymptotic formula \eqref{eq:asymp-approx}. These and other computations,
see e.g. \cite[Tables 2.1 -- 2.7]{varga1992some}, show 
that asymptotic formula is quite accurate and the relation \eqref{eq:asymp-approx} could be used for fairly low $k$.

\begin{table}[h!]
\begin{center}
\begin{tabular}{||c|c|c|c|r|r||}
\hline \hline
$\alpha$&$E_\alpha(5,5;1)$&$E_\alpha(6,6;1)$&$E_\alpha(7,7;1)$&$E_\alpha(5,5;2)$&$E_\alpha(5,5;3)$\\ \hline
   0.75 & 2.7348E-3~(3.60E-3)& 1.4312E-3& 7.8650E-4~(9.82E-4)& 1.9015E-6& 6.8813E-8\\
   0.50 & 2.6896E-4~(3.88E-4)& 1.0747E-4& 4.6037E-5~(6.28E-5)& 9.5789E-7& 5.5837E-8\\
   0.25 & 2.8676E-5~(4.16E-5)& 9.2522E-6& 3.2566E-6~(4.47E-6)& 2.8067E-7& 2.4665E-8\\
\hline \hline
\end{tabular}\\[1ex]
\caption{Errors $ E_\alpha(k,k;\beta)$ of BURA $ P^\ast_k(t)/Q^\ast_k(t) $ of $ t^{\beta -\alpha}$ on $[0,1]$}
\label{ta:error}
\end{center}
\end{table}
%
The theoretical foundation of the proposed method is the following lemma, that is an immediate
consequence of Corollary \ref{c:BURAerror}:   
\begin{lemma}\label{th:error}
For $\beta \ge1$ integer there is a constant $C_{\alpha,\beta}>0$ and 
an integer $k_0 \ge 1$ such that for $k \ge k_0$ the following error bound holds true
$$
\|\bfu_r-\bfu\|_{\calAt^{\gamma}} 
\le  
 C_{\alpha,\beta} \Lambdatt^{\beta -\alpha}  e^{-2 \pi \sqrt{(\beta-\alpha)k} } 
 \| \bfftil \|_{\calAt^{\gamma -2\beta}}.
 $$
\end{lemma}
\begin{remark}\label{k-bound}
~ The error in solving the algebraic problem \eqref{eq:fal} using the proposed method should be
balanced with the approximation error given by \eqref{FEMerror}. In the case of second order problems
on a quasi-uniform mesh with size $h$ 
we have $\Lambdatt\approx h^{-2}$. 
Then in case of 
best possible convergence rate of the finite element solution, namely, $O(h^{2\alpha}|\ln h|)$,
cf. \cite[Remark 4.1]{BP15}, we can take 
$$
k \approx \frac{\beta^2}{\pi^2(\beta-\alpha)}
|\ln h|^2,
$$ 
and get the total error $O(h^{2\alpha}|\ln h|)$  (this includes the
finite element approximation error and error 
of approximately solving the algebraic problem).
\end{remark}
This represents the foundation of the method we propose and study in this paper. 
The feasibility of such approach depends substantially on the possibility to  
efficiently compute $r_\alpha^\beta(\calA)\bff$.
One possible implementation is proposed in next 
{subsection.}


\subsection{Efficient implementation of the method}\label{sec:implementation}

We first bring some important 
{facts} 
about the best uniform rational approximation
$r^1_\alpha(t)$,  $m=k$, of $t^{1-\alpha}$ on $[0,1]$ for $0< \alpha <1$, see, e.g. 
\cite{saff1992asymptotic, stahl2003}.
\begin{lemma}\label{BURA_caracter} ~ (\cite[Lemma 2.1]{saff1992asymptotic}) 
~Let $m=k$ and $0< \alpha <1$. Then the following statements are valid:
\begin{enumerate}
\item The best rational approximation $r^1_\alpha(t)$ has numerator and denominator of 
exact degree $k$;
\item All $k$ zeros $\zeta_1, \dots, \zeta_k$ and poles $d_1, \dots, d_k$ of $r^1_\alpha$ are real
and negative and are interlacing, i.e. with appropriate numbering one has
$$
0 > \zeta_1 > d_1 > \zeta_2 > d_2 > \dots > \zeta_k > d_k > - \infty;
$$
\item The function $t^{1-\alpha} - r^1_\alpha(t)$ has exactly $2k+2$ extreme points
$\eta_1, \dots, \eta_{2k+2}$ on $[0,1]$ and with appropriate numbering we have
$$
0=\eta_1 < \eta_2 < \dots < \eta_{2k+2} = 1,
$$
$$
\eta_j^{1-\alpha} -r^1_{\alpha}(\eta_j)= (-1)^j E_\alpha(k,k;1), \ \ j=1, \dots, 2k+2.
$$
\end{enumerate}
\end{lemma}
%
Then we introduce  $d_0=0$ so that $r^1_\alpha(t)$ is represented  as a sum of partial fractions
\begin{equation}\label{eq:simplified}
t^{-1}r^1_\alpha(t) := \displaystyle \frac{1}{t} ~ {P^\ast_k(t) \over Q^\ast_k(t) }=
 \frac{1}{t} ~ {\sum_{j=0}^k p_jt^j \over  \sum_{j=0}^k q_jt^j} =
 \sum\limits_{j=0}^{k}{c_j \over t-d_j}.
\end{equation} 
These notations are used in the tables below. 

%
This Lemma allows us to have the following implementation of the method:\\
Step 1: Find all poles $0=d_0 > d_1 >  d_2 > \dots >  d_k $;\\
Step 2: Find the representation \eqref{eq:simplified} of $r^1_k(t)$ as a sum of partial fractions;\\
Step 3: Compute the approximate solution  by 
\begin{equation}\label{eq:Rgeneral}
\bfu_r := \calA^{- 1}r^1_\alpha(\calA) \bff   =
\sum_{j=0}^{k} {c_j}{(\calA - d_j \calI)^{-1}} \bff \\
\end{equation}

This shows that to find 
$\bfu_r= r^1_\alpha(\calA) \calA^{-1}\bff$ we need to
solve one system $\calA \bfv= \bff$ and $k$ separate independent systems 
$(\calA - d_i \calI) \bfv = \bff$ for $i=1, \dots k$  with SPD matrices $ \calA - d_i \calI$.

\begin{remark}\label{rm:aaa}
~Our numerical tests show that often we achieve accuracy of 
$E_\alpha(k,k;1) \approx 10^{-4}$ or better 
with $k=5$. For example, for $\alpha =0.5$ and $k=5$ we get 
$E_\alpha(k,k;1)=2.69*10^{-4}$. The coefficients of the 1-BURA are
given in Table \ref{ta:coeff2} and they result in the following partial fractions representation: 
\begin{eqnarray}\label{fractions}
\frac{r^1_{0.5}(t)}{t} = {P^\ast_5(t) \over t\,Q^\ast_5(t)}
 =  \frac{0.0002689}{t}\ +\ \frac{0.0055848}{t+0.0000122} 
                        &+& \frac{0.0272036}{t+0.0006621} \\ \nonumber
                       \ +\ \frac{0.0965749}{t+0.0127955}
                       \ +\ \frac{0.3202068}{t+0.1626313}
                        &+& \frac{2.5105702}{t+3.2129222}.                                          
\end{eqnarray}
\end{remark}
We note that  the nonpositivity of $d_j$ ensure that the sytems 
$(\calA - d_j \calI) \bfv = \bff$ can be solved efficiently and the positivity of $c_j$, 
shown in \cite[Theorem 1]{HM17},
guarantees no loss of significant digits due to subtraction of large numbers.

\begin{remark}\label{r:product}
~ Using representation of $r^1_\alpha(t)$ by partial fractions is just one possible way to 
compute $  r^1_\alpha(\calA) \calA^{-1} \bff$. Another possibility is to use 
the zeros and the poles to compute consecutively the factors in the formula
$$
 r^1_\alpha(\calA) \calA^{-1} \bff 
= c_0 \prod_{j=1}^k (\calA - \zeta_j \calI)(\calA -d_j \calI)^{-1} \calA^{-1} \bff.
$$
Due to the interlacing of the zeroes and the poles this will lead to stable computations.
Moreover, it will preserve the monotonicity of the solution (see \cite{HM17}), which is a desired feature 
in some applications.  The only substantial difference is that the computations with partial fractions can be
done in parallel.
\end{remark}

Now we present some examples of BURA $r^1_\alpha$ within the class 
$\mathcal R(k,k)$ for $\alpha =0.75, 0.5, 0.25$.
%
In Tables \ref{ta:coeff1} - \ref{ta:coeff3} 
we show the computed coefficients \eqref{eq:simplified} of  $r^1_\alpha$ for 
$\alpha =\ 0.75,\, 0.5,\, 0.25$.

\begin{table}[h!]
\begin{center}
\begin{tabular}{|| c | c | c | c | c ||}
\hline \hline
$j$ &$p_j$ & $ q_j$   &   $ c_j$ & $d_j$ \\ \hline
0&1.98976E-20&7.27576E-18&2.73478E-03& 0.00000E+00\\
1&5.72723E-12&2.23068E-10&2.28202E-02&-3.27111E-08\\
2&1.76902E-06&1.96679E-05&6.31334E-02&-1.14734E-05\\
3&5.86823E-03&2.45055E-02&1.45484E-01&-8.15164E-04\\
4&4.89312E-01&8.76333E-01&3.05748E-01&-2.80630E-02\\
5&1.40048E+00&1.00000E+00&8.60558E-01&-8.47443E-01\\
 \hline \hline
\end{tabular}\\[1ex]
\caption{
The coefficients in the representation \eqref{eq:simplified} of the best rational approximation $ 
P^\ast_5(t)/Q^\ast_5(t) $
of $ t^{1 -\alpha}$ on $[0,1]$, $\alpha=0.75$; from Table \ref{ta:error} we have $E_\alpha(5,5;1)=$2.7348E-3
}
\label{ta:coeff1}
\end{center}
\end{table}
\begin{table}[h!]
\begin{center}
\begin{tabular}{|| c | c | c | c | c ||}
\hline \hline
$j$ &$p_j$ & $ q_j$   &   $ c_j$ & $d_j$ \\ \hline
0&1.45636E-14&5.41485E-11&2.68957E-04& 0.00000E+00\\
1&2.87192E-08&4.51317E-06&5.58483E-03&-1.22320E-05\\
2&2.69846E-04&7.06745E-03&2.72036E-02&-6.62106E-04\\
3&8.87796E-02&5.67999E-01&9.65749E-02&-1.27955E-02\\
4&1.91330E+00&3.38902E+00&3.20207E-01&-1.62631E-01\\
5&2.96041E+00&1.00000E+00&2.51057E+00&-3.21292E+00\\
 \hline \hline
\end{tabular}\\[1ex]
\caption{
The coefficients in the representation \eqref{eq:simplified} of the best rational approximation $ 
P^\ast_5(t)/Q^\ast_5(t) $
of $ t^{1 -\alpha}$ on $[0,1]$, $\alpha=0.5$; from Table \ref{ta:error} we have $E_\alpha(5,5;1)=$2.6896E-4
}
\label{ta:coeff2}
\end{center}
\end{table}
\begin{table}[h!]
\begin{center}
\begin{tabular}{|| c | c | c | c | c ||}
\hline \hline
$j$ &$p_j$ & $ q_j$   &   $ c_j$ & $d_j$ \\ \hline
0&3.45490E-12&1.20483E-07&2.86755E-05& 0.00000E+00\\
1&1.58841E-06&7.90871E-04&1.27509E-03&-1.59055E-04\\
2&4.13469E-03&2.10628E-01&9.58752E-03&-3.96701E-03\\
3&5.71109E-01&4.81422E+00&4.86842E-02&-4.47241E-02\\
4&7.40426E+00&1.11966E+01&2.55382E-01&-3.97136E-01\\
5&9.24225E+00&1.00000E+00&8.92729E+00&-1.07506E+01\\
 \hline \hline
\end{tabular}\\[1ex]
\caption{
The coefficients in the representation \eqref{eq:simplified} of the best rational approximation $ 
P^\ast_5(t)/Q^\ast_5(t) $
of $ t^{1 -\alpha}$ on $[0,1]$, $\alpha=0.25$; from Table \ref{ta:error} we have $E_\alpha(5,5;1)=$2.8676E-5
}
\label{ta:coeff3}
\end{center}
\end{table}

Based on these results we can make the following observations:
\begin{enumerate}
\item  Tables \ref{ta:coeff1} - \ref{ta:coeff3}  show that the approximation of the action of $\calA^{-\alpha}$ via 
the application of    
the operator $\calA^{-1} r^1_\alpha(\calA)$ involves solving six systems of linear equations 
with SPD matrices; we have assumed that each evaluation of $(\calA - d_j I)^{-1}\bff $
can be computed approximately by
PCG method with optimal complexity.

\item Summing these six solutions is a stable process since the coefficients in the sum of fractions \eqref{fractions}
are small and positive and there should not expect any loss of accuracy (or stability) that might come 
from subtracting large numbers.

\end{enumerate}

\section{Best uniform rational approximation of $t^{\beta-\alpha}$} 
\label{sec:approx}

\subsection{Theoretical background and numerical methods}\label{sec:theory}

Let $r^\beta_\alpha(t)=\frac{P^{\ast}_m(t)}{Q^{\ast}_k(t)}$ be the $\beta$-BURA of $t^{\beta-\alpha}$ for $t\in[0,1]$.
For given $k$ and $m$, the rational function has the following representation:
\begin{equation}\label{eq1bua}
{P^\ast_m(t)\over Q^\ast_k(t)} = 
{\sum\limits_{j=0}^{m} p_j  t^j \over \sum\limits_{j=0}^{k} q_j t^j} =  
{\sum\limits_{j=0}^{m} \bar{p}_j T_j(2t-1) \over \sum\limits_{j=0}^{k} \bar{q}_j T_j(2t-1)}
= {\bar{P}^\ast_m(s)\over \bar{Q}^\ast_k(s)}
\end{equation}
where $T_0(s)=1,\ T_1(s)=s,\ \ldots, T_j(s)=2s T_{j-1}(s)-T_{j-2}(s),\ j=2,3,\ldots ;\ s\in [-1,1]$
 are orthogonal base functions. These are the well-known Chebyshev polynomials and in our case
 $s=2t-1$, since $t\in [0,1]$. 

According to the theory, for the class of continuous functions on $[0,1]$ the element of best uniform approximation 
exists. Due to the equioscillation theorem, there are at least $(m+k+2)$ points $\{\eta_i\}_1^{m+k+2}$,
where the error $r^\beta_\alpha(t)-t^{\beta-\alpha}$ have extremes and the sign alternates. 
We use orthogonal base functions, because there are numerical difficulties (instabilities) for finding $r^\beta_\alpha$ 
in the standard monomial basis $\{t^j\}$ (see \cite{Dunham1984}). 
The benefits of working in Chebyshev basis are illustrated in Table~\ref{tabLastS}, where the maximal values of $m$ and 
$k$ for which the element of best uniform approximation of $t^{1-\alpha}$ can be successfully computed (the algorithm 
converges) are documented. The left pairs in the table correspond to best polynomial approximation ($k=0$), while the 
right ones correspond to best $(k,k)$-rational approximation ($m=k$). It is evident that apart from the choice of base 
functions, calculations heavily depend on the used precision for arithmetic operations (single, double, quadruple). 
This is due to the non-differentiability of $t^{1-\alpha}$ at zero. The function is only 
$(1-\alpha)$-H\"older continuous (i.e. in $C^{0,1-\alpha}[0,1]$) 
and as a result most of the extreme points $\{\eta_i\}_1^{2k+2}$ of 
$r^1_\alpha$ are clustered in a neighborhood of zero to account for the steep slope there. For example, when $k=5$ and 
$\alpha=0.75$ the first two points are $\eta_1=0$ and $\eta_2\approx 3\cdot10^{-9}$, while the ninth point value is 
still 
just $\eta_9\approx 0.05$. Therefore, to accurately compute $\{\eta_i\}$ and capture the sign 
changes of $r^1_\alpha(t)-t^{1-\alpha}$ between them one must use high precision arithmetics.   
This fact has been known and attempts to compute the BURA and the error $E_\alpha(k,k;\beta)$ 
for large $k$ has required using high precision arithmetic, e.g., see \cite{varga1992some}.

\begin{table}[h!]
\begin{center}
\begin{tabular}{||l|c|c|c|c||}
 \hline\hline
  Precision & Base     & $\alpha=0.25$ & $\alpha=0.50$ & $\alpha=0.75$ \\ \hline
  Single  & ${t^j}   $ &  (8,0), (4,4) & (8,0), (4,4)  & (8,0), (2,2)  \\
  Single  & ${T_j(s)}$ & (40,0), (3,3) &(50,0), (2,2)  &(50,0), (2,2)  \\
  Double  & ${t^j}   $ & (19,0), (6,6) &(19,0), (4,4)  &(19,0), (2,2)  \\
  Double  & ${T_j(s)}$ & (60,0), (5,5) &(60,0), (5,5)  &(50,0), (4,4)  \\
  Quadro  & ${t^j}   $ & (35,0), (6,6) &(35,0), (4,4)  &(35,0), (2,2)  \\
  Quadro  & ${T_j(s)}$ &(95,0), (11,11)&(95,0), (11,11)&(95,0), (7,7)\\ \hline \hline
\end{tabular}\\[1ex]
\caption{Maximal values $(m,k)$ for which Algorithm~\ref{alg1} converges.}\label{tabLastS}
\end{center}
\end{table}

\subsection{Modified Remez algorithm for computing BURA}
\label{sec:algorithm}

We suggest the following (modified Remez) algorithm for finding
the $\beta$-BURA of $t^{\beta-\alpha}$ on $[0,1]$.
To improve the stability of the approximation method we use the presentation 
\eqref{eq1bua}, so that we work with the function $f(s)=\left (\frac{1+s}{2} \right ) ^ {\beta - \alpha}$
for $s\in [-1,1]$ (see \cite{PGMASA1987}, \cite{CheneyPowell1987}).

\noindent
\hrulefill
\begin{algorithm}\label{alg1}
~~{\bf Input:} $(\alpha,\beta)$, $(m,k)$, $N$ (maximal number of algorithm iterations), $V$ (maximal number of inside 
iterations for solving the non-linear system in Step 3(ii)), $\delta > 0$ (accuracy).
\\{\bf Initialization:} $\ell$, $s^{(0)}$, and $\bar{r}_0$, satisfying  
\begin{itemize}
\item $\ell=m+k+2$.
 \item $\left\{ s_i^{(0)}\right\}_{i=1}^{\ell}$ -- strictly monotonically increasing sequence in $[-1,1]$. 
\item $\bar{r}_0(s)={\bar{P}_m(s) \over \bar{Q}_k(s)}\;:\;\left(f(s_i^{(0)})-\bar 
r_0(s_i^{(0)})\right)\;/\;\left(f(s_{i+1}^{(0)})-\bar r_0(s_{i+1}^{(0)})\right)<0,\;\forall i=1,\dots,\ell-1$.
\end{itemize}
{\bf FOR} $n=1,2,\dots$ {\bf DO}
\begin{enumerate}
 \item {\em Updating the equioscillation point set:} FOR $i=1,\dots,\ell$ DO
 \begin{itemize}
  \item[(i)] $\ubar{\tau}^{(n)}_i:=\displaystyle\sup_{-1\le\tau\le s^{(n-1)}_i}\left\{f(\tau)=\bar 
r_{n-1}(\tau)\right\},\qquad
  \bar{\tau}^{(n)}_i:=\displaystyle\inf_{s^{(n-1)}_i\le\tau\le 1}\left\{f(\tau)=\bar r_{n-1}(\tau)\right\}.$
  \item[(ii)] $s^{(n)}_i=\displaystyle\argmax_{\ubar{\tau}^{(n)}_i\le s\le \bar{\tau}^{(n)}_i}|f(s)-\bar 
r_{n-1}(s)|,\qquad \eta^{(n)}_i=|f(s^{(n)}_i)-\bar r_{n-1}(s^{(n)}_i)|$. END FOR 
\item[(iii)] $s^{(n)}_\ast=\displaystyle\argmax_{-1\le s\le 1}|f(s)-\bar r_{n-1}(s)|,\qquad 
\eta^{(n)}_\ast=|f(s^{(n)}_\ast)-\bar r_{n-1}(s^{(n)}_\ast)|$.
\item[(iv)] IF $\left(s^{(n)}_\ast\notin\{s^{(n)}_i\}_1^\ell\right)$ THEN FIND $j$ s.t. 
$s^{(n)}_j<s^{(n)}_\ast<s^{(n)}_{j+1}$.\\ IF $\left(\sgn\left(f(s^{(n)}_\ast)-\bar 
r_{n-1}(s^{(n)}_\ast)\right)=\sgn\left(f(s^{(n)}_j)-\bar r_{n-1}(s^{(n)}_j)\right)\right)$ THEN 
$s^{(n)}_j=s^{(n)}_\ast$.\\ ELSE $s^{(n)}_{j+1}=s^{(n)}_\ast$. 
 \end{itemize}
\item {\em Convergence check:} IF $\left(\max_i\eta^{(n)}_i-\min_i\eta^{(n)}_i<\delta\right)$ OR 
$\left(n=N+1\right)$ {\bf STOP}. ELSE
\item {\em Updating the rational approximation:} Solve iteratively the non-linear system  
 $$ \left(f\left(s_i^{(n)}\right) - \bar{r}_{n}\left( s_i^{(n)}\right) \right) = (-1)^i E_n, \qquad 
i=1,\ldots,\ell$$
for the unknown $E_n$ and the coefficients of $\bar{r}_n$: 
\begin{itemize}
 \item[(i)] $(E^0_n\;,\;\bar{r}^0_n)=(E_{n-1}\;,\;\bar{r}_{n-1})$.
 \item[(ii)] FOR $v=1,2,\dots$ DO: Solve the $\ell\times\ell$ linear system of equations
 $$\sum\limits_{j=0}^{m} \bar{p}_j^{(n,v)} T_j(s_i^{(n)}) -  \left(f(s_i^{(n)}) - (-1)^i E_n^{(v-1)} \right)  
\sum\limits_{j=1}^{k} \bar{q}_j^{(n,v)} T_j(s_i^{(n)})  +(-1)^i E_n^{(v)}=f(s_i^{(n)}).$$
IF $\left| E_n^{(v)}-E_n^{(v-1)}\right|<\epsilon$, OR $v>V$ GO to (iii). ELSE $v=v+1$ and REPEAT.
\item[(iii)] $\bar{r}_n(s)=\left (\sum_{j=0}^{m} \bar{p}_j^{(n,v)} T_j(s) \right 
)/\left(1+\sum_{j=1}^{k}\bar{q}_j^{(n,v)} T_j(s) \right)$. 
\item[(iv)] $E_n=E_n^{(v)}$.
\end{itemize}
\item $n=n+1$. GO to Step 1. 
\end{enumerate}
{\bf Output:} $n;\quad \bar{r}^\beta_\alpha(s)=\bar{r}_n;\quad 
E_\alpha(m,k;\beta)=|E_n|;\quad s^\ast=\{s_i^{(n)}\}_{i=1}^{\ell}$.\\
\end{algorithm}
\noindent
\hrulefill

Then we take $\eta=(s^\ast+1)/2$ and get $r^\beta_\alpha(t)$ from $\bar{r}^\beta_\alpha(s)$ using 
\eqref{eq1bua}. 

Several remarks, concerning the computer implementation of the proposed algorithm are in order: 
(1) In the {\em 
Initialization} step, we usually take $s^{(0)}$ to be uniformly sampled on $[-1,1]$, while for the derivation of an 
admissible $\bar{r}_0$ we apply least-squares optimization techniques. All the rational functions $\bar{r}_n$ are 
normalized with respect to the constant term in the denominator, i.e.,
$$ \bar{r}_n(s)=\left (\sum\limits_{j=0}^{m} \bar{p}_j^{(n)} T_j(s) \right )/\left ( 1+\sum\limits_{j=1}^{k} 
\bar{q}_j^{(n)} T_j(s) \right )\qquad\forall n\ge0.$$
(2) In order to increase the computational efficiency of the algorithm, we compute neither the sequences 
$\ubar{\tau}^{(n)}$ and $\bar{\tau}^{(n)}$ in {\em Step 1(i)} nor the local extrema $s^{(n)}$ in {\em Step 1(ii)}. 
Instead, we search for the maximal value of $|f(s)-\bar r_{n-1}(s)|$ on a small, discretized interval around 
$s^{(n-1)}_i$, decrease the mesh size, and repeat the process several times around the current maximizer. Such 
simple localization techniques seem to work fine for our numerical examples. \\
(3) In {\em Step 3(ii)} we apply 
Aitken-Steffensen acceleration but instead of $E^{(v-1)}_n$ for the system splitting we use a combination of the 
values$\{E^{(v-i)}_n\}_{i=1}^3$ from the previous three steps.

\section{Numerical accuracy and Multi-step BURA method}\label{sec:multi-step BURA}
In this section we investigate the numerical accuracy of the proposed algorithm and a multi-step generalization of the 
BURA-method. The analysis, presented here is theoretical in nature, so we consider the full generality of the proposed 
solution strategy, namely the $(m,k)$ $\beta$-BURA approximation. 
\subsection{Properties of the fractional decomposition}\label{subsec:preliminaries}
For given $(m,k,\beta)$ the partial fraction decomposition of $t^{-\beta}r^\beta_\alpha(t)$ has the general form
\begin{equation}\label{eq:FactFull}
t^{-\beta}r^\beta_\alpha(t) =
\sum_{j=0}^{m-k-\beta} b_j\, t^{j}+
\sum_{j=1}^{\beta}\frac{c_{0,j}}{t^j}    +\sum_{j=1}^{p_1}
\frac{c_j}{t - d_j} + \sum_{j=1}^{p_2} \frac{B_j t + C_j}{(t-F_j)^2 + D_j^2}
\end{equation}
where $k=p_1+2\,p_2$. We always consider triples for which $m<k+\beta$. One reason for such a parameter 
constraint comes from 
the fact that $t^{-\beta}r^\beta_\alpha$ has a leading term of degree $t^{m-k-\beta}$, while it approximates the power 
function $t^{-\alpha}$, $\alpha>0$. Another reason is the numerical 
simplification of \eqref{eq:FactFull}, where the 
index set for the first sum becomes empty. 
In all our numerical examples the denominator of $r^\beta_\alpha$ has no 
complex roots, thus we concentrate on the case $p_2=0$ from now on. 
Then  $\beta$-BURA can be rewritten in the following way:
\begin{equation}\label{eq:other}
 \frac{1}{t^\beta} ~ {P^\ast_m(t) \over Q^\ast_k(t) }= 
{ {\sum\limits_{j=0}^{m} p_j~t^j} \over t^{\beta} \left ({ \sum\limits_{j=0}^{k} q_j~t^j} \right )} =  
 \sum\limits_{j=1}^{\beta}{c_{0,j} \over t^j} 
+ \sum\limits_{j=1}^{k}{c_j \over t-d_j}. 
\end{equation}
The first representation in (\ref{eq:other}) is the best approximation written 
as a standard rational function, while the second one is 
its partial fraction decomposition \eqref{eq:FactFull}, the way this approximation is used in the 
implementation of the method.

Let 
\begin{equation*}
{P^\ast_m(t) \over Q^\ast_k(t) }=\sum_{j=0}^{\beta-1}b^\ast_j t^j+\sum_{j=1}^k \frac{c^\ast_j}{t-d_j}. 
\end{equation*}
We have used that $\beta>m-k$ and we set the extra coefficients 
$\{b^\ast_j\}_{m-k+1}^{\beta-1}$ to zero  whenever $m-k<\beta-1$. 
Then, straightforward computations give rise to
\begin{equation}\label{eq:other2}
\frac{1}{t^\beta} ~ {P^\ast_m(t) \over Q^\ast_k(t) 
}=\sum_{j=1}^\beta\frac{b^\ast_{\beta-j}}{t^j}+\sum_{j=1}^k\left(\frac{c^\ast_j/d^\beta_j}{t-d_j}-\sum_{i=1}^\beta 
\frac{c^\ast_j/d^{\beta-i+1}_j}{t^i} \right) 
\end{equation}
Comparing the coefficients in front of the corresponding terms in \eqref{eq:other} and \eqref{eq:other2}, we derive
\begin{equation}\label{eq:coeff representation}
c_{0,j}=b^\ast_{\beta-j}-\sum_{i=1}^k c^\ast_i/d^{\beta-j+1}_i,\qquad
c_j=c^\ast_j/d^\beta_j.
\end{equation}
Various useful identities follow from \eqref{eq:coeff representation}. We want to highlight a couple of them. Due to 
the Chebyshev's equioscillation theorem
\begin{equation}\label{eq:c0,beta}
c_{0,\beta}=b^\ast_0-\sum_{i=1}^kc^\ast_j/d_j= {P^\ast_m(0) \over Q^\ast_k(0)}=\frac{p_0}{q_0}=\pm E_\alpha(m,k;\beta).
\end{equation} 
In particular, for $(k,k;1)$ we have $c_0=E_\alpha(k,k;1)$ due to Lemma~\ref{BURA_caracter}.

The second one is
\begin{equation}\label{eq:sum coeff}
c_{0,1}+\sum_{i=1}^k c_i=b^\ast_{\beta-1}=\left\{\begin{array}{ll} p_m/q_k, & m-k=\beta-1;\\ 0, & 
m-k<\beta-1.\end{array}\right.
\end{equation}
Finally, \eqref{eq:coeff representation} allows for stable numerical computations of the coefficients $\{c_j\}$, 
as the fractional decomposition of $r^\beta_\alpha$ can be accurately derived in Chebyshev basis.
\subsection{Accuracy Analysis}\label{42-accur}

In this subsection we briefly discuss issues related to the numerical accuracy of the developed framework. We do not go 
into details, since thorough analysis of the algorithm is outside the scope of the paper. However, certain 
observations in this direction are worth mentioning, so that the reader can make conclusions for the full picture. 

Lemma~\ref{l:BURAerror} quantifies the error between $\bfu_r= r^\beta_\alpha(\calA) \calA^{-\beta} \bff$ and 
$\bfu=  \calA^{-\alpha} \bff$. All the estimations are under the assumption that $\bfu_r$ can be exactly computed by an 
optimal numerical solver. Within the adopted setup $m<k+\beta$ and $p_2=0$, $\bfu_r$ has the following representation
\begin{equation}\label{eq:fract.decomp.ur}
\bfu_r=\sum_{i=1}^{\beta}c_{0,i}\calA^{-i}\bff + \sum_{i=1}^{k} c_i(\calA - d_i I)^{-1} \bff
  :=\sum_{i=1}^{\beta}c_{0,i}\bfv_{0,i}+\sum_{i=1}^{k} {c_i}\bfv_i, 
\end{equation}
which is the corresponding simplification of \eqref{eq:Rgeneral}. The practical derivation of $\bfu_r$ involves 
$k+\beta$ applications of such a solver, that independently solves each of the involved large-scale linear systems
with a right-hand-side  $\bff$. The numerical stability of each solution process 
depends on the condition number of the underlined linear operator. 

\begin{remark}\label{remark: conditioning}
~The matrix $\calA-d_iI$ is better conditioned than $\calA$ whenever $d_i<0$ or $d_i>\Lambda_1+\Lambda_N$.
If $d_i>\Lambda_1+\Lambda_N$, the condition numler $\mathrm{k}(-\calA+d_iI)$ is uniformly bounded dependinding 
only on $d_i$.   
\end{remark}

Since we are interested in operators $\calA$ which spectrum is normalized to lie inside 
$(0,1]$, and is not well-separated from zero, in every numerical example we have 
$\Lambda_1\approx0$ and $\Lambda_N\approx1$. 
The poles of $r^\beta_\alpha$ are outside of (a neighborhood of) the unit interval $[0,1]$, 
therefore all the $d_i$'s naturally satisfy the condition in Remark~\ref{remark: conditioning}. Thus, the numerical 
computation $\bfv^{MG}_i$ of $\bfv_i=(\calA - d_i I)^{-1}\bff$ , $i=1,\dots,k$ is a stable process.  

When $\beta=1$, using the notation \eqref{eq:simplified}, we observe that the most time-consuming procedure is the 
derivation of $\bfv^{MG}_0$ that corresponds to inverting $\calA$ ($\bfv_0=\calA^{-1}\bff$). In our numerical 
experiments, we use algebraic multigrid (AMG) \cite{HENSON2002155} as a preconditioner in a CG method. The same 
preconditioner can be used to operators like $\calA-d_iI$. We already 
observed that, provided $m=k$, all the coefficients $c_i$ are positive and sum to $p_m/q_m$ (see \eqref{eq:sum 
coeff}). The ratio $p_m/q_m$ increases with $\alpha$ (see Tables \ref{ta:coeff1} - \ref{ta:coeff3} for $m=k=5$) but 
seems 
to always be $\mathcal{O}(1)$. Therefore, for this setting the accuracy of the numerical derivation of $\bfu^{MG}_r$ is 
proportional to the accuracy of inverting $\calA$. Hence, numerics are trustworthy in general and unsubstantial 
additional errors for $\bfu^{MG}_r-\bfu$ are accumulated.

\subsection{Further analysis in the case $\beta >1$}
\label{sec:more}

When $\beta>1$, since $\kkk(\calA^\beta)=\kkk(\calA)^\beta$, 
we need to find an approximate solution of a system with much worst condition number than the original system.
This could cause loss of stability (or loss of accuracy).
Since we solve $\calA^\beta\bfv=\bff$ iteratively 
via $\beta$ consecutive applications of the AMG solver as a preconditioner for $\calA$,
we need to analyze the stability of such computational strategy.

\begin{lemma}\label{thm:2-step PCG accuracy}  
~~Let $\mu, \nu, \varepsilon >0$  be given and $\bfv=\calA^{-n}\bff$, where $n\ge2$. Assume 
that $\bfz^{MG}$ is
a numerical solution for $\bfz=\calA^{-(n-1)}\bff$, while $\bfv^{MG}$ 
is a numerical solution for $\bar\bfv=\calA^{-1}\bfz^{MG}$. Then
\begin{equation}\label{eq:output}
 \frac{\|\bfz^{MG} -\bfz\|_{\calA^{-1}}}{\| \bfz \|_{\calA^{-1}}} \le \mu \varepsilon, \quad  
  \frac{\|\bfv^{MG}- \bar \bfv \|_\calA}{\| \bar \bfv \|_{\calA}}\le\nu\varepsilon
\quad \mbox{imply} \quad  
\frac{\|\bfv^{MG}-\bfv \|_\calA}{\| \bfv \|_\calA} \le (\mu+\nu+\mu\nu\varepsilon)\;\varepsilon.
\end{equation}
\end{lemma}
\begin{proof}
Applying triangle inequality, we derive 
\begin{align*}
\|\bfv^{MG} - \bfv\|_\calA\le 
\|\bfv^{MG} - \bar \bfv\|_\calA+ \|\bar \bfv -\bfv\|_\calA 
                                    & \le\nu\varepsilon\|\bar \bfv \|_\calA + \|\bar \bfv-\bfv\|_\calA \\
                                     & \le\nu\varepsilon\|\bfv\|_\calA
                                     +(1+\nu\varepsilon)\|\bar \bfv-\bfv\|_\calA;\\
\|\bar \bfv-\bfv\|_\calA=\|A^{-1}(\bfz^{MG} -\bfz)\|_\calA=\| \bfz^{MG} - \bfz \|_{\calA^{-1}} 
                                     & \le\mu\varepsilon\|\bfz\|_{\calA^{-1}}
                                     =\mu\varepsilon\|\calA 
\bfv\|_{\calA^{-1}}=\mu\varepsilon\|\bfv\|_\calA.
\end{align*}
These imply                                      
$
\|\bfv^{MG} -\bfv\|_\calA\le\left(\mu+\nu+\mu\nu\varepsilon\right)\varepsilon\|\bfv\|_{\calA},
$
which completes the proof.
\end{proof}

Note that $\bfz$ serves as a right-hand-side for the linear system $\calA\bfv=\bfz$. 
Thus, for the stability in computing $\bfv$ the $\bfz$-related quantities 
need to be measured in the $\|\cdot\|_{\calA^{-1}}$ norm.  

Now we are ready to quantify the accuracy of the numerical derivation of $\bfv_{0,\beta}=\calA^{-\beta}\bff$ under the 
proposed above computational strategy. Iteratively, we define $\bfv^{MG}_{0,1}$ to be the 
output of the applied optimal solver (the notation MG doesn't mean that only
multigrid solver can be used) to the linear system $\bfv_{0,1}=\calA^{-1}\bff$ and 
$$\bfv^{MG}_{0,j+1}\text{ -- numerical solution of } \calA\bfv=\bfv^{MG}_{0,j},\qquad j=1,\dots,\beta-1.$$ 

\begin{corollary}\label{cor:beta accuracy}
~~Let $\beta\ge1$ and $\bfv^{MG}_{0,\beta}$ be derived as described above. Assume that a numerical solver for the 
system 
$\calA\bfv=\bfz$, with arbitrary $\bfv,\bfz\in\mathbb R^N$ has computed $  \bfv^{MG}$
with guaranteed relative error  $\varepsilon$, i.e. 
$${\|\bfv^{MG}-\bfv\|_{\calA}}/{\|\bfv\|_{\calA}}\le \varepsilon.$$
Then
\begin{equation*}
{\|\bfv^{MG}_{0,\beta}-\bfv_{0,\beta}\|_{\calA}}/{\|\bfv_{0,\beta}\|_{\calA}}\le a_\beta\varepsilon
\quad \mbox{with} \quad a_{\beta+1}=1+(1+\varepsilon)\kkk(\calA)a_\beta,\quad a_1=1.
\end{equation*}
Consequently, $a_\beta=\mathcal{O}(\kkk(\calA)^{\beta-1})$.
\end{corollary}
\begin{proof}
The proof is by induction. For $\beta=1$, we have that $\bfv^{MG}_{0,1}$ is the solver output for 
$\bfv_{0,1}=\calA^{-1}\bff$. Thus, $a_1=1$ follows directly from the assumption on the solver accuracy. Now, let 
$$
{\|\bfv^{MG}_{0,\beta}-\bfv_{0,\beta}\|_{\calA}}/{\|\bfv_{0,\beta}\|_{\calA}}\le a_\beta\varepsilon
$$
holds true for $\beta$. Denote by $\bar\bfv_{0,\beta+1}:=\calA^{-1}\bfv^{MG}_{0,\beta}$. Again, due to the solver 
accuracy, we have
$$
{\|\bfv^{MG}_{0,\beta+1}-\bar\bfv_{0,\beta+1}\|_{\calA}}/{\|\bar\bfv_{0,\beta+1}\|_{\calA}}\le \varepsilon.
$$
Applying the obvious inequalities  $\Lambda^2_1 \langle\calA^{-1}\bfv,\bfv\rangle \le \langle\calA \bfv,\bfv\rangle \le 
\Lambda^2_N \langle\calA^{-1}\bfv,\bfv\rangle $
we obtain
$$\frac{\|\bfv^{MG}_{0,\beta}-\bfv_{0,\beta}\|_{\calA^{-1}}}{\|\bfv_{0,\beta}\|_{\calA^{-1}}}\le
\frac{\Lambda^{-1}_1\|\bfv^{MG}_{0,\beta}-\bfv_{0,\beta}\|_{\calA}}{\Lambda^{-1}_N\|\bfv_{0,\beta}\|_{\calA}}\le
\kkk(\calA)\frac{\|\bfv^{MG}_{0,\beta}-\bfv_{0,\beta}\|_{\calA}}{\|\bfv_{0,\beta}\|_{\calA}}\le 
\kkk(\calA)a_\beta\varepsilon.$$
The result follows from Lemma~\ref{thm:2-step PCG accuracy} which we apply with $\mu=\kkk({\calA})a_\beta$ and 
$\nu=1$.
\end{proof}

The coefficient $c_{0,\beta}=\pm E_\alpha(m,k;\beta)$, due to \eqref{eq:c0,beta}. 
Therefore the numerical accuracy for the computation of the term 
$c_{0,\beta}\bfv_{0,\beta}$ in $\bfu_r$ (see (\ref{eq:fract.decomp.ur})) depends on the product 
$E_\alpha(m,k;\beta)\mathrm 
k(\calA)^{\beta-1}$. 
For $\beta >1$ this product contains two factors that behave differently when $\beta$ grows:  
the first decreases (see, Table \ref{ta:error}) while the  second increases.
As a result, we conclude that computing with $\beta =1$ is a reasonable practical choice.
\subsection{Multi-step BURA approximation}\label{sec:multiBURA}
From Table~\ref{ta:error} we observe that when $\alpha$ increases, so does the error $E_\alpha(k,k;\beta)$. In 
particular, for $\beta=1$, the quantities $E_{0.25}(k,k;1)$, $E_{0.50}(k,k;1)$ and $E_{0.75}(k,k;1)$ are 
all of different order. This is due to the steeper slopes in a neighborhood of zero for the function 
$t^{1-\alpha}$, which results in higher and more frequent oscillations of the residual 
$r^1_\alpha(t)-t^{1-\alpha}$ there. Apart from such theoretical drawbacks, there are also additional numerical 
difficulties with the convergence of Algorithm~\ref{alg1} as the set of extreme points 
$\{\eta_i\}_1^{2k+2}$ for the residual cluster around zero (see \cite[Theorem 4]{saff1992asymptotic}). Indeed, higher 
numerical precision is needed for the correct separation of the extreme points, as well as more internal and 
external iterations are executed for solving the ill-conditioned linear systems in Step 3(ii) and for reaching the 
stopping criterion of the Algorithm, respectively. As an alternative approach, we study the possibility to replace the 
action of $r^1_\alpha$ by the joint action of several $r^1_{\alpha_i}$ rational functions, where each $\alpha_i$ is 
smaller than the original $\alpha$, thus $r^1_{\alpha_i}$ is cheaper to be generated and its approximation error 
$E_{\alpha_i}(k,k;1)$ is smaller. 

Our idea is to apply a multi-step procedure, based on the identity
$$
\calA^{-\alpha}\bff=\calA^{-\alpha_n}\circ\calA^{-\alpha_{n-1}}\circ\dots\circ\calA^{-\alpha_1}\bff,
\qquad\sum_{i=1}^n \alpha_i=\alpha.
$$ 
First, we approximate $\calA^{-\alpha_1}\bff$ by $\bfu_1:=r^1_{\alpha_1}(\calA)\calA^{-1}\bff$. Then we approximate 
$\calA^{-\alpha_2}\circ\calA^{-\alpha_1}\bff$ by $\bfu_2:=r^1_{\alpha_2}(\calA)\calA^{-1}\bfu_1$ and so on. Finally, we 
approximate $\bfu=\calA^{-\alpha}\bff$ by $\bfu_n=r^1_{\alpha_n}(\calA)\calA^{-1}\bfu_{n-1}$. Following \eqref{errorU} 
and setting $\gamma=1$ we are interested in the theoretical and numerical behavior of the error ratio 
$\|\bfu_n-\bfu\|_{\calA}/\|\bff\|_{\calA^{-1}}$.

The theoretical error analysis is based on Lemma~\ref{l:BURAerror}. Denote by $\varepsilon_i(t)$ the 
residual of $r^1_{\alpha_i}(t)$ with respect to $t^{1-\alpha_i}$. Then, for each $i=1,\dots,n$ we have
\begin{equation}\label{eq:residual}
 r^1_{\alpha_i}(t)=t^{1-\alpha_i}+\varepsilon_i(t),\qquad |\varepsilon_i(t)|\le E_{\alpha_i}(k,k;1)\quad \forall 
t\in[0,1].
\end{equation}
The multi-step approximation $\bfu_n$ can be rewritten in the form
\begin{equation*} 
\bfu_n=\prod_{i=1}^{n}r^1_{\alpha_i}(\calA)\,\calA^{-n}\bff=\left(\calA^{n-1}\prod_{i=1}^{n}r^1_{\alpha_i}(\calA)\right)
\calA^ {-1} \bff.
\end{equation*}
Therefore the proof of Lemma~\ref{l:BURAerror} implies that we need to estimate the approximation error
\begin{equation}\label{eq:multi-step error}
E_{\alpha_1\dots\alpha_n}(k,k;1):=\max_{t\in\{\Lambda_i\}_1^N}\left|\frac{r^1_{\alpha_1}(t)r^1_{\alpha_2}(t)\dots 
r^1_{\alpha_n}(t)}{t^{n-1}}-t^{1-\alpha} \right|.
\end{equation}
Consider $n=2$. Using \eqref{eq:residual} we obtain
\begin{equation*}
\frac{r^1_{\alpha_1}(t)r^1_{\alpha_2}(t)}{t}=t^{1-\alpha}+t^{-\alpha_1}\varepsilon_2(t)+t^{-\alpha_2}\varepsilon_1(t)+ 
t^{-1}\varepsilon_1(t)\varepsilon_2(t).
\end{equation*}
Denote by $E_{\alpha_i}=E_{\alpha_i}(k,k;1)$, $i=1,2$.
Since the spectrum of $\calA$ is 
normalized and $\Lambda_N  \le 1$, 
so that 
$ t^{-\alpha_i} \le \Lambda_1^{-\alpha_i} \approx k(\calA)^{\alpha_i}$
and from \eqref{eq:multi-step error} we conclude
\begin{equation}\label{eq:general 2-step error}
E_{\alpha_1\alpha_2}(k,k;1)\le 
E_{\alpha_2}\mathrm{k}(\calA)^{\alpha_1} +
E_{\alpha_1}\mathrm{k}(\calA)^{\alpha_2} +
E_{\alpha_1}E_{\alpha_2}\mathrm{k}(\calA).
\end{equation}
Comparing with the error estimate $E_\alpha(k,k;1)$ we observe that unlike the direct approach, the two-step 
approximation error depends on the condition number of $\calA$, thus is not dimension-invariant in general. 
Therefore, the benefits from the proposed multi-step procedure for computing $\calA^{-\alpha}\bff$ 
are limited in theory since the overall error is magnified by the condition number of $\calA$.  

On the other hand, the multi-step BURA method possesses several interesting properties that are worth investigating 
further. First of all, it provides good approximation on the high part of the spectrum of $\calA$. From 
Lemma~\ref{BURA_caracter} we know that 
$$\left(r^1_{\alpha_i}(t)-t^{1-\alpha_i}\right)\big|_{t=1}=-E_{\alpha_i},\qquad i=1,2,$$
so for $\bff=\Psi_N$ we get
$${\|\bfu_2-\calA^{-\alpha}\Psi_N\|_{\calA}}/{\|\Psi_N\|_{\calA^{-1}}}\approx 
E_{\alpha_1}+E_{\alpha_2}-E_{\alpha_1}E_{\alpha_2},$$
since $\Lambda_N\approx1$. This is much better than $E_\alpha(k,k;1)$. Thus, as $N\to\infty$, the approach is 
beneficial when $\bff\in 
span\{\Psi_{N-\ell},\dots,\Psi_N\}$ for some $\ell\ll N$. Second of all, especially if $\alpha_1\neq\alpha_2$, 
$E_{\alpha_1\alpha_2}$ might remain significantly smaller than the right-hand-side of \eqref{eq:general 2-step error}. 
Indeed, the extreme points of $r^1_{\alpha_1}$ and $r^1_{\alpha_2}$ are with high probability disjoint sets, therefore 
$|\varepsilon_1(t)|$ and $|\varepsilon_2(t)|$ cannot simultaneously attend their maximums, meaning that the factor 
$E_{\alpha_1}E_{\alpha_2}$ in front of $\mathrm{k}(\calA)$ is an overestimate.
\section{Numerical tests}
\label{sec:NumTest}

A comparative analysis of the numerical accuracy of the proposed 
solvers and the related theoretical estimates are 
presented in this section. The first group of tests concerns normalized 
matrices obtained from a three-point 
approximation of the Poisson equation in one space dimension. 
For this setting we are able to 
directly compute the 
exact solution $\bfu=\calA^{-\alpha}\bff$, as well as the approximate solution 
$\bfu_r= r^\beta_\alpha(\calA) \calA^{-\beta} \bff$, 
thus no additional numerical errors are accumulated in the process. The 
first experimental set is devoted to the numerical validation of Lemma~\ref{l:BURAerror}. 
The second one studies possible improvements in the accuracy of the approximation 
$\bfu_r$ for larger $\alpha$ when a multi-step 
approximation process that involves smaller $\alpha$'s is applied. 
A third experiment deals with a 2D fractional 
Laplacian operator and illustrates the rescaling effect in \eqref{errorU}.
Finally, we confirm the accuracy 
analysis in Section~\ref{42-accur} by running 3D numerical experiments, 
where $\bfu$ is unknown, while $\bfu_r$ is 
computed by a numerical solver that uses algebraic multigrid as a 
preconditioner in the conjugate gradient method. 
\begin{table}[h!]
\begin{center}
\begin{tabular}{||r|r|r|r|r|r|r||}
\hline \hline
$\alpha$&$E_\alpha(5,5;1)$&$E_\alpha(5,4;2)$&$E_\alpha(5,3;3)$&$E_\alpha(7,7;1)$&$E_\alpha(7,6;2)$&$E_\alpha(7,5;3)$\\ 
\hline
   0.75 & 2.7348E-3& 3.8415E-6& 4.6657E-7& 7.8650E-4& 2.0108E-7& 6.6194E-9\\
   0.50 & 2.6896E-4& 2.0349E-6& 4.0421E-7& 4.6037E-5& 7.8577E-8& 4.3899E-9\\
   0.25 & 2.8676E-5& 6.2333E-7& 1.8958E-7& 3.2566E-6& 1.8043E-8& 1.5792E-9\\
   0.10 & 4.9432E-6& 1.7490E-7& 6.7114E-8& 4.5139E-7& 4.2824E-9&  4.7675E-10\\
\hline \hline
\end{tabular}\\[1ex]
\caption{
Errors $ E_\alpha(m,m+1-\beta;\beta)$ of BURA $ P^\ast_m(t)/Q^\ast_{m+1-\beta}(t) $
of $ t^{\beta -\alpha}$ on $[0,1]$ for $m=5,7$.
}\label{tab: errors}
\end{center}
\end{table}
In all the experiments we take $m+1=k+\beta$, thus we solve $m+1$ linear systems in 
order to determine $\bfu_r$ (see \eqref{eq:Rgeneral}). We consider $m=\{5,7\}$, $\beta=\{1,2\}$, and 
$\alpha=\{0.25,0.5,0.75\}$. For each of the corresponding BURA functions all the zeros $d_i$ of the denominator are 
real and of multiplicity one, so only systems of the type ${(\calA - d_i I)^{-1}} \bff$ appear. 
The approximation 
errors are summarized in Table~\ref{tab: errors}. 
In the discussion below the Euclidean norm of a vector in $\R^N$ is denoted as $\ell^2$-norm.

\subsection{Numerical validation of Lemma~\ref{l:BURAerror}}\label{sec41}
We consider the $N\times N$ stiffness matrix $\mathcal A$, corresponding to 
a three-point finite difference approximation (or FE approximation with linear elements)
of the operator  $ \calL u=-u^{\prime \prime}$   with zero Dirichet boundary conditions
on a uniform partitioning of  $(0,1)$   with  {mesh-size} 
$h=1/(N+1)$. The tridiagonal matrix is normalized so that its spectrum lies 
inside 
$[0,1]$  and has enrties  $1/2$ on 
the main diagonal and $-1/4$  {on the upper and lower co-diagonals.}

The eigenvalues and eigenvectors of $\mathcal A$ are
$$\Lambda_i=\sin^2\left(\frac{i\pi}{2(N+1)}\right),\qquad \bPsi_i=\left\{\sin\frac{ik\pi}{N+1}\right\}_{k=1}^N,\qquad 
i=1,\dots,N.$$
Note that all the eigenvectors $\bPsi_i$ are of the same length, due to 
\begin{align*}
\|\bPsi_i\|_2^2&=\langle\bPsi_i,\bPsi_i\rangle=\sum_{k=1}^N\sin^2\frac{ik\pi}{N+1}=\frac{N}{2}-\frac12\sum_{k=1}
^N\cos\frac{2ik\pi}{N+1}=\frac{N+1}{2}
\end{align*}
so we do not normalize them.

\begin{figure}[thp]
\begin{center}
\includegraphics[width=0.48\textwidth]{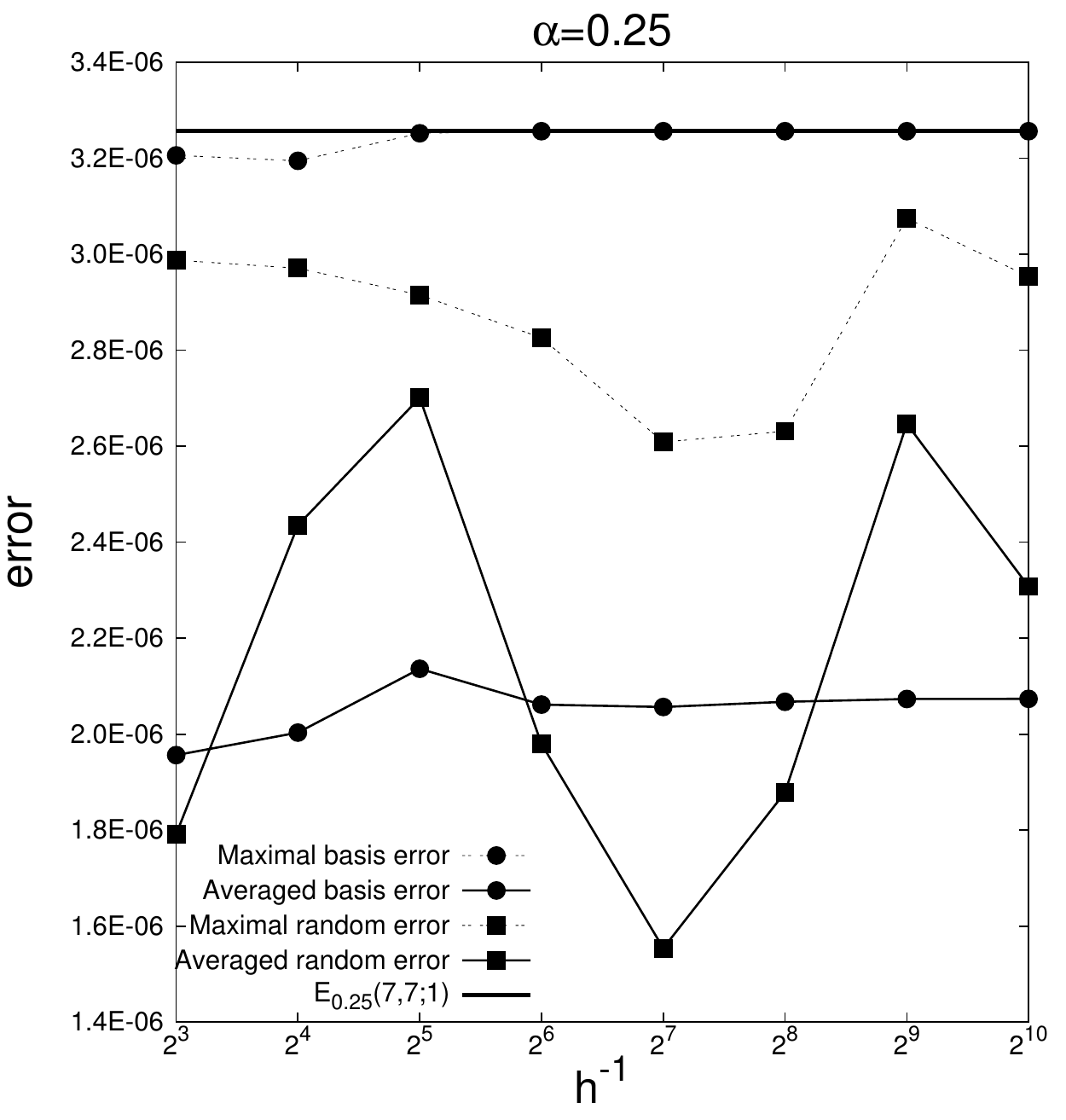} 
\includegraphics[width=0.48\textwidth]{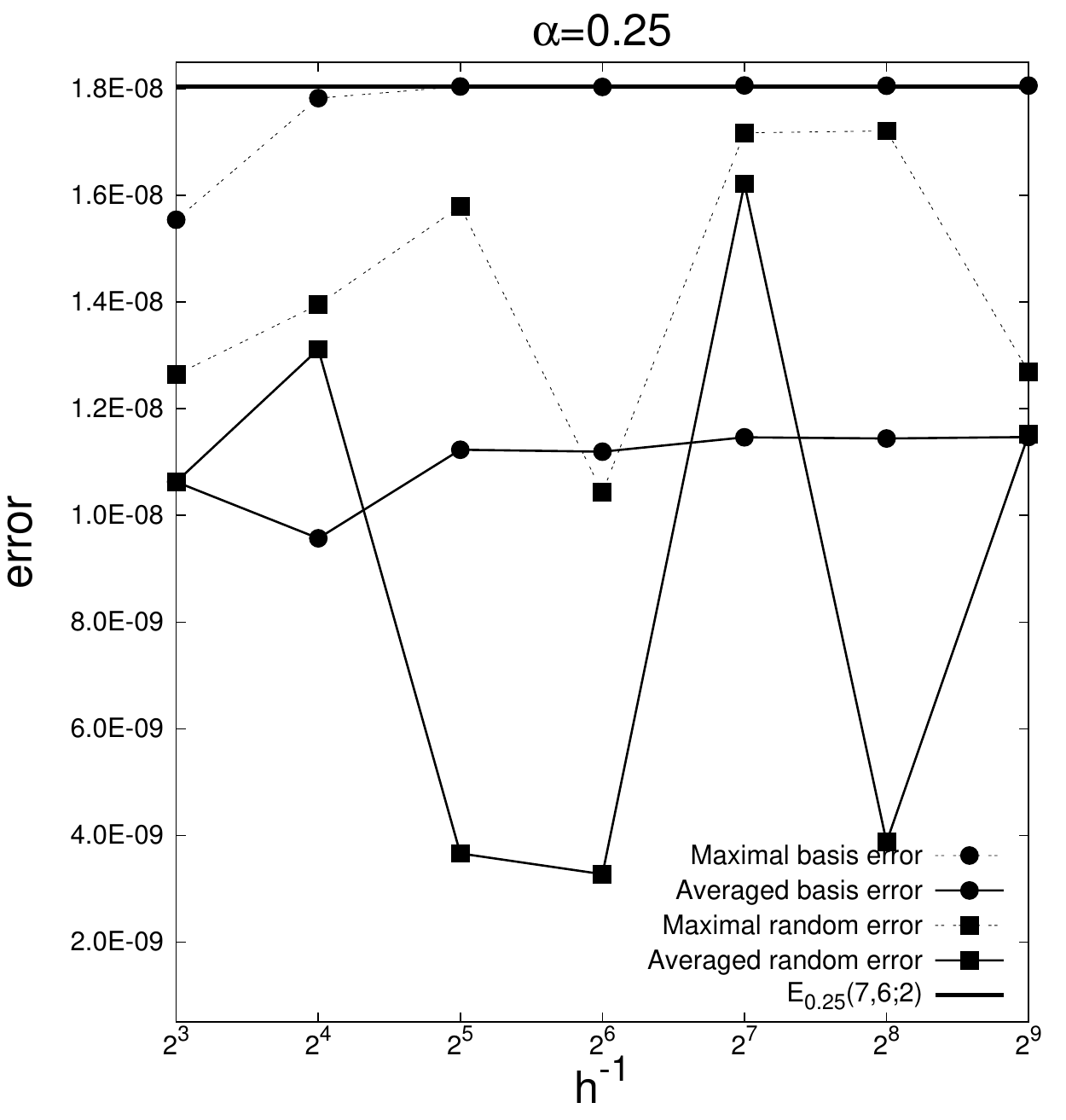}\\ 
\includegraphics[width=0.48\textwidth]{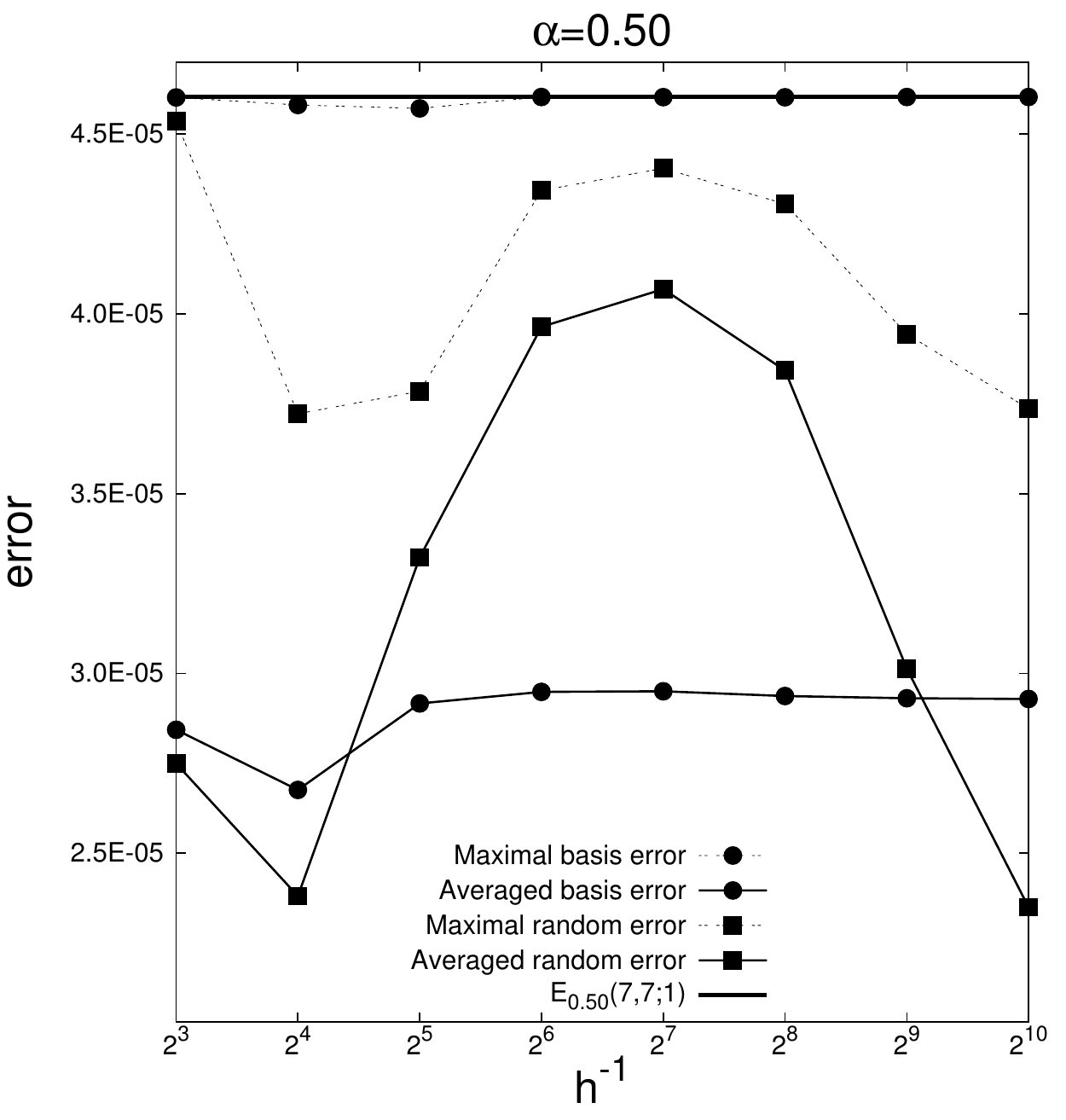} 
\includegraphics[width=0.48\textwidth]{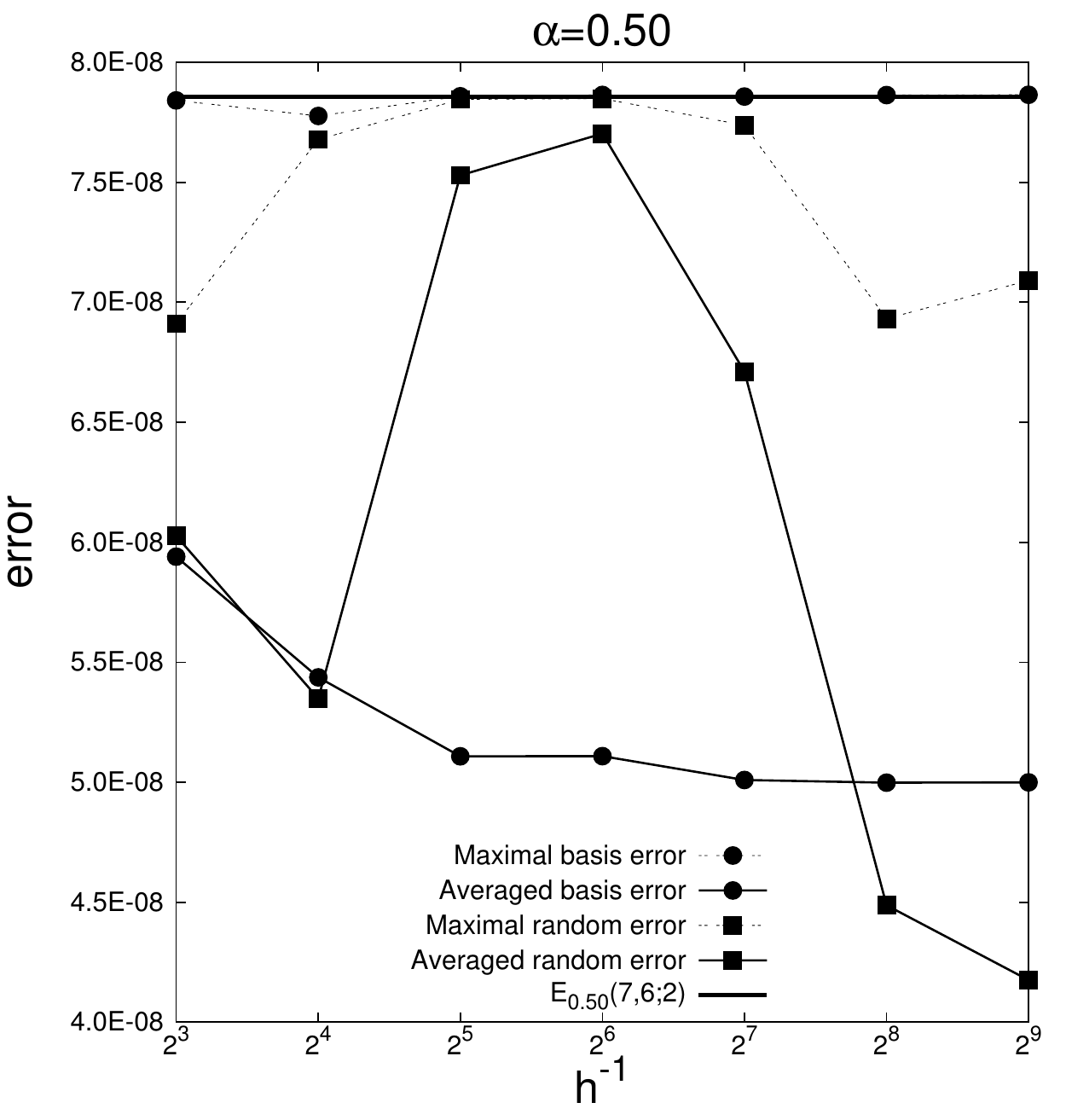} \\ 
\includegraphics[width=0.48\textwidth]{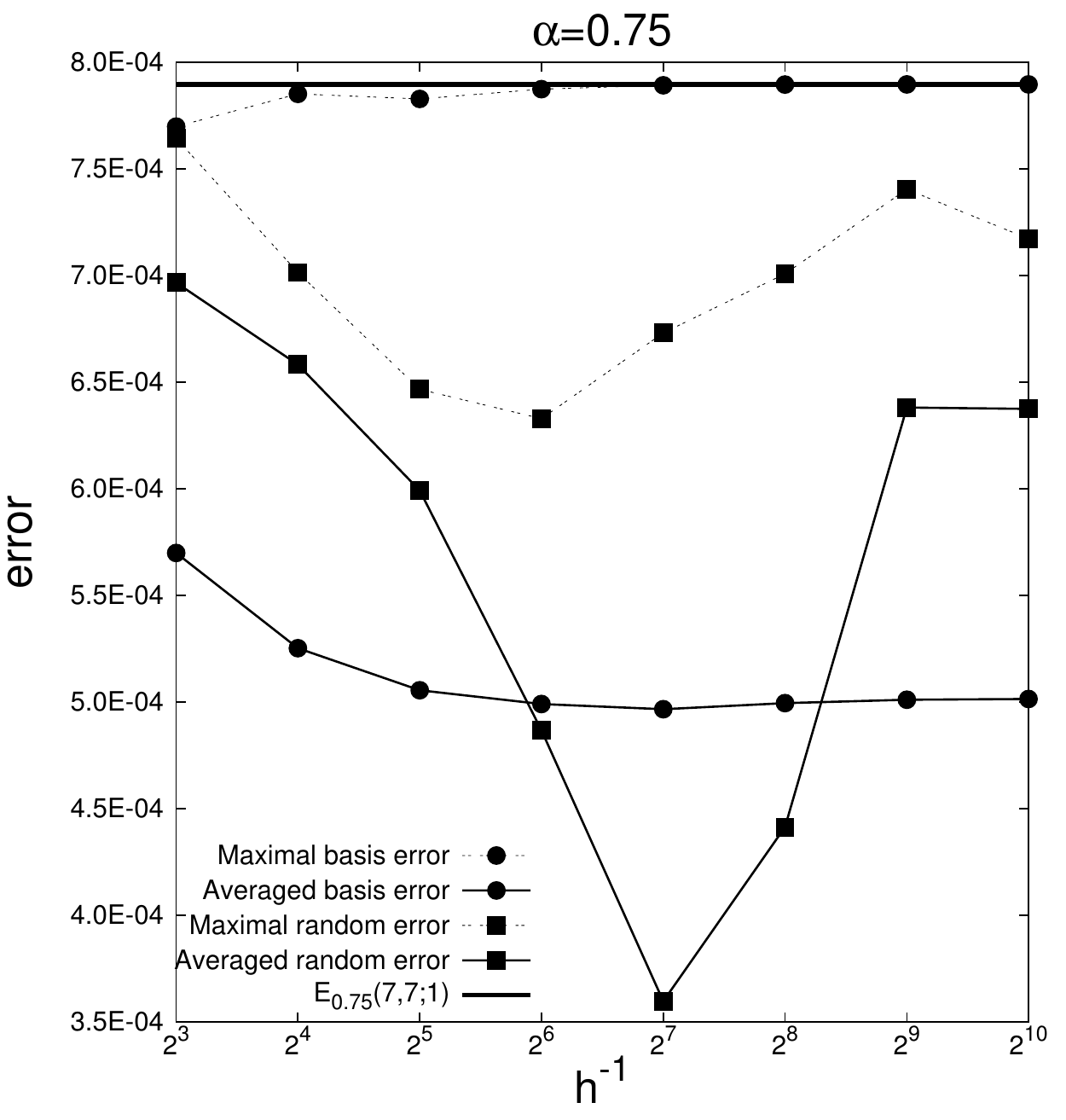} 
\includegraphics[width=0.48\textwidth]{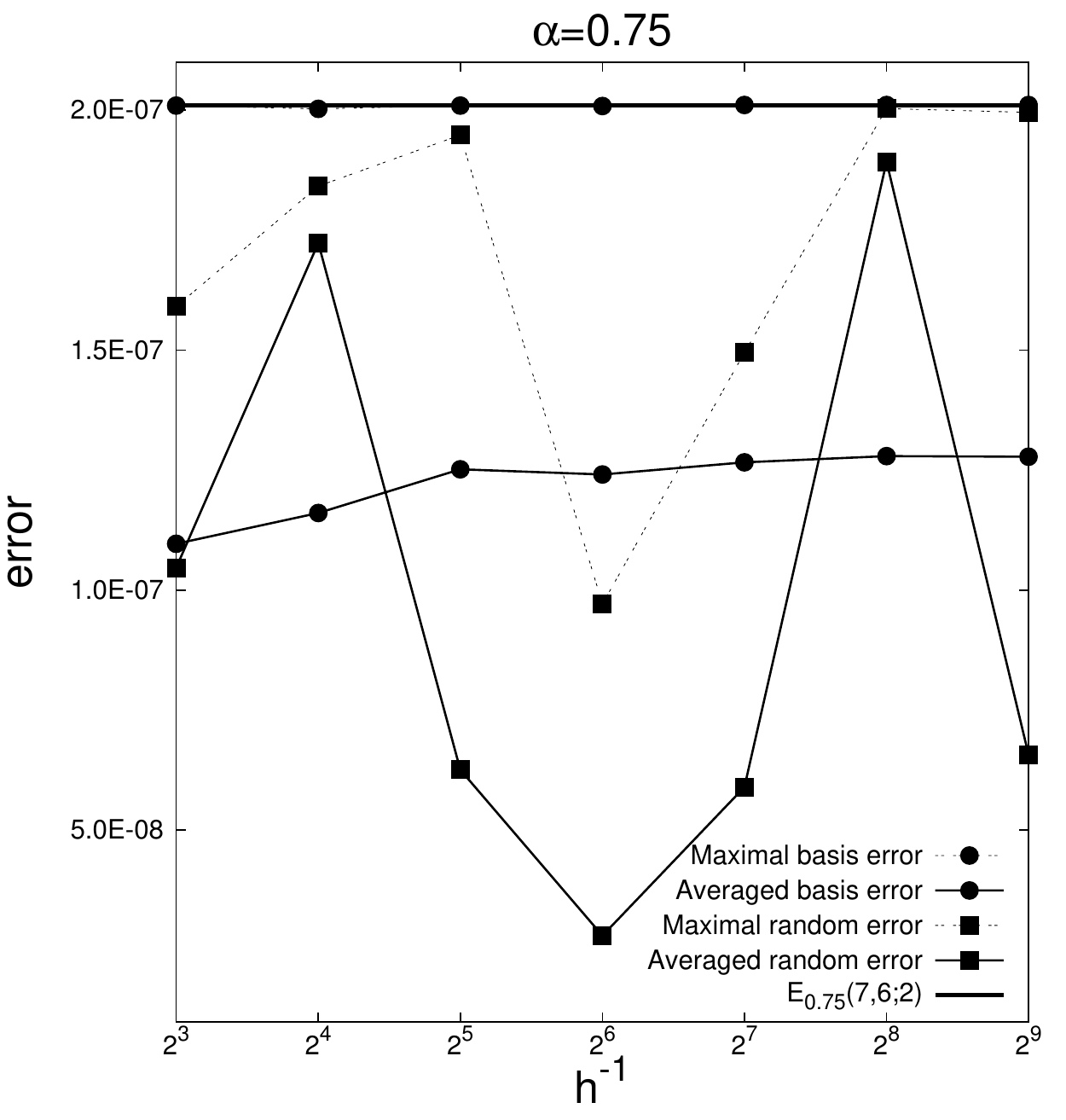} 
\end{center}
\caption{1D numerical validation of Lemma~\ref{l:BURAerror} for $m=7$. Left: $(7,7;1)$. Right: $(7,6;2)$. 
Error is measured as indicated in \eqref{bound}}\label{fig:error_m7}.
\end{figure}

Numerical results for $m=7$, $\beta=1,2$ are summarized in 
Fig.~\ref{fig:error_m7}. As suggested by \eqref{bound}, we measure 
the relative  error
$ 
 \| \bfu_r- \bfu\|_{\calA} / \|\bff\|_{\calA^{-1}}
$ 
for $\beta=1$ 
and the relative  error
$ 
\| \bfu_r- \bfu\|_{\calA} / \|\bff\|_{\calA^{-3}}
$ 
for $\beta=2$.
We use as input the coefficient vector of $\bff$ with respect to the basis 
$\{\bPsi_i\}_{i=1}^N$, so the derivation of the exact solution $\bfu$ as well as the computation 
of the norms $\|\bff\|_{\calA^{-1}}$, respectively $\|\bff\|_{\calA^{-3}}$, is straightforward. In order to 
compute the approximated solution $\bfu_r$, we first generate the coefficient vector of $\bff$ with respect to the 
standard basis $\{\delta_{ik}\}_{i,k=1}^N$ and then solve exactly the 
corresponding $m+\beta$
tridiagonal linear systems that originate from the fractional decomposition of 
$t^{-\beta}r^\beta_\alpha(t)$. Randomness is with respect to the entries of the input coefficient vector.

We study four different error quantities: the maximal error over the eigenvectors $\{\bPsi_i\}_{i=1}^N$, which 
coincides with the true estimate of the approximation error; the maximal error over a randomized set of 1000 $\bff$'s, 
which is the numerical approach for estimating the former; the averaged error over the eigenvectors; and the 
averaged error over the random right-hand-side set. The last two quantities provide information about the general 
behavior of the error and its expectation value. The main observation is that the errors, related to the 
eigenvectors set behave quite stably 
with respect to the size  of $\calA$, unlike the errors related to 
random vector input. Such ``dimension-invariance'' of the results from the first class is due to the almost uniform 
distribution of the eigenvalues $\{\Lambda_i\}_{i=1}^N$ of $\calA$ along the interval $[0,1]$ and that for every 
$\beta$-BURA function the endpoints 0 and 1 are extreme points for the residual 
$r^\beta_\alpha(t)-t^{\beta-\alpha}$ (i.e., $\{0,1\}\subset\{\eta_i\}_1^{m+k+2}$). As $N\to\infty$, 
we have $\Lambda_1\to0$, $\Lambda_N\to1$ and we observe that all the maximal norm ratio errors for 
eigenvectors input 
tend to the corresponding univariate error $E_\alpha(m,k;\beta)$. 
The $\beta$-BURA functions oscillate mainly close to $0$ and are stable close to 1, making rapid convergence
$|r^\beta_\alpha(\Lambda_N)-\Lambda_N^{\beta-\alpha}|\to E_\alpha(m,k;\beta)$. Therefore, the placement 
of the remaining spectrum $\{\Lambda_i\}_{i=1}^{N-1}$ of $\calA$ with respect to $\{\eta_i\}$ is not significant for 
this quantity. On the other hand, it is practically impossible to generate (a rescaled version 
of) $\bPsi_N$ at random, so the randomized errors heavily depend on the placement of the whole spectrum of $\calA$ with 
respect to the extreme points of the $\beta$-BURA function.  As a result, both maximal and averaged 
random errors can be anywhere in the interval between the minimal and maximal error of the eigenvectors. 
We generated various random  sets of different size 
(e.g., $10^3,10^4$) and checked that for a fixed $N$ the two errors behave stably with 
respect to the choice of randomness. This allows us to conclude that the ``dimension-instability'' phenomenon is indeed 
fully due to the specifics of the spatial distribution of the spectrum of $\calA$. 

 
\subsection{Multi-step 1-BURA approximation for $\alpha=\{0.5,0.75\}$}\label{sec42}

\begin{figure}[htp]
\begin{center}
\includegraphics[width=0.8\textwidth]{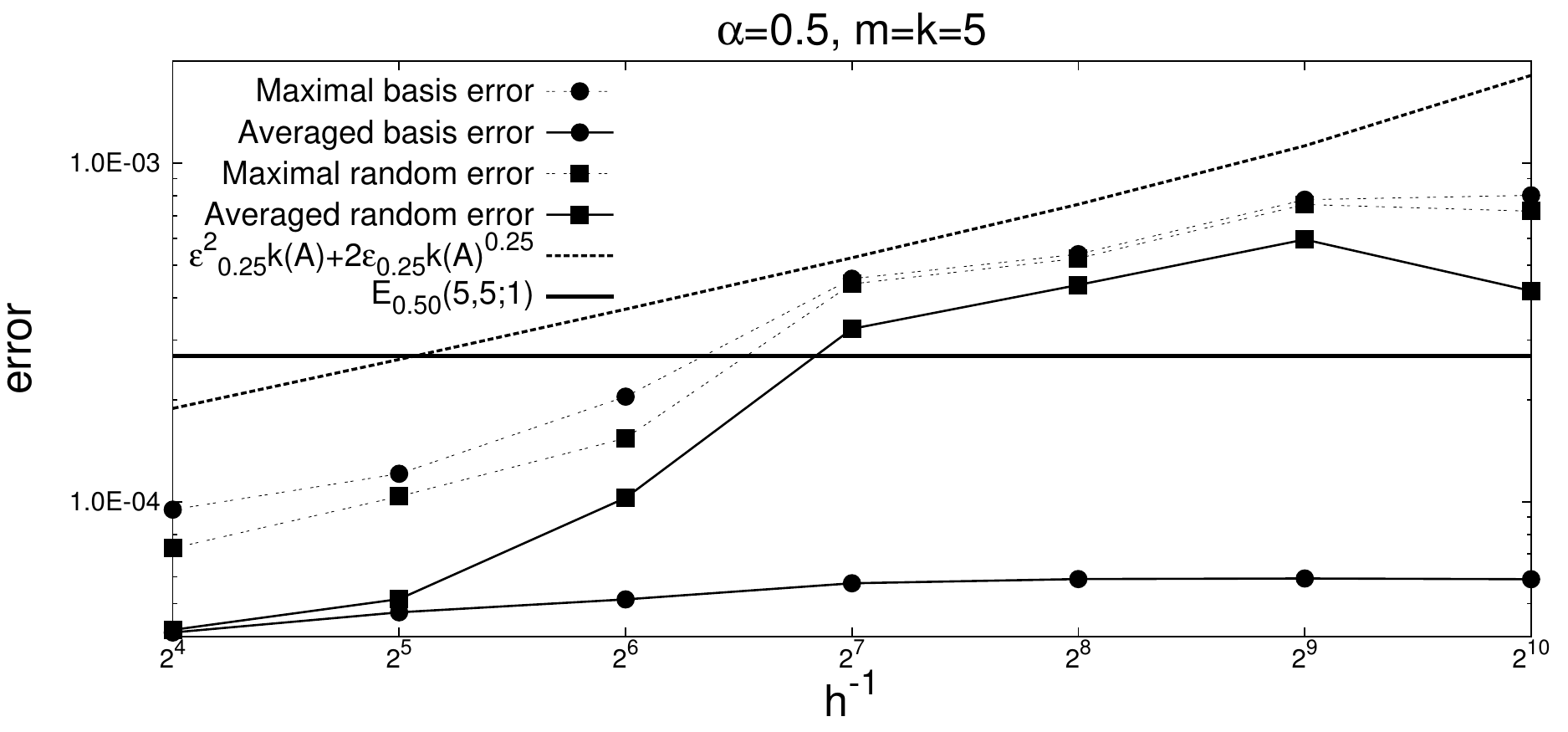}\\ 
\includegraphics[width=0.8\textwidth]{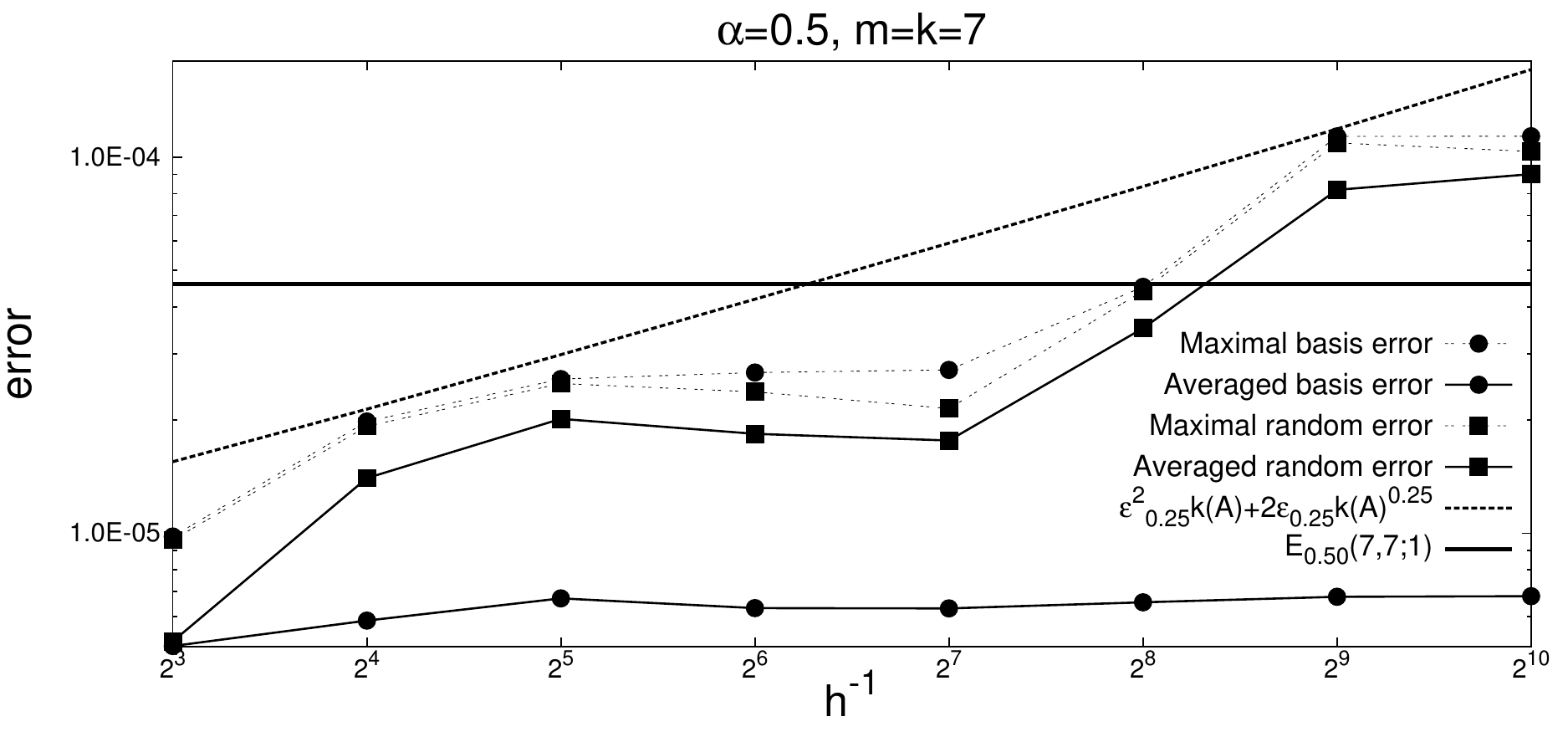} \\ 
\includegraphics[width=0.8\textwidth]{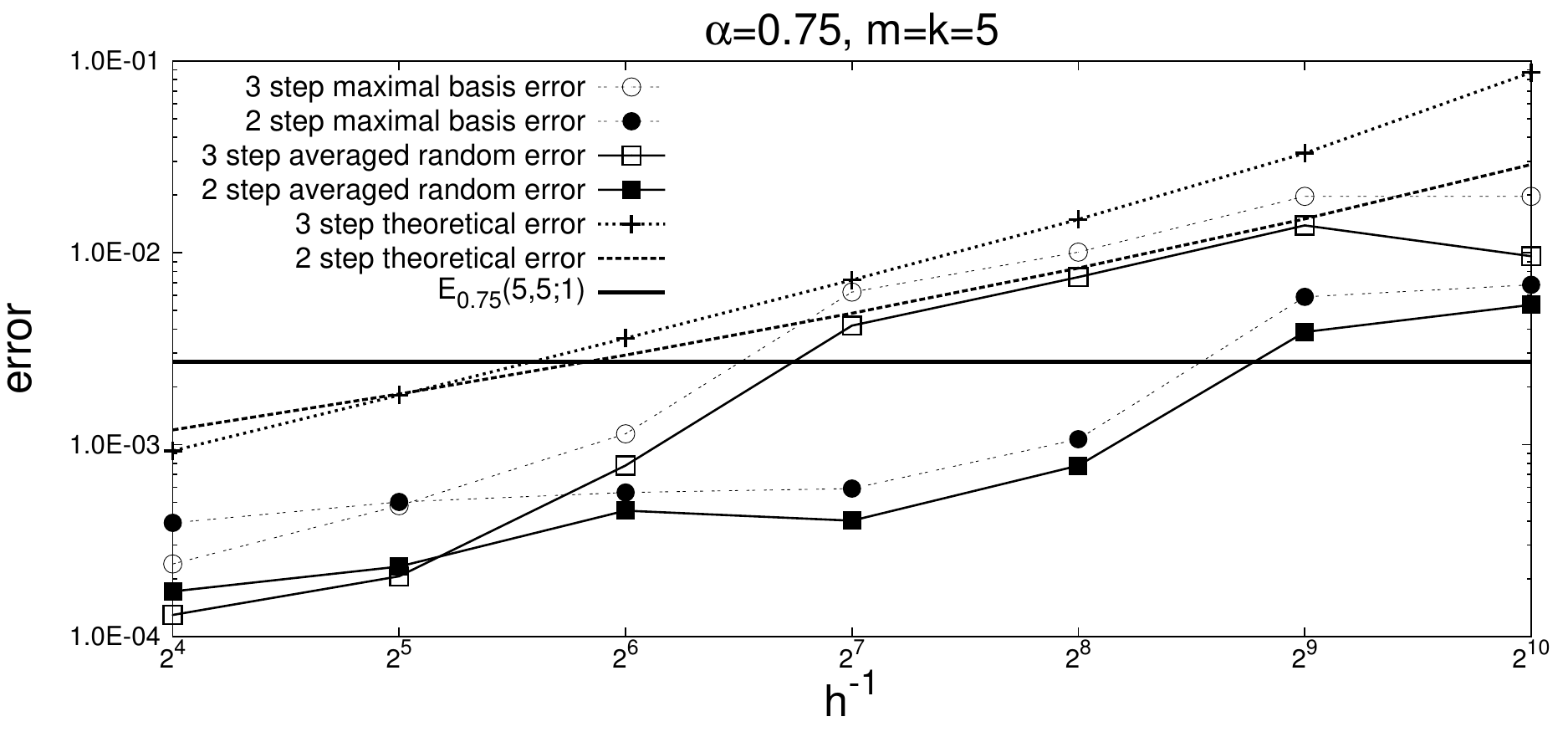} \\ 
\includegraphics[width=0.8\textwidth]{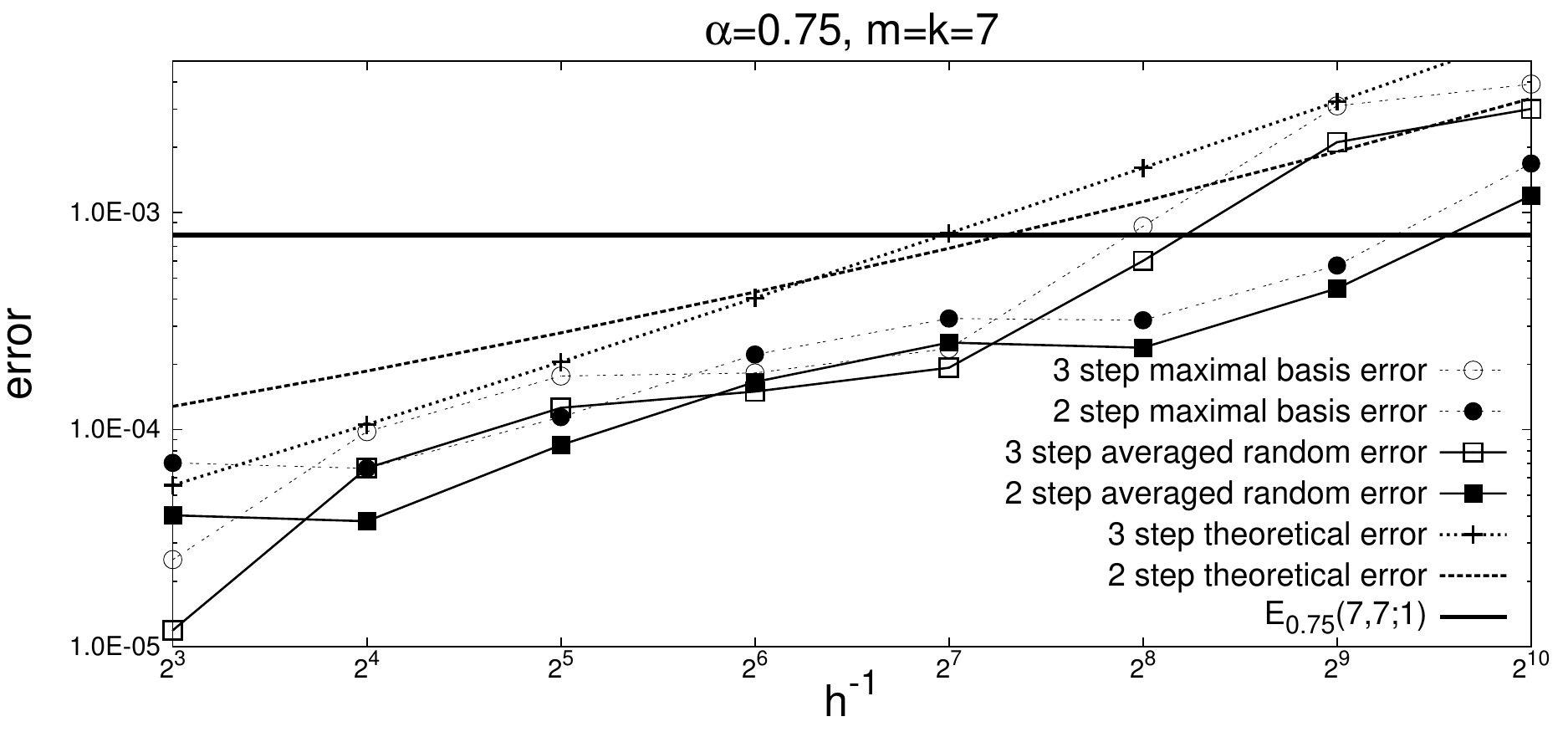} 
\end{center}
\vspace*{-0.5cm}
\caption{1D numerical error analysis for the multi-step case. 
The relative errors $\|\bfu_r-\bfu\|_{\calA}/\|\bff\|_{\calA^{-1}}$ are plotted.}
\label{fig:multi_step_error}
\end{figure}

The second series of numerical experiments are devoted to the multi-step generalization of the method.
The presented
numerical experiments for $\calA=tridiag (-0.25,0.5,-0.25)$ as in Section~\ref{sec41}, $k=\{5,7\}$ and 
$\alpha=\{0.5,0.75\}$ confirm the theoretical analysis in Section~\ref{sec:multiBURA}.  The related results 
are summarized in Fig.~\ref{fig:multi_step_error}. When 
$\alpha=0.5$, we study the two step procedure based on $\alpha_1=\alpha_2=0.25$. When $\alpha=0.75$, we investigate 
both the two step procedure with $(\alpha_1,\alpha_2)=(0.5,0.25)$ and the three step procedure based on 
$\alpha_1=\alpha_2=\alpha_3=0.25$. In the latter case, it is straightforward to derive the three-step analogous formula 
to \eqref{eq:general 2-step error}, which in the particular setup implies
\begin{equation}\label{eq:3-step error}
E_{0.25\,0.25\,0.25}(k,k;1)=E_{0.25}^3\mathrm{k}^2(\calA)+3E_{0.25}^2\mathrm{k}^{5/4}(\calA) 
+3E_{0.25}\mathrm{k}^{1/2}(\calA).
\end{equation}
Again, as in \eqref{eq:general 2-step error}, we use the short notation $E_{0.25}$ for $E_{0.25}(k,k;1)$.

We set $h^{-1}=N+1=2^i$, $i\in\{3,4,\dots,10\}$ and plot various errors. Namely, for $\alpha=0.5$ those are: the 
theoretically estimated two-step error \eqref{eq:general 2-step error}, the true two-step error estimate that coincides 
with the maximal error over the eigenvectors $\{\bPsi_i\}_{i=1}^N$, the averaged two-step error over the eigenvectors, 
the maximal and averaged two-step errors over a thousand randomly generated vectors, and the one-step error 
$E_{0.50}(k,k;1)$. For $\alpha=0.75$ we plot the theoretically estimated two and three step errors \eqref{eq:general 
2-step error}-\eqref{eq:3-step error}, the true two and three step error estimates, the corresponding averaged errors 
of random right-hand-side data, and the one-step error $E_{0.75}(k,k;1)$. 

From the plots in Fig.~\ref{fig:multi_step_error} we observe that when the $\alpha_i$'s coincide the true 
multi-step error estimate reaches the theoretical bound for particular sizes of $\calA$. This happens when $\Lambda_1$ 
hits an extreme point of 
$r^1_{0.25}$, i.e., $|r^1_{0.25}(\Lambda_1)-\Lambda^{3/4}_1|\approx E_{0.25}$, as 
for $h=2^{-7}$, $k=5$ and $h=\{2^{-4},2^{-9}\}$, $k=7$. When $\alpha_1\neq\alpha_2$ 
we confirm that the theoretical 
bound $E_{0.25\;0.50}$ is an overestimation of the true maximal error, since the sets of internal extreme points for 
$r^1_{0.25}$ and $r^1_{0.50}$ are disjoint and it is not possible for $\Lambda_1$ to simultaneously hit 
both. Note that $\Lambda_1$ tends to zero as $N\to\infty$ but it never reaches zero  and the heavy oscillations of the 
residual in this area do not allow the maximal basis error to reach $E_{0.25\;0.50}$ even for $\calA$ of size 
$1023\times 1023$.  

Unlike the first experimental setup, here we witness similar behavior among the theoretical error, the maximal 
random error, and the averaged random error, meaning that the measured quantity is stable and does not heavily 
depend on $\bff$. This is due to the specifics of the multi-step procedure and the existence of pole at zero for the 
product rational approximation. Hence, whenever $\langle \bff,\bPsi_1\rangle\neq0$ this component dominates the overall 
error value. Because of that, the averaged basis error do not provide 
reliable information about the error in the general (worst) case,
since the eigenvectors of $\calA$ are mutually orthogonal. This error remains  substantially
below the one-step error $E_{0.50}$ in all conducted experiments. 

As $k$ increases, the two-step error remains better than the one-step error for a larger set of matrix sizes. For 
$\alpha=0.5$ and $k=5$ the two-step error overpasses $E_{0.50}$ for $h=2^{-7}$, while for $k=7$ this happens for 
$h=2^{-9}$. For $\alpha=0.75$ the benefits of the two-step process are bigger, as the two-step error remains in 
vicinity of $E_{0.75}$ even for $h=2^{-10}$. However, as in the theoretical analysis, we clearly see the 
dimension-dependence of the multi-step errors. Nevertheless, with respect to controlling the ratio
$\|\bfu_r-\bfu\|_{\calA}/\|\bff\|_{\calA^{-1}}$ in the cases when $r^1_\alpha$ cannot be numerically computed, the 
proposed two-step procedure seems a better asymptotic choice than $r^2_\alpha$, since
$$\frac{\|r^2_\alpha(\calA)\calA^{-2}\bff-\calA^{-\alpha}\bff\|_{\calA}}{\|\bff\|_{\calA^{-1}}}\le 
\frac{\|r^2_\alpha(\calA)\calA^{-2}\bff-\calA^{-\alpha}\bff\|_{\calA}}{\|\bff\|_{\calA^{-3}}}\mathrm{k}(\calA)\le 
E_\alpha(k,k;2)\mathrm{k}(\calA)$$
and, comparing to \eqref{eq:general 2-step error}, we experimentally observe that 
$E_{\alpha_1}(k,k;1)E_{\alpha_2}(k,k;1)<E_\alpha(k,k;2)$ if $\alpha_1+\alpha_2=\alpha$.

\subsection{Comparison BURA and the method of Bonito and Pasciak, \cite{BP15}.}
In this Subsection we experimentally compare the numerical efficiency of the BURA solver with the one, 
developed in \cite{BP15} on a test example taken from their paper \cite{BP15}.  We consider the 
problem
\begin{equation}\label{eq:2D Laplace}
(-\Delta)^\alpha u =   f,\quad  u_{\partial\Omega}=0, \qquad\Omega=[0,1]\times [0,1]
\end{equation}
and its finite element approximation on a uniform rectangular grid with mesh-size $h=1/(N+1)$. This leads to a 
5-point stencil approximation $\calAt$ of $-\Delta$  of the form
$$ 
\calAt=h^{-2}tridiag\left(-I_N,\calAt_{i,i},-I_N\right),\quad \calAt_{i,i}=tridiag (-1,4,-1),\qquad \forall i=1,\dots,N.
$$
Then we have the algebraic problem \eqref{algebraic} with  $\calA=h^{2}\calAt/8$, where $\calA$ is 
an $N^2\times N^2$ SPD matrix with spectrum in the interval $(0,1]$ and $\bff $ is the vector of the 
values of $f(x,y)$ at the grid points scaled by $h^2/8$ (using lexicographical ordering).

For the right-hand-side $  f$ we use the checkerboard function on 
$\Omega\setminus\partial\Omega=(0,1)\times(0,1)$
\begin{equation}\label{eq:checkerboard}
  f(x,y)=\left\{\begin{array}{rl} 1,& \text{if } (x-0.5)(y-0.5)>0,\\ -1,& \text{otherwise}.\end{array}\right. 
\end{equation}

In \cite[Remark 3.1]{BP15} it is observed that in order to balance all the three exponential terms in their 
error estimate, the optimal quadrature approximate of $\calAt^{-\alpha}\bfftil$ is 
\begin{equation*}\label{eq:BonitoPasciak}
\bfu_{Q}=\frac{2k'\sin(\pi\alpha)}{\pi}\sum_{\ell=-m}^M 
e^{2(\alpha-1)\ell k'}\left(e^{-2\ell k'} \calIt +\calAt\right)^{-1} \bff, \quad m=\left\lceil\frac{\pi^2}{4\alpha k'^2}\right\rceil, \; M=\left\lceil\frac{\pi^2}{4(1-\alpha) k'^2}\right\rceil, 
\end{equation*}
where $k'>0$ is a free parameter. 
The number of linear systems to be solved in order to 
compute $\bfu_Q$ can be trivially estimated via
$$\#\text{ systems} = M+m+1\ge k_Q+1, \qquad k_Q:=\frac{\pi^2}{4\alpha(1-\alpha)k'^2}.$$
For the $(k,k)$ 1-BURA approximation $\bfu_r$  we need to solve $k+1$ linear systems. 
Note that in both approaches all systems correspond to positive diagonal shifts of $\calAt$, 
thus they possess similar computational complexity. Therefore, in order to perform 
comparison analysis on the numerical efficiency of the two solvers we need to take $k_Q\sim k$.

As a reference solution $\bfu_{\text{ref}}$ for \eqref{eq:2D Laplace} we consider the solution  
$\bfu_Q$ for $h=2^{-10}$ with $k'=1/3$, which guarantees $\mathrm{O}(10^{-7})$ error, 
see \cite[Table 3]{BP15}.

In the numerical experiments we use the following parameters: $h=2^{-10}\approx 10^{-3}$ and 
$k=\{9,8,7\}$ for $\alpha=\{0.25,0.5,0.75\}$, respectively. 
The corresponding 1-BURA-approximations of $\calAt^{-\alpha}\bfftil$ are 
illustrated on Fig.~\ref{fig:2D}. Furthermore, we restrict our analysis to integer $k_Q$. 
Note that both $k_Q$ and the positive shifts $e^{-2\ell k'}$ 
are continuous functions of $k'$, meaning that there is a whole interval of values $k'$ leading to the same number of 
systems to be solved for $\bfu_Q$ and each $k'$ gives rise to different shift parameters, thus different quadrature rule, 
respectively approximation error. For us it is not clear which choice of $k'$ will 
lead to the smallest $\|\bfu -\bfu_Q\|_2/\|  \bfftil \|_2$.

\begin{figure}
\begin{tabular}{ccc}
$\alpha=0.25$, $k=9$ & $\alpha=0.5$  $k=8$ & $\alpha=0.75$, $k=7$\\
\includegraphics[width=0.33\textwidth]{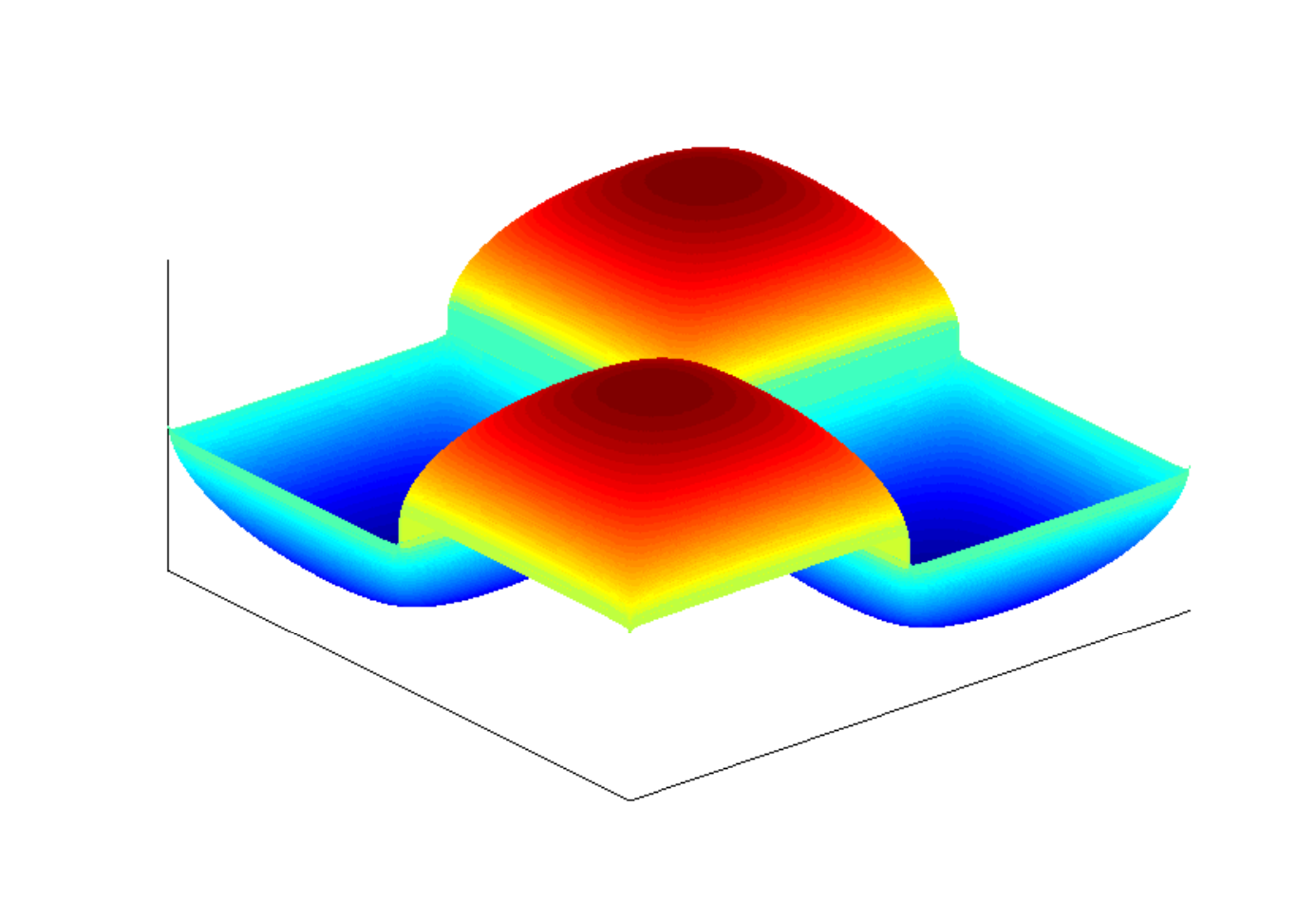} & 
\includegraphics[width=0.33\textwidth]{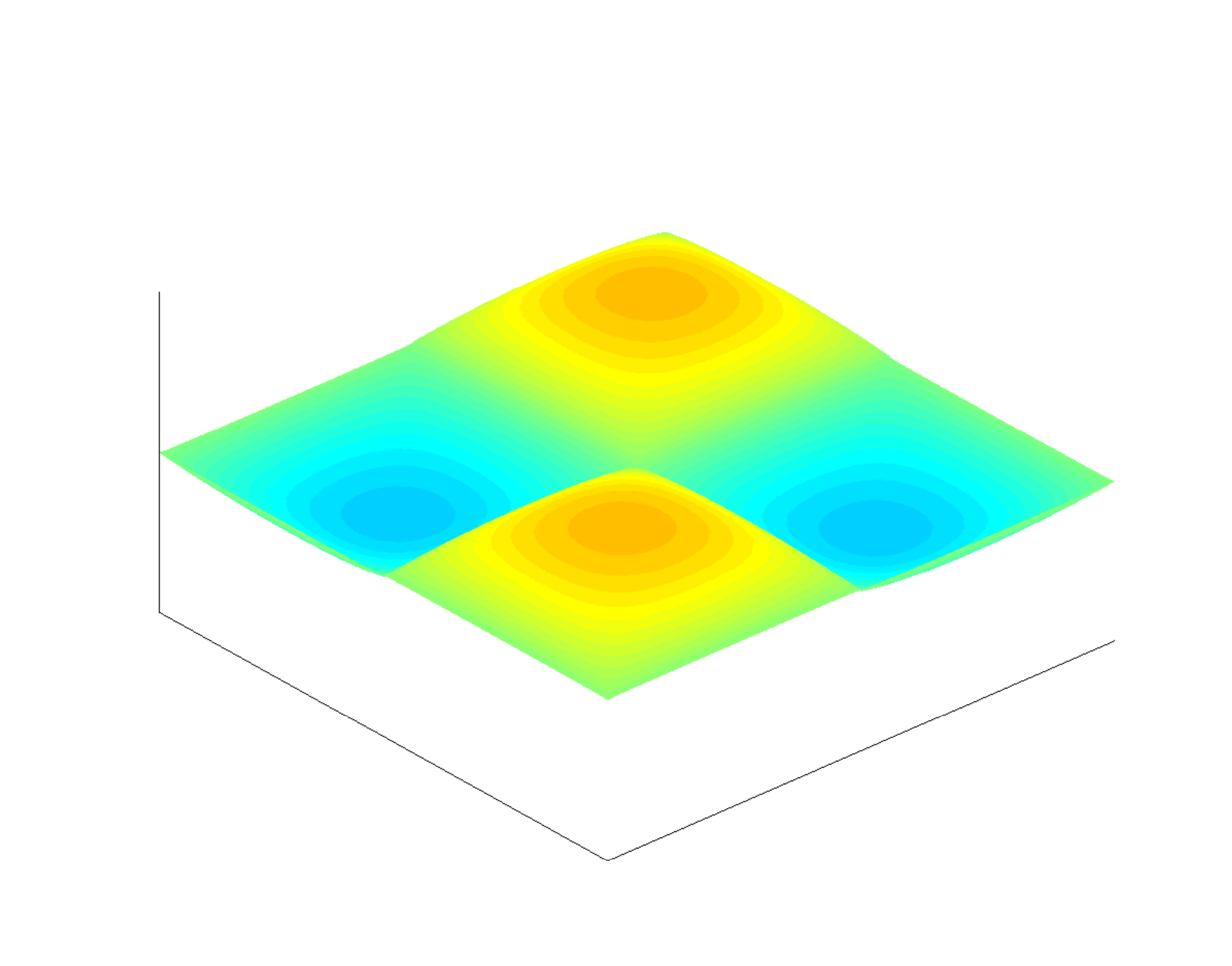} & 
\includegraphics[width=0.33\textwidth]{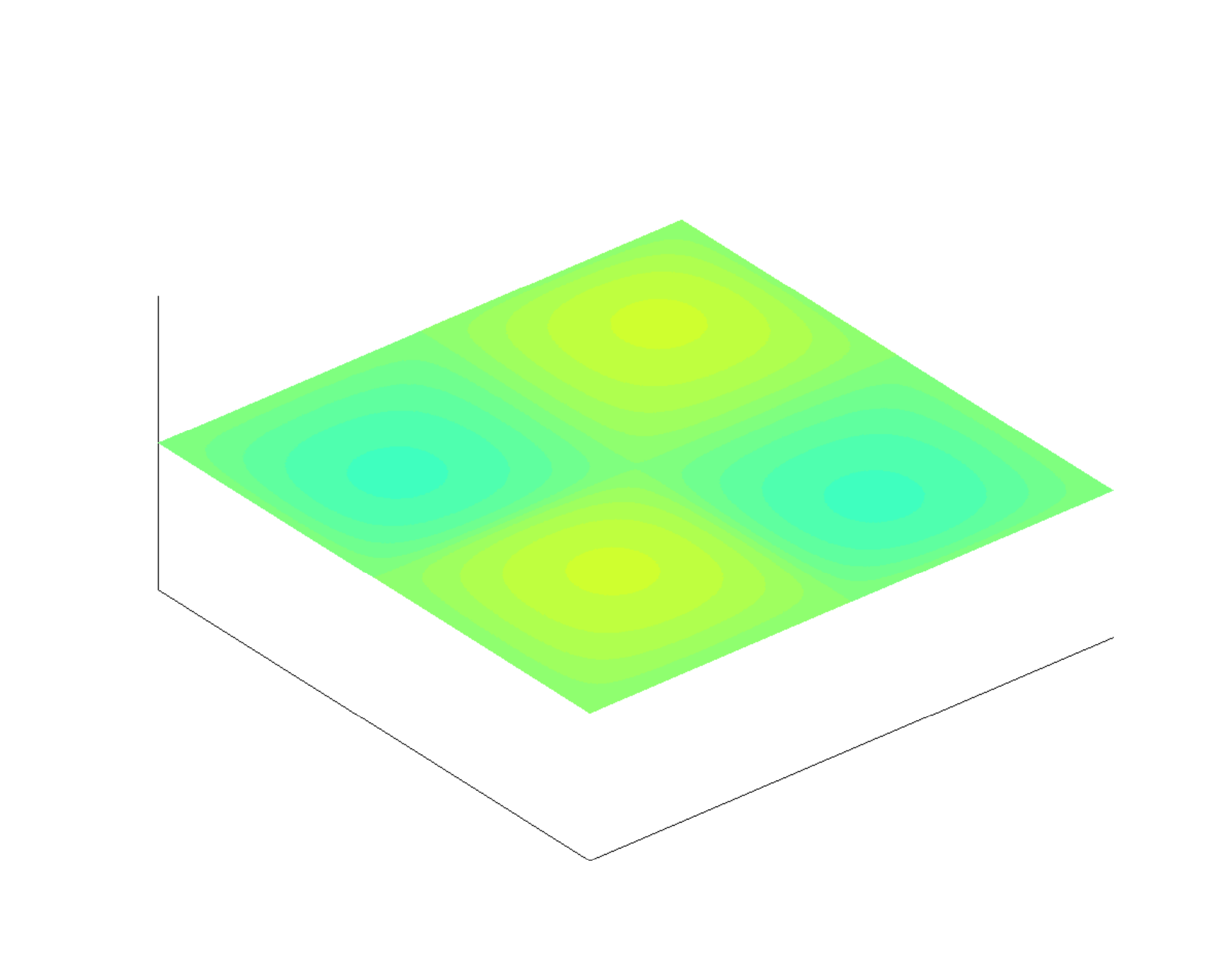} 
\end{tabular}
\caption{1-BURA-approximation of $\calAt^{-\alpha} \bfftil$ for $h=2^{-10}$}\label{fig:2D}
\end{figure}

Straighforward computations for the considered three choices of $\alpha$ and $\bfu_Q$ imply
$$\#\text{ systems} = \lceil (1-\alpha)k_Q\rceil +  \lceil \alpha k_Q\rceil + 1=\left\{
\begin{array}{ll}
k_Q+1+\lceil k_Q\pmod 4\rceil,& \alpha=\{0.25,0.75\};\\
k_Q+1+\lceil k_Q\pmod 2\rceil,& \alpha=0.50.
\end{array}
\right.
$$
Therefore, for $\alpha=0.25$ $\bfu_Q$ can never consist of $k+1=10$ summands, like $\bfu_r$; for $\alpha=0.5$ both $k_Q=\{7,8\}$ lead to $k+1=9$ linear systems for $\bfu_Q$; for $\alpha=0.75$ only $k_Q=6$ leads to $k+1=8$ linear systems for $\bfu_Q$.

\begin{table}
\centering
\caption{Relative $\ell_2$ errors for 2-D fractional diffusion. Top: $(k,k)$ 1-BURA. Bottom: The \cite{BP15} solver.}\label{table:2D-error}
 \begin{tabular}{|c|ccc|}\hline
   & $\alpha=0.25$ & $\alpha=0.5$ & $\alpha=0.75$\\\hline
  $\|\bfu_{{\text{ref}}}-\bfu_r\|_2/\|\bfftil\|_2$ & $1.756$E-4 & $3.833$E-4 & $4.180$E-4\\
  $\|\bfu_{{\text{ref}}}-\bfu_Q\|_2/\|\bfftil\|_2$ & $9.375$E-3 & $2.830$E-3 & $1.088$E-3\\\hline 
 \end{tabular}
\end{table}

Relative $\ell_2$ errors are documented in Table~\ref{table:2D-error}. To get the
approximate solution $\bfu_Q$ we consider $k_Q=\{9,7,6\}$, when $\alpha=\{0.25,0.5,0.75\}$. 
We observe that in all three cases the relative error of the BURA approximate solution $\bfu_r$  
is smaller than the error of $\bfu_Q$. Furthermore, for each $\alpha$ we keep on increasing 
$k_Q$ by one until the corresponding relative $\ell_2$ error of $\bfu_Q$ becomes smaller than the 
approximate solution $\bfu_r$ obtained for $(k,k)$ 1-BURA method. For $\alpha=0.25$, 
we get $k_Q=38$ to be the smallest such integer, meaning that we need to solve 4 times 
more linear systems ($40$ compared to $10$) in order to beat the numerical accuracy of BURA. 
For $\alpha=0.5$, we get $k_Q=20$, thus we need to solve $21$ linear systems if we 
apply \cite{BP15} instead of $9$, when we apply the BURA solver. Finally, for $\alpha=0.75$, 
we get $k_Q=13$ and $15$ linear systems to be solved, compared to $8$ in the BURA case. 
The relative errors as functions of the number of linear systems in $\bfu_Q$ are 
presented on Figure~\ref{fig:ErrorPlots}.

\begin{figure}
\begin{tabular}{ccc}
$\alpha=0.25$, $k=9$ & $\alpha=0.5$, $k=8$ & $\alpha=0.75$, $k=7$\\
\includegraphics[width=0.33\textwidth]{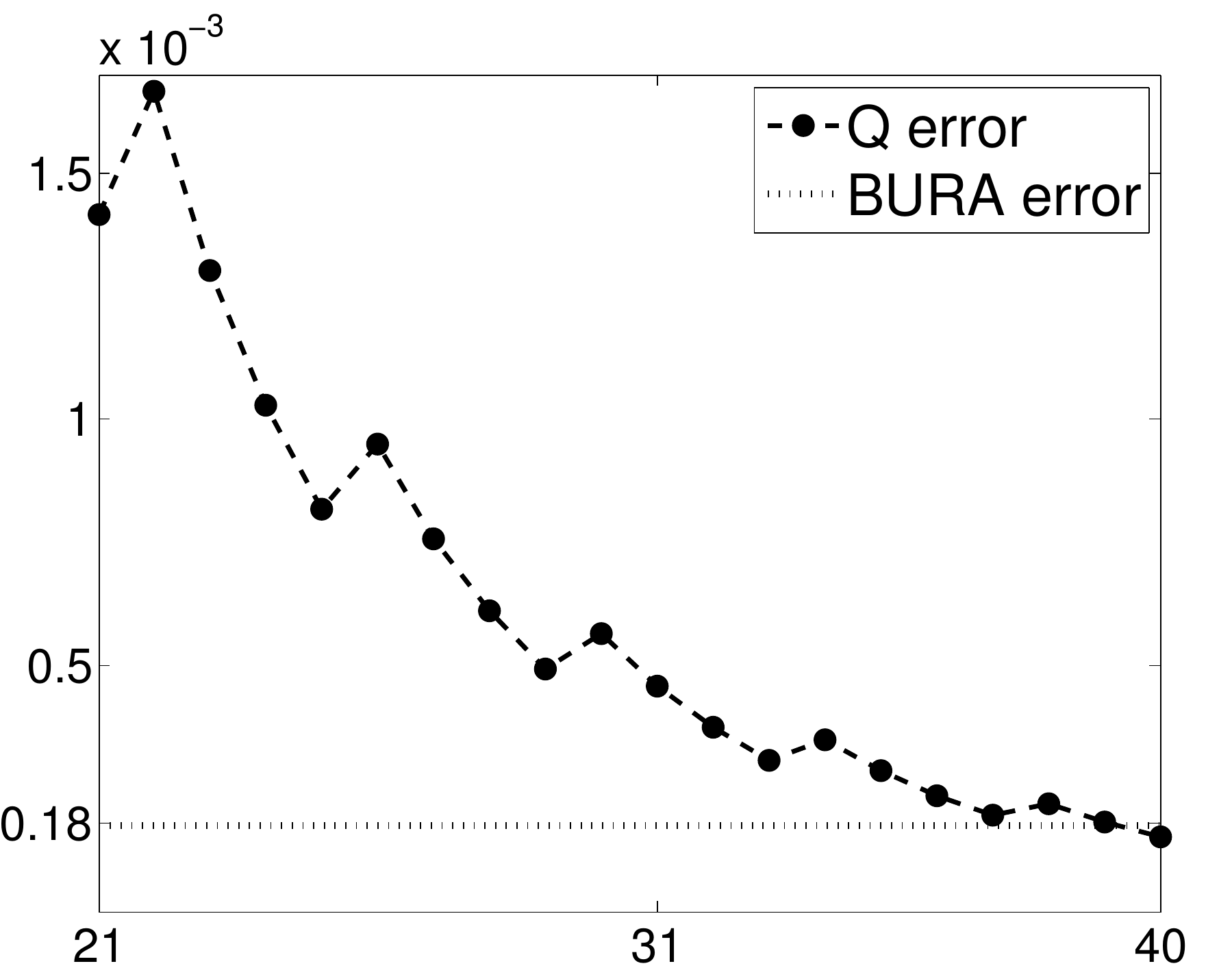} & 
\includegraphics[width=0.33\textwidth]{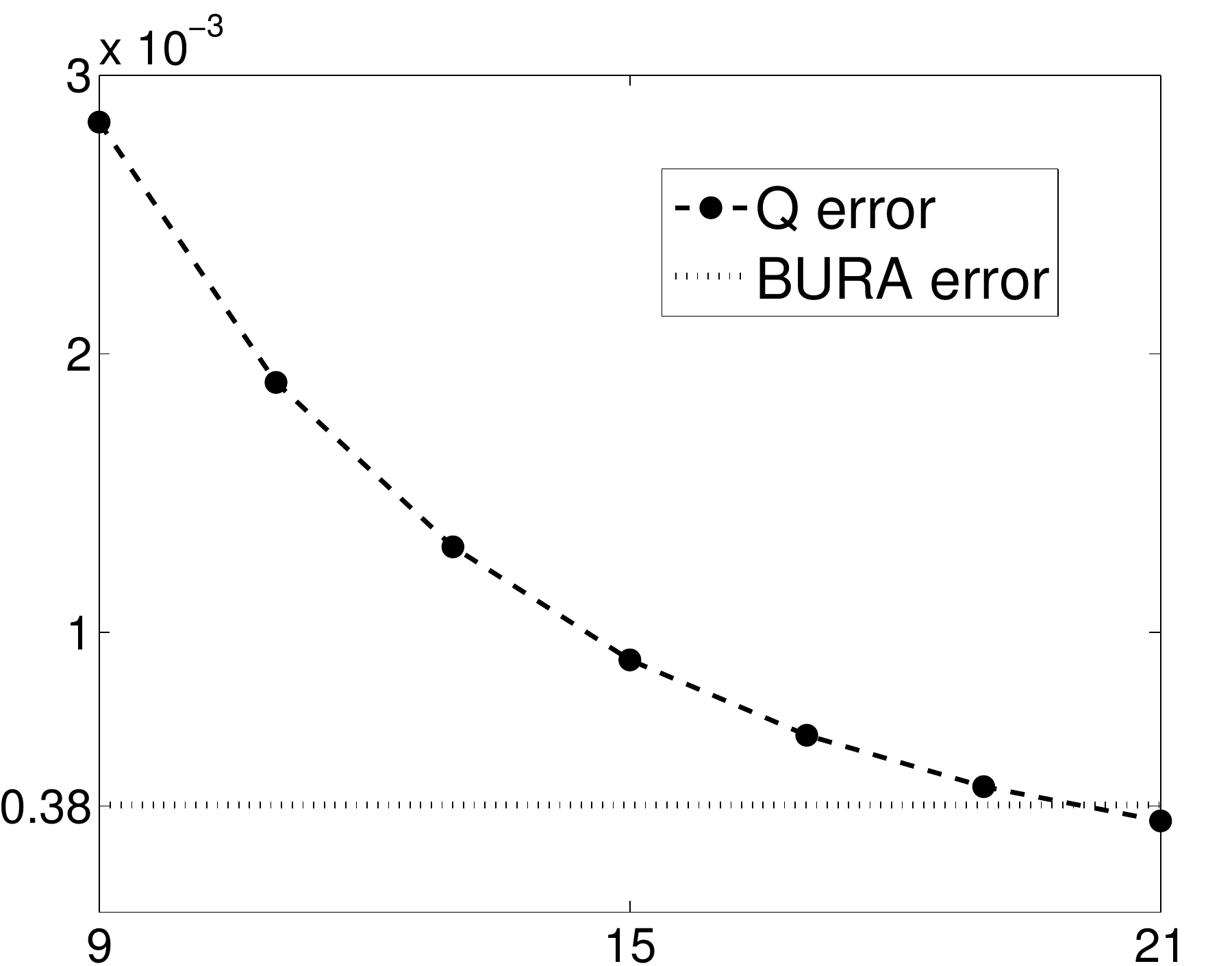} & 
\includegraphics[width=0.33\textwidth]{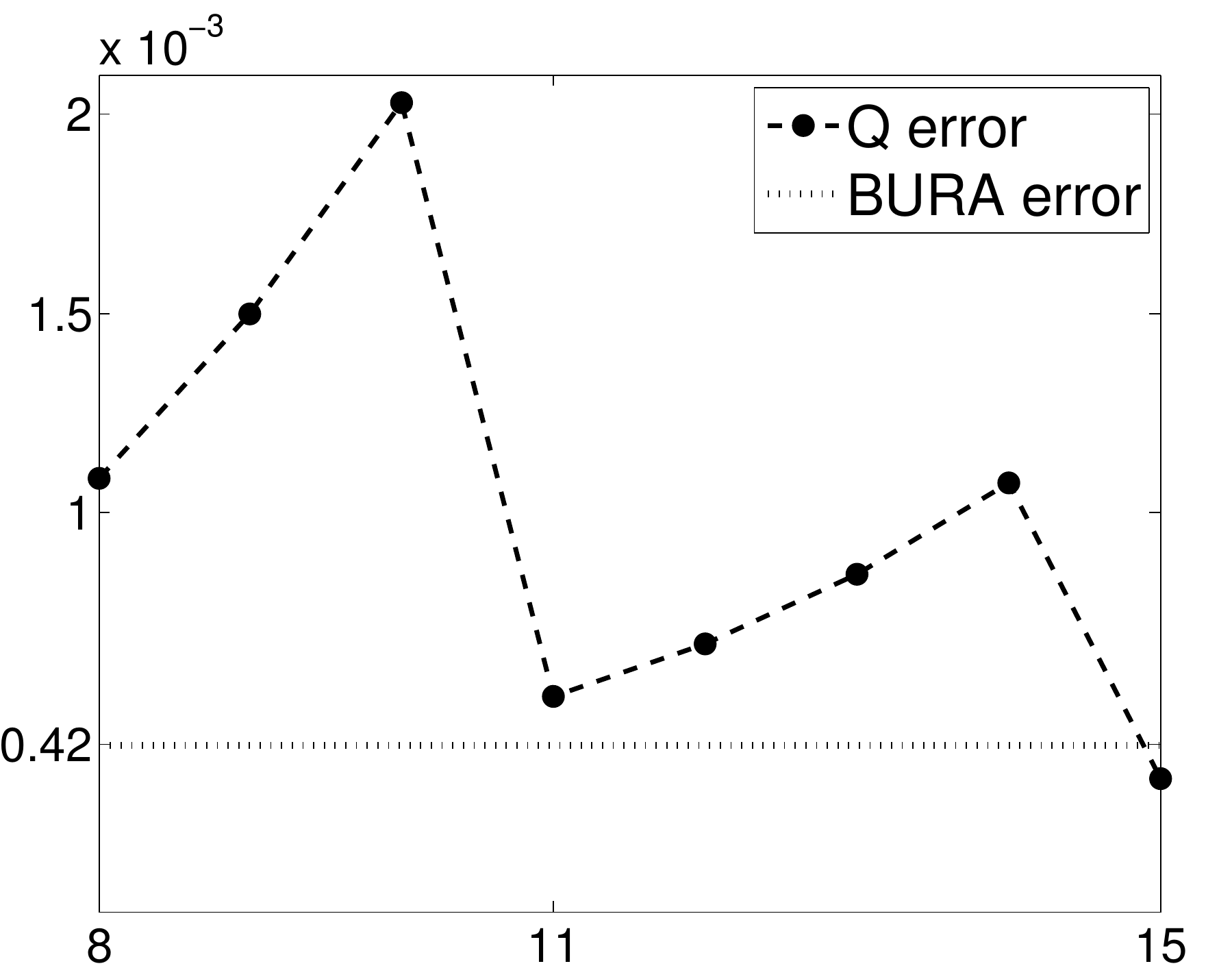} 
\end{tabular}
\caption{Relative $\ell_2$ errors for $\bfu_Q$ as functions on the number of solved linear systems.}\label{fig:ErrorPlots}.
\end{figure}
\subsection{1D and 3D numerical tests with approximate solving  of $\calA \bfu =\bff$ by PCG method}\label{sec43}

In higher spatial dimensions, when the domain $\Omega$ in \eqref{eqn:weak} is a subset of $\R^d$, $d>1$,
we cannot compute 
the exact solution $\bfu=\calA^{-\alpha}\bff$. In general, we don't have explicitly the 
eigenvalues and eigenvectors of $\calA$. Thus, in the analysis of the numerical tests, we cannot apply 
error estimates of the form \eqref{bound}.
In order to numerically validate our  theoretical estimates
we use the two-step procedure from Section~\ref{sec:multiBURA} through the following obvious identity 
\begin{equation}\label{eq:f_approx}
\bff =\calA\left(\calA^{-(1-\alpha)}\big(\calA^{-\alpha}\bff\big)\right), 
\end{equation}
that holds true for an arbitrary vector $\bff\in\R^N$. In \eqref{eq:general 2-step error} we 
argued that taking a  product of two rational functions as an approximation of $t^{\beta-\alpha}$ with $\beta=1$ the 
corresponding 
approximation error depends on the condition number of $\calA$. For the multivariate validation of 
Lemma~\ref{l:BURAerror} such dimension-dependence is not acceptable, so we choose $\beta=2$ here. 
In particular, we treat 
$r^1_{1-\alpha}r^1_\alpha$ as a 2-uniform rational approximation 
(this is not BURA, since the approximation error is not optimal) for the function 
$t^{2-1}=t$ on the unit interval $(0,1]$. Applying \eqref{bound} with $\gamma=2$ we deduce
\begin{equation}\label{eq:product}
\|r^1_{1-\alpha}(\calA)r^1_{\alpha}(\calA)\calA^{-2}\bff-\calA^{-1}\bff\|_{\calA^2}\le 
E_{1-\alpha,\alpha}(k,k;2)\|\bff\|_{\calA^{-2}}.
\end{equation}
To avoid additional numerical inaccuracies, it is more convenient from a computational point of view to introduce the 
approximation vector $\bff_r$ of $\bff$  in \eqref{eq:f_approx}, i.e., 
$$
\bff_r 
=\calA \Big (r^1_{1-\alpha}(\calA)r^1_{\alpha}(\calA)\calA^{-2}\bff \Big ):= \calA\bfu_r.
$$
Then we can rewrite the estimate \eqref{eq:product} in the form 
\begin{equation}\label{eq:3D error}
 {\|\bff_r-\bff\|}/{\|\calA^{-1}\bff\|}\le E_{1-\alpha,\alpha}(k,k;2).
\end{equation}
We estimate the two-step error $E_{1-\alpha,\alpha}(k,k;2)$ with the help of \eqref{eq:residual}, as in 
\eqref{eq:multi-step error}. The function $r^1_{1-\alpha}r^1_\alpha$ has no poles in $[0,1]$, thus we need not
restrict ourselves to the spectrum of $\calA$:
\begin{equation*}
\begin{split}
E_{1-\alpha,\alpha}(k,k;2)&=\max_{t\in[0,1]}|r^1_{1-\alpha}(t)r^1_\alpha(t)-t|=\max_{t\in[0,1]}
|t^\alpha\varepsilon_\alpha(t)+t^{1-\alpha}\varepsilon_{1-\alpha}(t)+\varepsilon_\alpha(t)\varepsilon_{1-\alpha}(t)|\\
&\le E_{1-\alpha}(k,k;1) + E_{\alpha}(k,k;1) + E_{1-\alpha}(k,k;1) E_{\alpha}(k,k;1).
\end{split}
\end{equation*}
In practice, however, we observe that the residuals $\varepsilon_\alpha$ and $\varepsilon_{1-\alpha}$ 
are negative  and 
monotonically decreasing in $[0.8,1]$. Due to this sign pattern and the following 
equalities  $\varepsilon_\alpha(1)=-E_\alpha(k,k;1)$,  $\varepsilon_{1-\alpha}(1)=-E_{1-\alpha}(k,k;1)$, 
we can improve the two-step error estimate and conclude that:
\begin{equation}\label{eq:3D error final}
 {\|\bff_r-\bff\|}/{\|\calA^{-1}\bff\|}\le E_{1-\alpha} + E_{\alpha} - E_{1-\alpha} E_{\alpha} .
\end{equation}

The last estimate has been numerically confirmed as sharp for $k=\{5,7\}$. 
Inspecting closely the above proof, we realize that 
the two-step residual is of order $E_{1-\alpha} E_{\alpha}$ or lower around zero. Furthermore, unlike the 
equioscillation BURA setting for the one-step process, the two-step residual reaches its maximum in absolute value only 
at $t=1$ and the amplitudes of its other oscillations gradually decrease as $t$ approaches zero. As a result, the 
numerically computed values for the maximal and averaged random errors w.r.t. \eqref{eq:3D error final} are closer to 
$E_{1-\alpha} E_{\alpha}$ than to $E_{1-\alpha} + E_{\alpha}$, meaning that the typical error is significantly smaller 
than the worst case scenario $E_{1-\alpha,\alpha}$.      

We investigate two particular choices for $\bff$, namely $\bff^1=(1,\dots,1)$ and $\bff^0=(1,0,\dots,0)$. In 1D, the 
matrix $\calA$ remains $tridiag (-0.25,0.5,-0.25)$ and as before  we 
exactly  solve the corrsponding  linear systems. In 3D, we use the finite element method in space with linear 
conforming tetrahedral 
finite elements and an algebraic multigrid (AMG) preconditioner in the PCG solutions of the corresponding linear  
systems. 
To be more precise, the BoomerAMG implementation, e.g. \cite{HENSON2002155}, is utilized in the presented numerical 
tests.
We consider $\Omega=[0,1]^3$ and $\calA$ to be the stiffness matrix from the FE discretization of 
the problem \eqref{eqn:weak} with ${\bf a}= a(x) I$, with $I$ the identity matrix in $\R^d$
and $a(x)$ is a piece-wise constant function in $\Omega$.
In this case, the jump of the coefficient $a(x)$ is introduced via the scaling factor $0 < \mu \le 1$.

The motivation for choosing these particular $\bff$'s comes from the 1D case. Since for $i=1,\dots,N$
\begin{equation*}
\langle\bPsi_i,\bff^1\rangle=\left\{\begin{array}{ll} 0, & i ~~\mbox{is even} \\ \cot{(i\pi h/2)}, & i ~~\mbox{is odd} 
\end{array}\right.
\qquad \langle\bPsi_i,\bff^0\rangle=\sin(i\pi h)=\bPsi_{1,i}, 
\end{equation*}
the decompositions of the two vectors with respect to the eigen-vectors $\{\bPsi_i \}_{i=1}^N$ are
\begin{align}\label{eq:bff1}
& \bff^1= 
\sum_{i  ~\mbox{\tiny{is odd}}} 2h
\cot(i\pi h/2)\bPsi_{i}=\sum_{i ~\mbox{\tiny{is even}}}\frac{4}{i\pi} \frac{i\pi h/2}{\tan(i\pi h/2)}\bPsi_{i}; \quad
\bff^0= 
\sum_{i=1}^N 2h\sin(i\pi  h)\bPsi_i.
\end{align}
Therefore, from $x/\tan(x)<1$ in $(0,\pi/2)$, we derive that the coefficients in the $\bff^1$-decomposition 
\eqref{eq:bff1} rapidly decay as $i$ increases, meaning that the $\bPsi_1$ component dominates and the two-step 
residual at $\Lambda_1$ determines the behavior of the error ratio $\|\bff^1_r-\bff^1\|/\|\calA^{-1}\bff^1\|$. To 
summarize
\begin{equation}
{\|\bff^1_r-\bff^1\|}/{\|\calA^{-1}\bff^1\|}\approx\left|r^1_{1-\alpha}(\Lambda_1)r^1_{\alpha}(\Lambda_1)-\Lambda_1 
\right|\xrightarrow[h \to 0]{} E_{1-\alpha} E_{\alpha}.           
\end{equation}

\begin{figure}[htp]
\begin{center}
\includegraphics[width=0.49\textwidth]{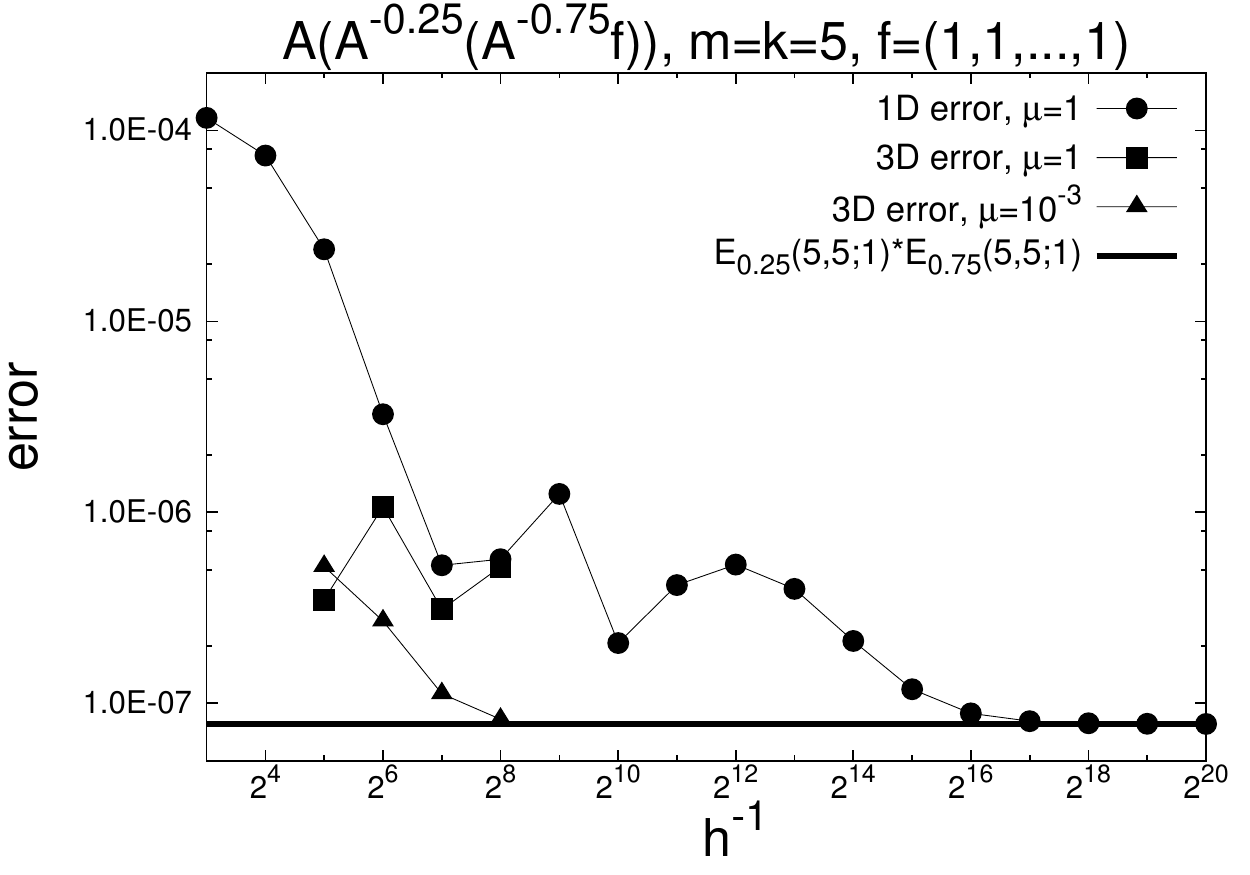} 
\includegraphics[width=0.49\textwidth]{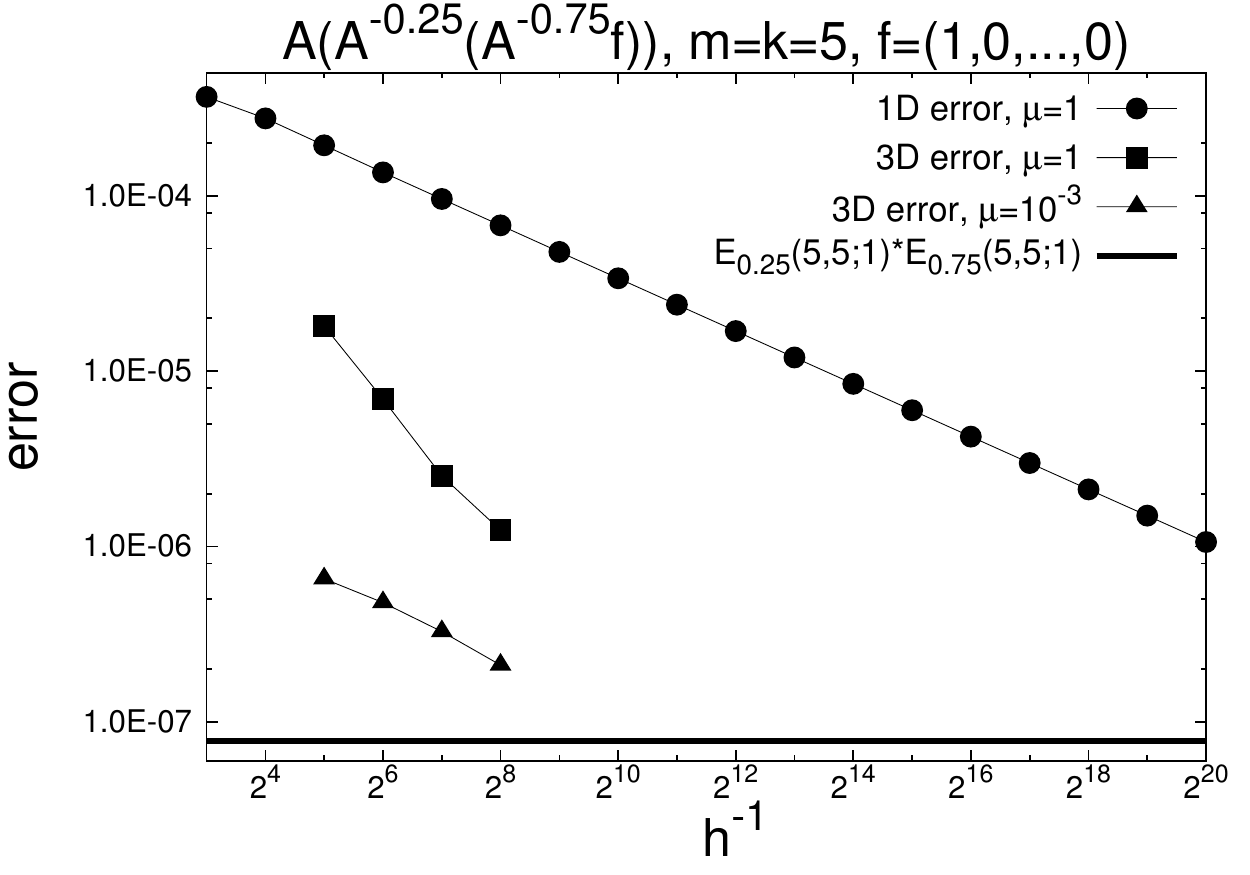} \\ 
\includegraphics[width=0.49\textwidth]{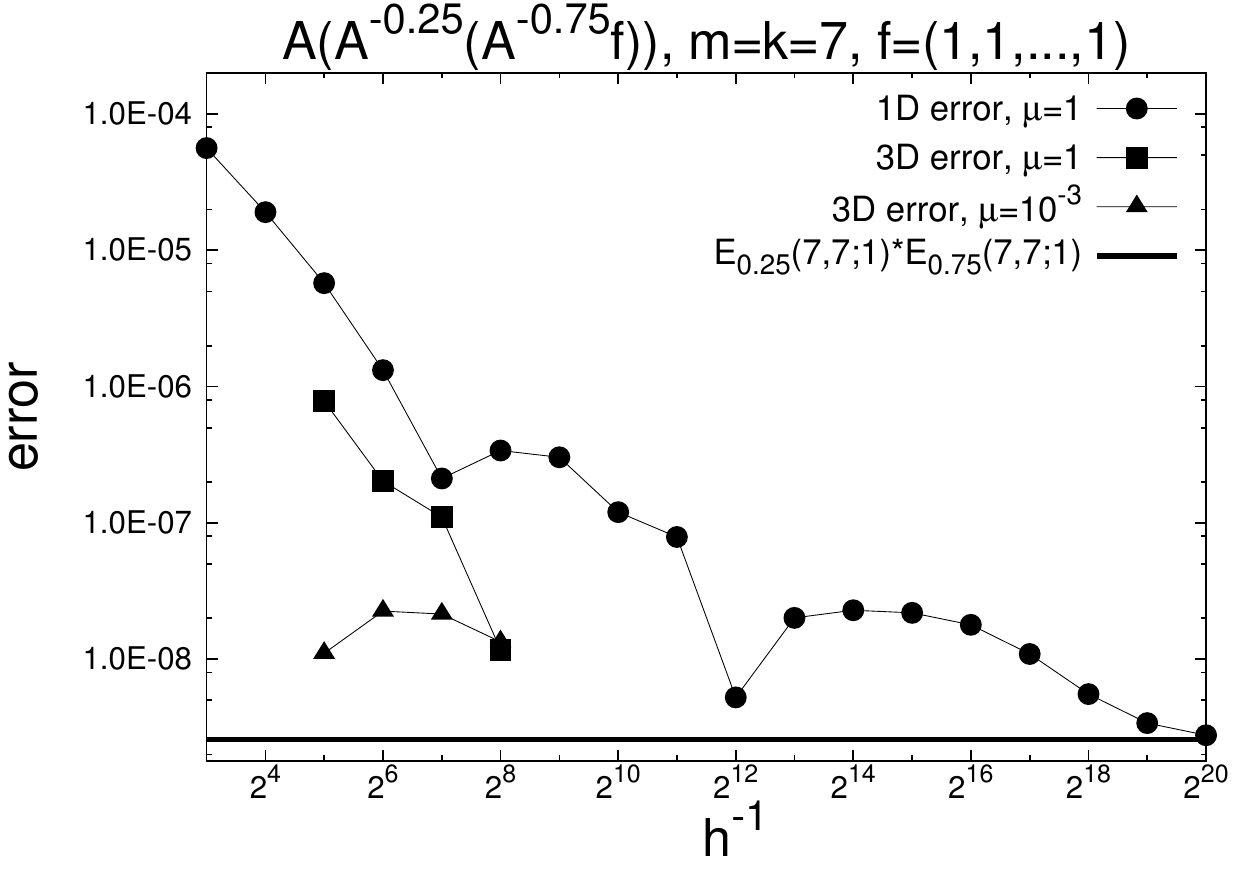} 
\includegraphics[width=0.49\textwidth]{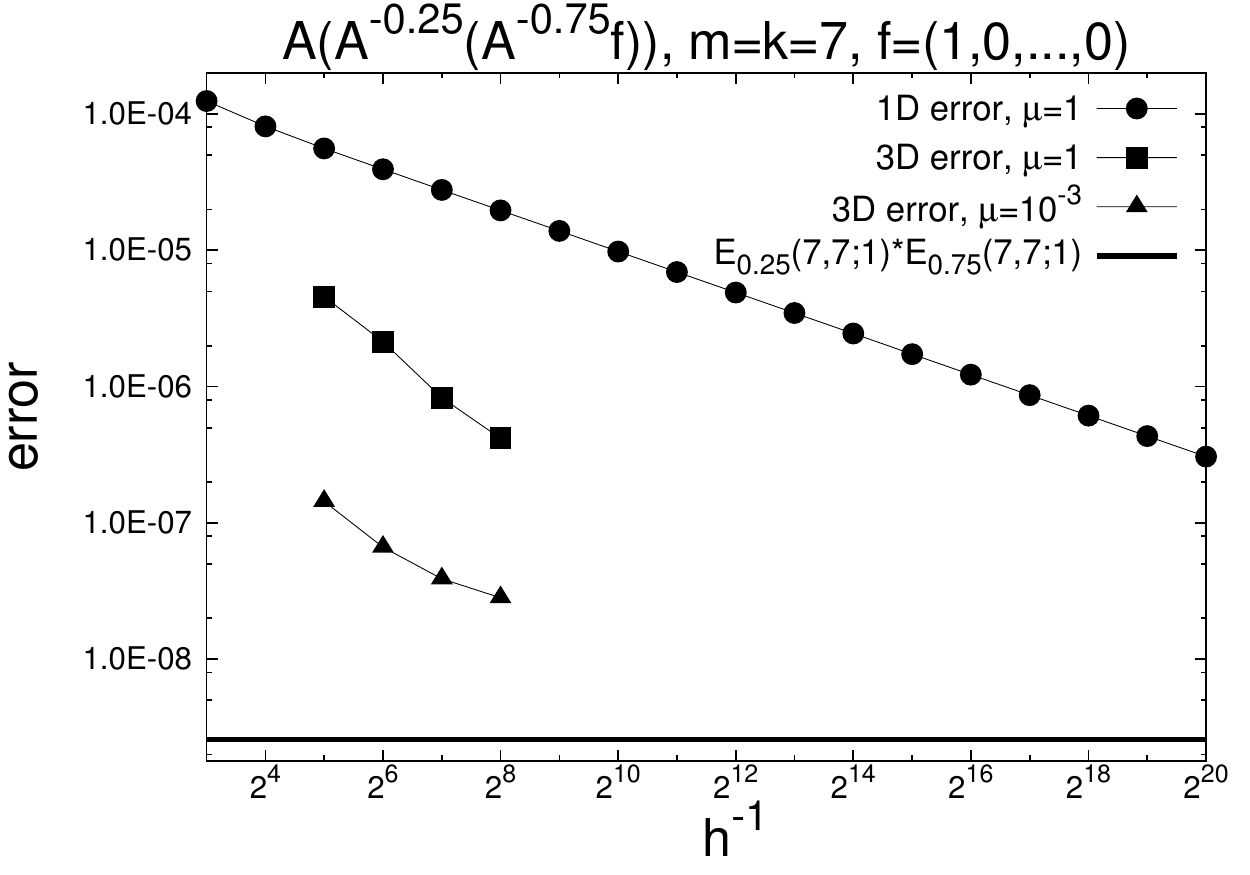} \\ 
\includegraphics[width=0.49\textwidth]{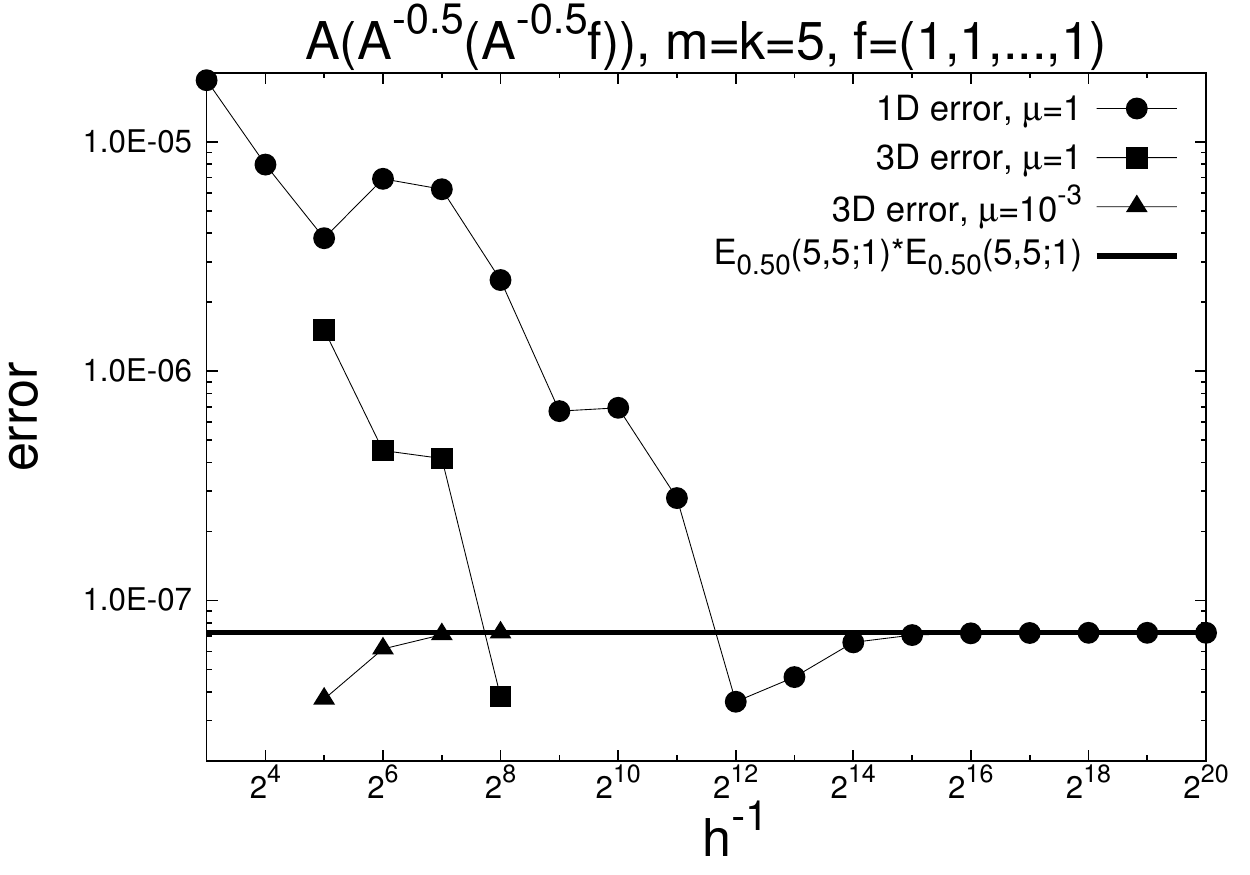} 
\includegraphics[width=0.49\textwidth]{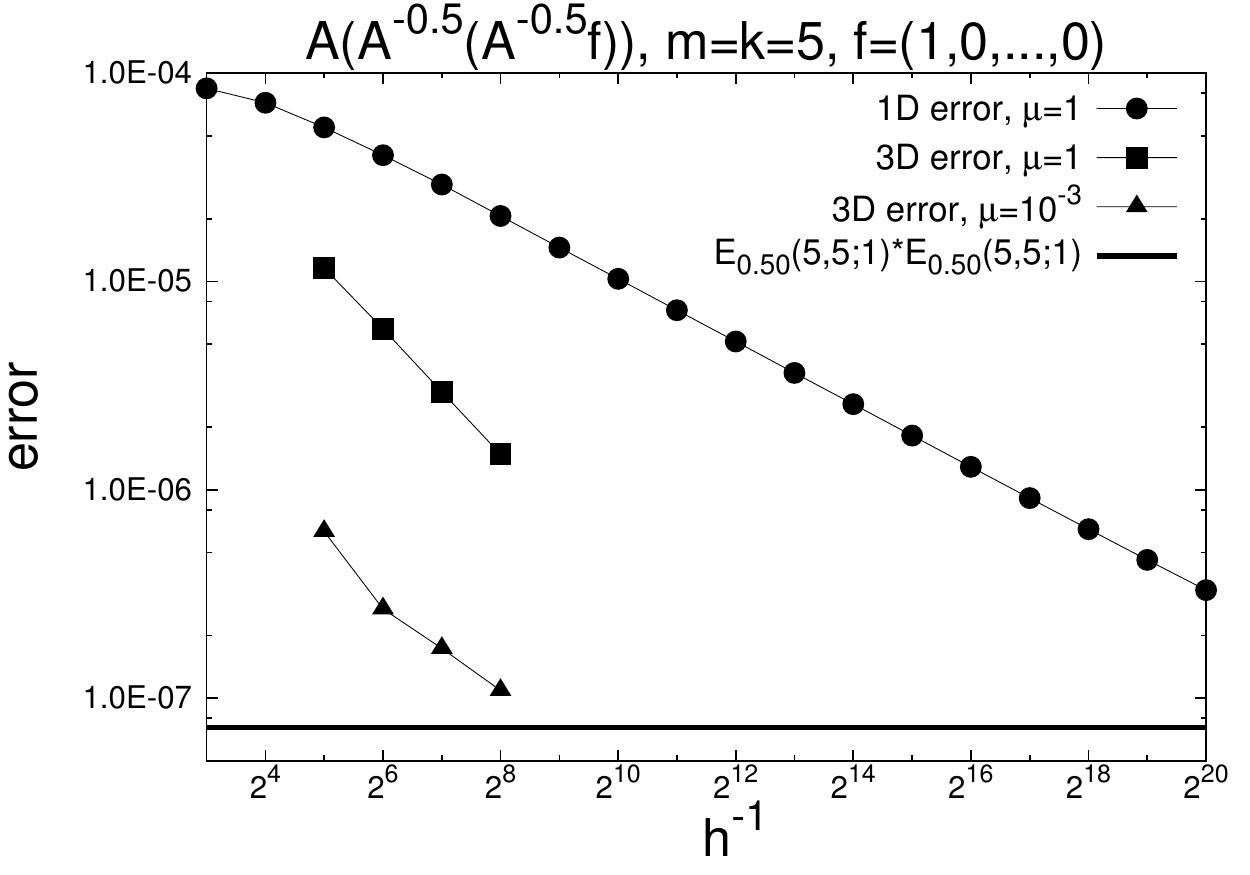} \\ 
\includegraphics[width=0.49\textwidth]{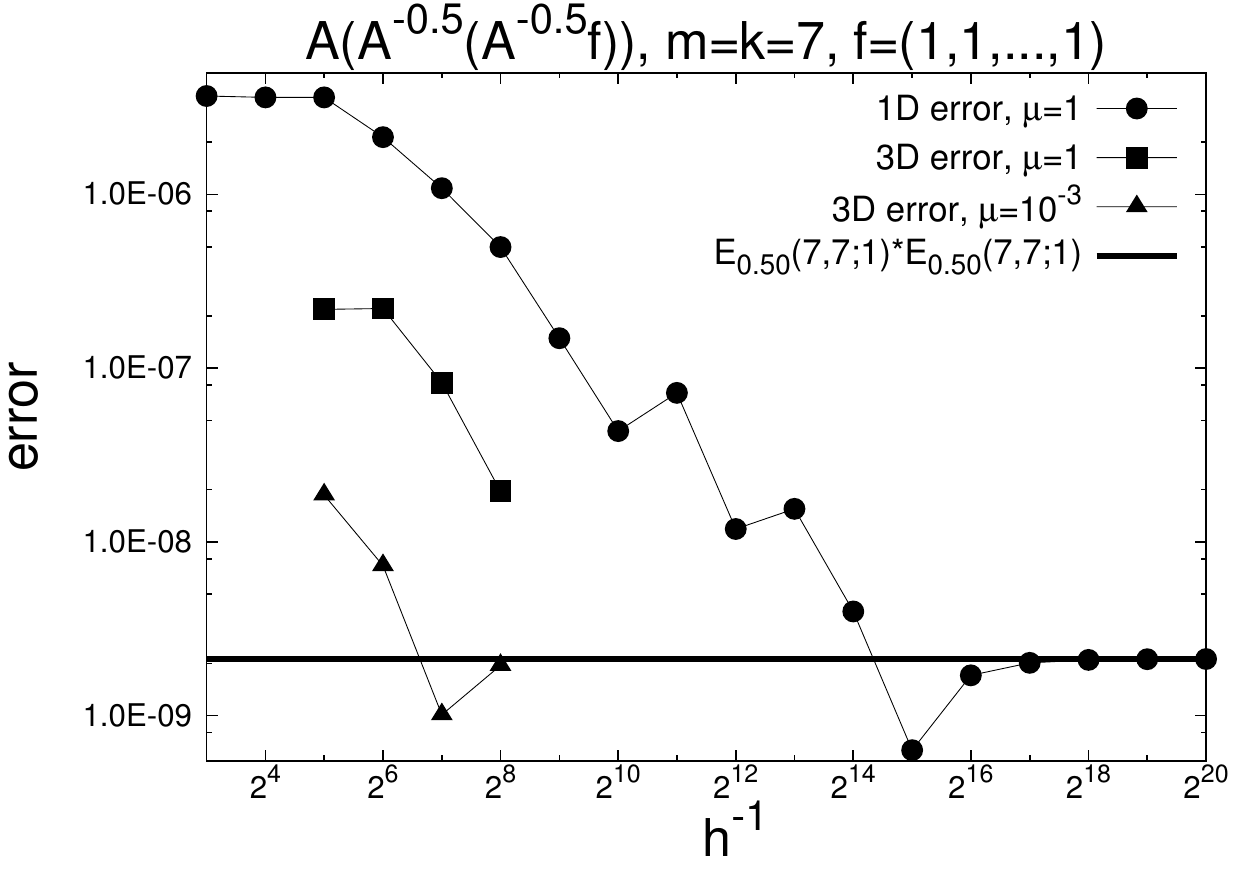} 
\includegraphics[width=0.49\textwidth]{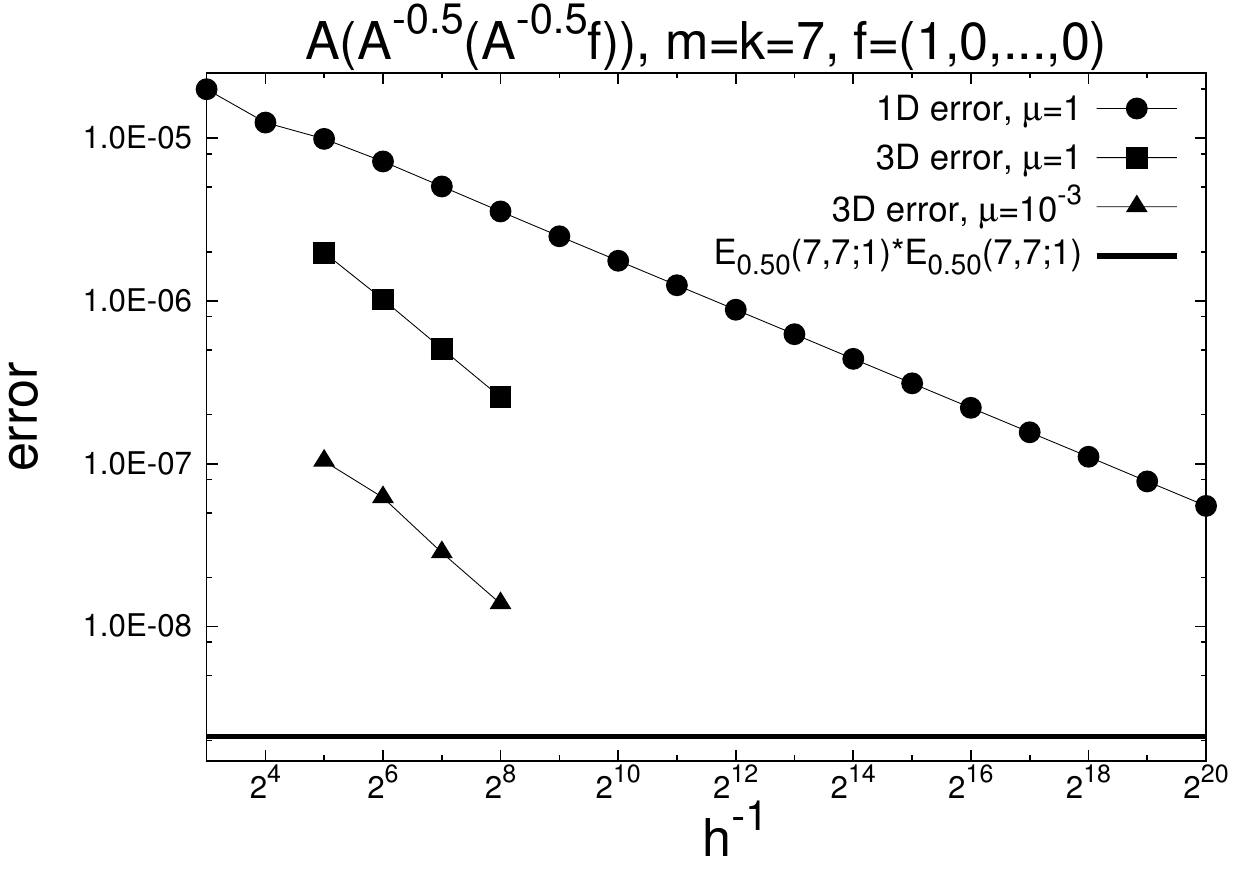} 
\end{center}
\vspace*{-0.5cm}
\caption{1D and 3D numerical error analysis. Left:  $\bff^1=(1,1,\dots,1)$. Right: $\bff^0=(1,0,\dots,0)$. The relative 
errors $\|\bff_r-\bff\|/\|\calA^{-1}\bff\|$ are plotted.}\label{fig:IDCheck} 
\end{figure}

Such asymptotic behavior of the $\ell^2$-norm error ratio of $\bff^1$  is numerically confirmed by the conducted 1D 
and 3D numerical experiments with $k=\{5,7\}$ and $\alpha=\{0.25,0.5\}$, as illustrated on the left of 
Fig.~\ref{fig:IDCheck}. 
In 3D we run simulations up to $h=2^{-8}$, which corresponds to $N=6(h^{-1}+1)^3$, while in 1D we go up to 
$h=2^{-20}$, and corresponding number of degrees of freedom $N=h^{-1}+1$. 

Clearly, the error 
tends to $E_{1-\alpha} E_{\alpha}$ as $\Lambda_1\to 0$. We observe that in 
the 3D case with homogeneous coefficient ($\calL$ is the Laplacian)
the error for mesh-size $h^{-1}\in[2^{6},2^8]$ mimics the 1D error for  mesh-size $h^{-1}\in[2^{9},2^{12}]$. 
For the heterogeneous 3D case, when in a half the unit cube the diffusion coefficients are scaled by $\mu=10^{-3}$, 
$\mathrm{k}(\calA)$ is increased (approximately) by a factor of $\mu^{-1}$,  implying that $\Lambda_1$ is closer 
to zero than in the homogeneous case. As a result, the 3D error for the 
mesh-size $h^{-1}\in[2^{6},2^8]$ mimics the 1D error for the mesh-size $h^{-1}\in[2^{14},2^{17}]$. 

The coefficients in the $\bff^0$-decomposition \eqref{eq:bff1} have symmetry due to the relations 
 $2h\sin(i\pi h)=2h\sin((N+1-i)\pi h$ with $h=1/(N+1)$. 
Unlike for $\bff^1$ case, the 
contribution of $\bPsi_1$ here is negligible and the behavior of the relative error 
$\|\bff^0_r-\bff^0\|/\|\calA^{-1}\bff^0\|$ is dominated by 
the error $|r^1_{1-\alpha}(0.5)r^1_{\alpha}(0.5)-0.5|$. This effect 
weakens with $h\to0$, because the coefficients depend on the 
mesh-size and their distribution spreads away (standard 
deviation increases) when the grid is refined. The two-step residual is stable around $t=1/2$ and the coefficients 
decay proportionally to the refinement scale, which results in monotone linear behavior of the error as a function of 
$h^{-1}$. The numerical results perfectly agree with this argument and, similar to the 1D--3D correspondence for 
$\bff^1$, we observe that the slope of the error decay is steeper in 3D and the error decreases in 
the case of piece-wise constant coefficient ${\bf a}(x)$. 

In the presented numerical tests, as a stopping criteria for the BoomerAMG PCG 
solver we have used a relative error less or equal to $10^{-12}$. 
However, we want to note that the numerical results 
are practically not affected by using stopping criteria
$10^{-6}$, instead. Furthermore, for precision $10^{-12}$ the order of 
applying the 1-BURA functions $r^1_{0.25}$ and $r^1_{0.75}$ (i.e., taking $\alpha=0.25$ or $\alpha=0.75$ first) seems 
irrelevant and the corresponding relative errors have the same  first five meaningful digits. 
This implies that the main numerical difficulties are related to the performance of Algorithm~\ref{alg1} and the 
correctness of the subsequent 
representation of $r^\beta_\alpha$ as a sum of partial fractions.

\section{Concluding remarks}

In this paper we propose algorithms of optimal complexity for solving the 
linear algebraic  system $\calA^\alpha \bfu=\bff$, $0< \alpha <1$, where 
$\calA$ is a sparse SPD matrix.
The target class of applied problems 
$\calA$ obtained 
by a finite difference or finite element discretization of a second order 
elliptic problem. 
Our main assumption is  that the system 
$\calA \bfu=\bff$ can be solved 
with optimal computational complexity, e.g. by multi-grid, multi-level 
or other efficient solution 
technique. The proposed in the paper method is applicable also when the 
matrix is not given explicitly, but one has at hand an optimal solution 
procedure for the linear system $\calA \bfu=\bff$ and a upper bound for 
the spectrum of $\calA$. 


The method is based on best uniform rational approximations (BURA) of 
$t^{\beta-\alpha}$ for $0 \le t \le 1$
and natural $\beta$. Bigger $\beta$ means stronger regularity assumptions 
and this is the reason to concentrate our considerations mostly to the cases 
$\beta \in \{1,2\}$. Depending on $\alpha$, $\beta$ and the degree $k$ of 
best uniform rational approximation 
a relative accuracy of the method between $O(10^{-3})$ and  $O(10^{-7})$ 
can be obtained  for $k \in \{5,6,7\}$. 
Then solution of $\calA^\alpha \bfu=\bff$  reduces to  solving $k+\beta$ 
problems with sparse SPD matrices of the form $\calA +c \calI$, $c \ge 0$.
 
The method has been extensively tested on a number system  arising in finite element 
approximation of one- and three-dimensional elliptic problems of second order.  
In the 3D examples we have used BoomerAMG PCG solver, \cite{HENSON2002155},
of optimal complexity.


Unlike the integral quadrature formula method from \cite{BP15}, the approximation
properties of BURA algorithm are not symmetric with respect to $\alpha=0.5$, 
$\alpha\in (0,1)$.  Some favorable results are presented for the standard (one-step) 
1-BURA, $ m=k$, in the case of smaller $\alpha$. For larger $\alpha$, the 
multi-step algorithm has some promising features. Future theoretical and experimental
investigations are needed for better understanding the observed
superior convergence of two-step BURA when $\alpha_1\ne\alpha_2$.

The method and the integral quadrature formula method from \cite{BP15} have been experimentally compared. The test setup has been taken from \cite[Section 4.1]{BP15} with $h=2^{-10}\approx10^{-3}$. The BURA method performs better in all numerical experiments and this effect increases as $\alpha$ decreases.

\section*{Acknowledgement}
This research has been partially supported by the Bulgarian National Science Fund under grant No. BNSF-DN12/1. The work of  R. Lazarov has been partially supported by the grant NSF-DMS \#1620318. The work of S. Harizanov and Y. Vutov has been partially supported by the Bulgarian National Science 
Fund under grant No. BNSF-DM02/2.

  \bibliographystyle{abbrv}
  \bibliography{references}

\begin{thebibliography}{10}

\bibitem{bakunin2008turbulence}
O.~G. Bakunin.
\newblock {\em Turbulence and diffusion: scaling versus equations}.
\newblock Springer Science \& Business Media, 2008.

\bibitem{bates2006nonlocal}
P.~W. Bates.
\newblock On some nonlocal evolution equations arising in materials science.
\newblock {\em Nonlinear dynamics and evolution equations}, 48:13--52, 2006.

\bibitem{BP15}
A.~Bonito and J.~Pasciak.
\newblock Numerical approximation of fractional powers of elliptic operators.
\newblock {\em Mathematics of Computation}, 84(295):2083--2110, 2015.

\bibitem{BP17}
A.~Bonito and J.~Pasciak.
\newblock Numerical approximation of fractional powers of regularly accretive
  operators.
\newblock {\em IMA J Numer Anal}, 37(3):1245--1273, 2017.

\bibitem{CNOSaldago_2016}
L.~Chen, R.~Nochetto, O.~Enrique, and A.~J. Salgado.
\newblock Multilevel methods for nonuniformly elliptic operators and fractional
  diffusion.
\newblock {\em Mathematics of Computation}, 85:2583--2607, 2016.

\bibitem{CheneyPowell1987}
E.~W. Cheney and M.~J.~D. Powell.
\newblock The differential correction algorithm for generalized rational
  functions.
\newblock {\em Constructive Approximation}, 3(1):249--256, 1987.

\bibitem{druskin1998extended}
V.~Druskin and L.~Knizhnerman.
\newblock Extended {K}rylov subspaces: approximation of the matrix square root
  and related functions.
\newblock {\em SIAM Journal on Matrix Analysis and Applications},
  19(3):755--771, 1998.

\bibitem{Dunham1984}
C.~B. Dunham.
\newblock Difficulties in rational {C}hebyshev approximation.
\newblock In {\em Conference on Constructive Theory of Functions, Varna,
  Bulgaria}, pages 319--327, 1984.

\bibitem{GHK05}
I.~Gavrilyuk, W.~Hackbusch, and B.~Khoromskij.
\newblock Hierarchical tensor-product approximation to the inverse and related
  operators for high dimensional elliptic problems.
\newblock {\em Computing}, 74(2):131--157, 2005.

\bibitem{gilboa2008nonlocal}
G.~Gilboa and S.~Osher.
\newblock Nonlocal operators with applications to image processing.
\newblock {\em Multiscale Modeling \& Simulation}, 7(3):1005--1028, 2008.

\bibitem{HM17}
S.~Harizanov and S.~Margenov.
\newblock Positive approximations of the inverse of fractional powers of {SPD
  M}-matrices.
\newblock submitted, posted as arXiv:1706.07620v1, June 2017.

\bibitem{HMMV2016}
S.~Harizanov, S.~Margenov, P.~Marinov, and Y.~Vutov.
\newblock Volume constrained 2-phase segmentation method utilizing a linear
  system solver based on the best uniform polynomial approximation of
  $x^{-1/2}$.
\newblock {\em Journal of Computational and Applied Mathematics}, 310:115--128,
  2017.

\bibitem{HENSON2002155}
V.~E. Henson and U.~M. Yang.
\newblock Boomeramg: A parallel algebraic multigrid solver and preconditioner.
\newblock {\em Applied Numerical Mathematics}, 41(1):155 -- 177, 2002.

\bibitem{Higham1997}
N.~J. Higham.
\newblock Stable iterations for the matrix square root.
\newblock {\em Numerical Algorithms}, 15(2):227--242, 1997.

\bibitem{ilic2009numerical}
M.~Ili{\'c}, I.~W. Turner, and V.~Anh.
\newblock A numerical solution using an adaptively preconditioned {L}anczos
  method for a class of linear systems related with the fractional {P}oisson
  equation.
\newblock {\em International Journal of Stochastic Analysis}, 2008, 2009.

\bibitem{Kenney1991}
C.~Kenney and A.~J. Laub.
\newblock Rational iterative methods for the matrix sign function.
\newblock {\em SIAM J. Matrix Anal. Appl.}, 12(2):273--291, Mar. 1991.

\bibitem{KilbasSrivastavaTrujillo:2006}
A.~Kilbas, H.~Srivastava, and J.~Trujillo.
\newblock {\em Theory and {A}pplications of {F}ractional {D}ifferential
  {E}quations}.
\newblock Elsevier, Amsterdam, 2006.

\bibitem{LV17}
R.~Lazarov and P.~Vabishchevich.
\newblock A numerical study of the homogeneous elliptic equation with
  fractional order boundary conditions.
\newblock {\em Fractional Calculus and Applied Analysis}, 20(2):337--351, 2017.

\bibitem{PGMASA1987}
P.~G. Marinov and A.~S. Andreev.
\newblock A modified {R}emez algorithm for approximate determination of the
  rational function of the best approximation in {H}ausdorff metric.
\newblock {\em Comptes rendus de l'Academie bulgare des Scieces}, 40(3):13--16,
  1987.

\bibitem{MU93}
M.~Matsuki and T.~Ushijima.
\newblock A note on the fractional powers of operators approximating a positive
  definite selfadjoint operator.
\newblock {\em J. Fac. Sci. Univ. Tokyo Sect. IA Math.}, 40(2):517--528, 1993.

\bibitem{mccay1981theory}
B.~McCay and M.~Narasimhan.
\newblock Theory of nonlocal electromagnetic fluids.
\newblock {\em Archives of Mechanics}, 33(3):365--384, 1981.

\bibitem{metzler2014anomalous}
R.~Metzler, J.-H. Jeon, A.~G. Cherstvy, and E.~Barkai.
\newblock Anomalous diffusion models and their properties: non-stationarity,
  non-ergodicity, and ageing at the centenary of single particle tracking.
\newblock {\em Physical Chemistry Chemical Physics}, 16(44):24128--24164, 2014.

\bibitem{Nepomn_1991}
S.~Nepomnyaschikh.
\newblock Mesh theorems on traces, normalizations of function traces and their
  inversion.
\newblock {\em Sov. J. Numer. Anal. Math. Modelling}, 6(3):223--242, 1991.

\bibitem{saff1992asymptotic}
E.~Saff and H.~Stahl.
\newblock Asymptotic distribution of poles and zeros of best rational
  approximants to x$\alpha$ on [0, 1].
\newblock {\em Sta}, 299(4):2, 1992.

\bibitem{silling2000reformulation}
S.~A. Silling.
\newblock Reformulation of elasticity theory for discontinuities and long-range
  forces.
\newblock {\em Journal of the Mechanics and Physics of Solids}, 48(1):175--209,
  2000.

\bibitem{stahl2003}
H.~R. Stahl.
\newblock Best uniform rational approximation of $x^\alpha$ on $[0, 1]$.
\newblock {\em Acta Mathematica}, 190(2):241--306, 2003.

\bibitem{Vabishchevich14}
P.~N. Vabishchevich.
\newblock Numerical solving the boundary value problem for fractional powers of
  elliptic operators.
\newblock {\em CoRR}, abs/1402.1636, 2014.

\bibitem{Vabishchevich15}
P.~N. Vabishchevich.
\newblock Numerically solving an equation for fractional powers of elliptic
  operators.
\newblock {\em Journal of Computational Physics}, 282:289--302, 2015.

\bibitem{varga1992some}
R.~S. Varga and A.~J. Carpenter.
\newblock Some numerical results on best uniform rational approximation ofx
  $\alpha$ on [0, 1].
\newblock {\em Numerical Algorithms}, 2(2):171--185, 1992.

\bibitem{XuZ2002}
J.~Xu and L.~Zikatanov.
\newblock The method of alternating projections and the method of subspace
  corrections in {H}ilbert space.
\newblock {\em Journal of the American Mathematical Society}, 15(3):573--597,
  2002.

\bibitem{zaslavsky2002chaos}
G.~M. Zaslavsky.
\newblock Chaos, fractional kinetics, and anomalous transport.
\newblock {\em Physics Reports}, 371(6):461--580, 2002.

\bibitem{ZHAO2017}
X.~Zhao, X.~Hu, W.~Cai, and G.~E. Karniadakis.
\newblock Adaptive finite element method for fractional differential equations
  using hierarchical matrices.
\newblock {\em Computer Methods in Applied Mechanics and Engineering},
  325(Supplement C):56 -- 76, 2017.

\end{thebibliography}


\section{Appendix}\label{append}

\begin{table}[h!]
\caption{{\small The coefficients in the representation \eqref{eq:other} of 1-BURA $ P^\ast_7(t)/Q^\ast_7(t) $
of $ t^{1 -\alpha}$ on $[0,1]$}}\label{tab:App1}
\centering
\begin{tabular}{|| c | c | c | c | c | c | c ||}
\hline \hline
\multirow{2}{*}{j} & \multicolumn{2}{|c|}{$\alpha=0.25$} & \multicolumn{2}{|c|}{$\alpha=0.5$} & 
\multicolumn{2}{|c|}{$\alpha=0.75$}\\ \cline{2-7}
  &$c_j$ & $ d_j$  &$c_j$ & $ d_j$  &   $ c_j$ & $d_j$ \\ \hline
0 & 3.25659E-06 &  0.00000E+00 & 4.60366E-05 &  0.00000E+00 & 7.85127E-04 &  0.00000E+00\\
1 & 1.44761E-04 & -8.74568E-06 & 9.55918E-04 & -3.58368E-07 & 6.54730E-03 & -2.21777E-10\\
2 & 1.08271E-03 & -2.17427E-04 & 4.65253E-03 & -1.93872E-05 & 1.81424E-02 & -7.80406E-08\\
3 & 5.25468E-03 & -2.38575E-03 & 1.63200E-02 & -3.71546E-04 & 4.17928E-02 & -5.55397E-06\\
4 & 2.05418E-02 & -1.77397E-02 & 4.80082E-02 & -4.34363E-03 & 8.61599E-02 & -1.88388E-04\\
5 & 7.43766E-02 & -1.07563E-01 & 1.28889E-01 & -3.80180E-02 & 1.65247E-01 & -4.07531E-03\\
6 & 3.36848E-01 & -6.71407E-01 & 3.73943E-01 & -3.00901E-01 & 3.11865E-01 & -6.65806E-02\\
7 & 1.16449E+01 & -1.55256E+01 & 2.94945E+00 & -4.68768E+00 & 8.94453E-01 & -1.30039E+00\\
 \hline \hline
\end{tabular}
%
\caption{{\small The coefficients in the representation \eqref{eq:other} of 2-BURA $ P^\ast_5(t)/Q^\ast_4(t) $
of $ t^{2 -\alpha}$ on $[0,1]$}}\label{tab:App2}
\centering
\begin{tabular}{|| c | c | c | c | c | c | c ||}
\hline \hline
\multirow{2}{*}{j} & \multicolumn{2}{|c|}{$\alpha=0.25$} & \multicolumn{2}{|c|}{$\alpha=0.5$} & 
\multicolumn{2}{|c|}{$\alpha=0.75$}\\ \cline{2-7}
  &$c_j$ & $ d_j$  &$c_j$ & $ d_j$  &   $ c_j$ & $d_j$ \\ \hline
0,1 & 3.37593E-03 &  0.00000E+00 & 2.34402E-02 &  0.00000E+00 & 1.42137E-01 &  0.00000E+00\\
0,2 & -6.2333E-07 &  0.00000E+00 & -2.0349E-06 &  0.00000E+00 & -3.8415E-06 &  0.00000E+00\\
  1 & 2.40583E-02 & -1.47434E-02 & 7.84172E-02 & -8.08787E-03 & 1.69113E-01 & -3.82073E-03\\
  2 & 8.72123E-02 & -1.22415E-01 & 1.75667E-01 & -7.81739E-02 & 2.20935E-01 & -4.55009E-02\\
  3 & 3.80068E-01 & -7.92754E-01 & 4.54976E-01 & -5.27883E-01 & 3.41427E-01 & -3.37721E-01\\
  4 & 1.30317E+01 & -1.80742E+01 & 3.58723E+00 & -7.18890E+00 & 1.04996E+00 & -3.71162E+00\\
 \hline \hline
\end{tabular}

\caption{{\small The coefficients in the representation \eqref{eq:other} of 2-BURA $ P^\ast_7(t)/Q^\ast_6(t) $
of $ t^{2 -\alpha}$ on $[0,1]$}}\label{tab:App3}
\centering
\begin{tabular}{|| c | c | c | c | c | c | c ||}
\hline \hline
\multirow{2}{*}{j} & \multicolumn{2}{|c|}{$\alpha=0.25$} & \multicolumn{2}{|c|}{$\alpha=0.5$} & 
\multicolumn{2}{|c|}{$\alpha=0.75$}\\ \cline{2-7}
  &$c_j$ & $ d_j$  &$c_j$ & $ d_j$  &   $ c_j$ & $d_j$ \\ \hline
0,1 & 7.38825E-04 &  0.00000E+00 & 7.91901E-03 &  0.00000E+00 & 7.87824E-02 &  0.00000E+00\\
0,2 & -1.8043E-08 &  0.00000E+00 & -7.8577E-08 &  0.00000E+00 & -2.0108E-07 &  0.00000E+00\\
  1 & 5.14919E-03 & -1.91822E-03 & 2.62088E-02 & -9.16055E-04 & 9.33258E-02 & -3.59264E-04\\
  2 & 1.66782E-02 & -1.48538E-02 & 5.50057E-02 & -8.44288E-03 & 1.17911E-01 & -4.15349E-03\\
  3 & 4.59429E-02 & -7.22366E-02 & 1.06623E-01 & -4.61173E-02 & 1.53620E-01 & -2.65144E-02\\
  4 & 1.29584E-01 & -2.99678E-01 & 2.11649E-01 & -2.05570E-01 & 2.09664E-01 & -1.31566E-0\\
  5 & 5.25079E-01 & -1.41237E+00 & 5.39001E-01 & -9.66103E-01 & 3.42629E-01 & -6.42203E-01\\
  6 & 1.81241E+01 & -2.83519E+01 & 4.40913E+00 & -1.12571E+01 & 1.13935E+00 & -5.82558E+00\\
 \hline \hline
\end{tabular}
\end{table}

\begin{table}[thp]
\caption{{\small Numerical error  $\varepsilon$ in \eqref{bound} 
for $(m,k)=(7,7)$ and $\beta=1$. 
For each $\alpha$ the left column shows the results for $\mathbf f$ 
consisting of eigenvectors, while the right column is the error for $\mathbf f$ taken as 1000
random eigenvector combinations.  
In each box averaged error (top) and the 
maximal error (bottom) are reported.}}\label{table:App6}
\centering
\begin{tabular}{|c|cc|cc|cc|}
\hline
\multirow{3}{*}{$h^{-1}$} & \multicolumn{2}{|c|}{$\alpha=0.25$} & \multicolumn{2}{|c|}{$\alpha=0.5$} & 
\multicolumn{2}{|c|}{$\alpha=0.75$}\\ \cline{2-7}
& \multicolumn{2}{|c|}{$E_\alpha(7,7;1)=${\bf 3.2566}{\small E-06}} & \multicolumn{2}{|c|}{$E_\alpha(7,7;1)=${\bf 
4.6037}{\small 
E-05}} & \multicolumn{2}{|c|}{$E_\alpha(7,7;1)=${\bf 7.8966}{\small E-04}}\\ \cline{2-7}
 & $\{\bPsi_i \}_{i=1}^N$ & rand1000 & $\{\bPsi_i \}_{i=1}^N$ & rand1000 & $\{\bPsi_i \}_{i=1}^N$ & rand1000\\ \hline
\multirow{2}{*}{8 }  & 1.9565E-06 & 1.7915E-06 & 2.8431E-05 & 2.7493E-05 & 5.6995E-04 & 6.9670E-04\\
                     & 3.2061E-06 & 2.9876E-06 & 4.6024E-05 & 4.5356E-05 & 7.6989E-04 & 7.6437E-04\\\hline
\multirow{2}{*}{16 } & 2.0032E-06 & 2.4354E-06 & 2.6759E-05 & 2.3804E-05 & 5.2532E-04 & 6.5826E-04\\
                     & 3.1948E-06 & 2.9711E-06 & 4.5812E-05 & 3.7222E-05 & 7.8518E-04 & 7.0151E-04\\\hline
\multirow{2}{*}{32 } & 2.1362E-06 & 2.7011E-06 & 2.9166E-05 & 3.3229E-05 & 5.0559E-04 & 5.9927E-04\\
                     & 3.2522E-06 & 2.9153E-06 & 4.5720E-05 & 3.7850E-05 & 7.8286E-04 & 6.4667E-04\\\hline
\multirow{2}{*}{64 } & 2.0616E-06 & 1.9800E-06 & 2.9487E-05 & 3.9642E-05 & 4.9911E-04 & 4.8683E-04\\
                     & 3.2564E-06 & 2.8259E-06 & 4.6035E-05 & 4.3436E-05 & 7.8744E-04 & 6.3295E-04\\\hline
\multirow{2}{*}{128} & 2.0567E-06 & 1.5541E-06 & 2.9505E-05 & 4.0691E-05 & 4.9672E-04 & 3.5982E-04\\
                     & 3.2565E-06 & 2.6089E-06 & 4.6033E-05 & 4.4048E-05 & 7.8922E-04 & 6.7316E-04\\\hline
\multirow{2}{*}{256} & 2.0675E-06 & 1.8786E-06 & 2.9370E-05 & 3.8431E-05 & 4.9953E-04 & 4.4143E-04\\
                     & 3.2566E-06 & 2.6313E-06 & 4.6029E-05 & 4.3061E-05 & 7.8959E-04 & 7.0087E-04\\\hline
\multirow{2}{*}{512} & 2.0734E-06 & 2.6468E-06 & 2.9312E-05 & 3.0132E-05 & 5.0116E-04 & 6.3817E-04\\
                     & 3.2566E-06 & 3.0750E-06 & 4.6036E-05 & 3.9441E-05 & 7.8965E-04 & 7.4051E-04\\\hline
\multirow{2}{*}{1024}& 2.0736E-06 & 2.3078E-06 & 2.9288E-05 & 2.3497E-05 & 5.0152E-04 & 6.3749E-04\\
                     & 3.2566E-06 & 2.9537E-06 & 4.6037E-05 & 3.7366E-05 & 7.8965E-04 & 7.1731E-04\\\hline
\end{tabular}
\caption{{\small Numerical error  $\varepsilon$ in \eqref{bound} 
for $(m,k)=(7,6)$ and $\beta=2$. 
For each $\alpha$ the left column shows the results for $\mathbf f$ 
consisting of eigenvectors, while the right column is the error for $\mathbf f$ taken as 1000
random eigenvector combinations.  
In each box averaged error (top) and the 
maximal error (bottom) are reported.}}\label{table:App7}
\centering
\begin{tabular}{|c|cc|cc|cc|}
\hline
\multirow{3}{*}{$h^{-1}$} & \multicolumn{2}{|c|}{$\alpha=0.25$} & \multicolumn{2}{|c|}{$\alpha=0.5$} & 
\multicolumn{2}{|c|}{$\alpha=0.75$}\\ \cline{2-7}
& \multicolumn{2}{|c|}{$E_\alpha(7,6;2)=${\bf 1.8043}{\small E-08}} & \multicolumn{2}{|c|}{$E_\alpha(7,6;2)=${\bf 
7.8577}{\small E-08}} & \multicolumn{2}{|c|}{$E_\alpha(7,6;2)=${\bf 2.0108}{\small E-07}}\\ \cline{2-7}
 & $\{\bPsi_i \}_{i=1}^N$ & rand1000 & $\{\bPsi_i \}_{i=1}^N$ & rand1000 & $\{\bPsi_i \}_{i=1}^N$ & rand1000\\ \hline
\multirow{2}{*}{8 }  & 1.0632E-08 & 1.0631E-08 & 5.9400E-08 & 6.0261E-08 & 1.0966E-07 & 1.0468E-07\\
                     & 1.5547E-08 & 1.2642E-08 & 7.8416E-08 & 6.9105E-08 & 2.01E00-07 & 1.5912E-07\\\hline 
\multirow{2}{*}{16}  & 9.5732E-09 & 1.3119E-08 & 5.4368E-08 & 5.3487E-08 & 1.1609E-07 & 1.7237E-07\\
                     & 1.7828E-08 & 1.3955E-08 & 7.7768E-08 & 7.6782E-08 & 2.0032E-07 & 1.8421E-07\\\hline 
\multirow{2}{*}{32}  & 1.1235E-08 & 3.6625E-09 & 5.1080E-08 & 7.5291E-08 & 1.2520E-07 & 6.2571E-08\\
                     & 1.8049E-08 & 1.5797E-08 & 7.8585E-08 & 7.8456E-08 & 2.0100E-07 & 1.9484E-07\\\hline 
\multirow{2}{*}{64}  & 1.1198E-08 & 3.2757E-09 & 5.1082E-08 & 7.7033E-08 & 1.2412E-07 & 2.8021E-08\\
                     & 1.8040E-08 & 1.0433E-08 & 7.8644E-08 & 7.8491E-08 & 2.0090E-07 & 9.7132E-08\\\hline 
\multirow{2}{*}{128} & 1.1468E-08 & 1.6223E-08 & 5.0090E-08 & 6.7111E-08 & 1.2665E-07 & 5.8966E-08\\
                     & 1.8066E-08 & 1.7177E-08 & 7.8573E-08 & 7.7379E-08 & 2.0110E-07 & 1.4951E-07\\\hline 
\multirow{2}{*}{256} & 1.1446E-08 & 3.8759E-09 & 4.9977E-08 & 4.4867E-08 & 1.2795E-07 & 1.8929E-07\\
                     & 1.8062E-08 & 1.7213E-08 & 7.8638E-08 & 6.9319E-08 & 2.0111E-07 & 2.0039E-07\\\hline 
\multirow{2}{*}{512} & 1.1473E-08 & 1.1522E-08 & 4.9994E-08 & 3.1754E-08 & 1.2783E-07 & 6.5736E-08\\
                     & 1.8065E-08 & 1.2699E-08 & 7.8647E-08 & 7.0909E-08 & 2.0111E-07 & 1.9958E-07\\\hline 
\end{tabular}
\end{table}

\begin{table}[htp]
\caption{{\small Numerical error  $\varepsilon$ in \eqref{bound} 
for the multi-step case $(m,k)=(5,5)$ and $\beta=1$. For each setup the left column shows the results for 
$\mathbf f$ consisting of eigenvectors, while the right column is the error for $\mathbf f$ taken as 1000
random eigenvector combinations. In each box averaged error (top) and the 
maximal error (bottom) are reported.}}\label{table:errorMultiStepBeta1}
\centering
\begin{tabular}{|c|cc|cc|cc|}
\hline
\multirow{3}{*}{$h^{-1}$} & \multicolumn{2}{|c|}{$\alpha=0.5$} & \multicolumn{2}{|c|}{$\alpha=0.75$} & 
\multicolumn{2}{|c|}{$\alpha=0.75$}\\ \cline{2-7}
& \multicolumn{2}{|c|}{$\bfu_r=\calA^{-0.25}\left(\calA^{-0.25}\bff\right)$} & 
\multicolumn{2}{|c|}{$\bfu_r=\calA^{-0.25}\big(\calA^{-0.25}\left(\calA^{-0.25}\bff\right)\big)$} & 
\multicolumn{2}{|c|}{$\bfu_r=\calA^{-0.25}\left(\calA^{-0.50}\bff\right)$}\\ \cline{2-7}
 & $\{\bPsi_i) \}_{i=1}^N$ & rand1000 & $\{\bPsi_i) \}_{i=1}^N$ & rand1000 & $\{\bPsi_i) \}_{i=1}^N$ & rand1000\\ \hline
\multirow{2}{*}{16}  & 4.1065E-05 & 4.1827E-05 & 8.7666E-05 & 1.2954E-04 & 1.9972E-04 & 1.7173E-04\\
                     & 9.4745E-05 & 7.2813E-05 & 2.3891E-04 & 2.2187E-04 & 3.9098E-04 & 3.3170E-04\\\hline
\multirow{2}{*}{32}  & 4.7118E-05 & 5.1558E-05 & 1.0939E-04 & 2.0585E-04 & 2.1337E-04 & 2.3155E-04\\
                     & 1.2093E-04 & 1.0361E-04 & 4.8074E-04 & 3.9256E-04 & 5.0294E-04 & 4.1310E-04\\\hline
\multirow{2}{*}{64}  & 5.1447E-05 & 1.0267E-04 & 1.4007E-04 & 7.7816E-04 & 2.3385E-04 & 4.5311E-04\\
                     & 2.0432E-04 & 1.5404E-04 & 1.1396E-03 & 9.9375E-04 & 5.6349E-04 & 5.4832E-04\\\hline
\multirow{2}{*}{128} & 5.7422E-05 & 3.2390E-04 & 2.0400E-04 & 4.1774E-03 & 2.3648E-04 & 4.0115E-04\\
                     & 4.5572E-04 & 4.3974E-04 & 6.2524E-03 & 6.0171E-03 & 5.9066E-04 & 5.4683E-04\\\hline
\multirow{2}{*}{256} & 5.9104E-05 & 4.3670E-04 & 2.3325E-04 & 7.4754E-03 & 2.3848E-04 & 7.7769E-04\\
                     & 5.3831E-04 & 5.2302E-04 & 1.0097E-02 & 9.7472E-03 & 1.0680E-03 & 1.0328E-03\\\hline
\multirow{2}{*}{512} & 5.9310E-05 & 5.9445E-04 & 2.4906E-04 & 1.3916E-02 & 2.4811E-04 & 3.8679E-03\\
                     & 7.7978E-04 & 7.5465E-04 & 1.9729E-02 & 1.9016E-02 & 5.8996E-03 & 5.6791E-03\\\hline
\multirow{2}{*}{1024}& 5.9021E-05 & 4.1906E-04 & 2.4148E-04 & 9.6104E-03 & 2.5064E-04 & 5.3558E-03\\
                     & 8.0305E-04 & 7.2124E-04 & 1.9733E-02 & 1.7263E-02 & 6.8130E-03 & 6.5918E-03\\\hline
\end{tabular}

\caption{{\small Numerical error  $\varepsilon$ in \eqref{bound} 
for the multi-step case $(m,k)=(7,7)$ and $\beta=1$. For each setup the left column shows the results for 
$\mathbf f$ consisting of eigenvectors, while the right column is the error for $\mathbf f$ taken as 1000
random eigenvector combinations. In each box averaged error (top) and the 
maximal error (bottom) are reported.}}\label{table:App4}
\centering
\begin{tabular}{|c|cc|cc|cc|}
\hline
\multirow{3}{*}{$h^{-1}$} & \multicolumn{2}{|c|}{$\alpha=0.5$} & \multicolumn{2}{|c|}{$\alpha=0.75$} & 
\multicolumn{2}{|c|}{$\alpha=0.75$}\\ \cline{2-7}
& \multicolumn{2}{|c|}{$\bfu_r=\calA^{-0.25}\left(\calA^{-0.25}\bff\right)$} & 
\multicolumn{2}{|c|}{$\bfu_r=\calA^{-0.25}\big(\calA^{-0.25}\left(\calA^{-0.25}\bff\right)\big)$} & 
\multicolumn{2}{|c|}{$\bfu_r=\calA^{-0.25}\left(\calA^{-0.50}\bff\right)$}\\ \cline{2-7}
 & $\{\bPsi_i) \}_{i=1}^N$ & rand1000 & $\{\bPsi_i) \}_{i=1}^N$ & rand1000 & $\{\bPsi_i) \}_{i=1}^N$ & rand1000\\ \hline
\multirow{2}{*}{8} & 5.0192E-06 & 5.1628E-06 & 1.0010E-05 & 1.1877E-05 & 3.6246E-05 & 4.0301E-05\\  
                   & 9.8128E-06 & 9.6002E-06 & 2.5169E-05 & 2.4547E-05 & 7.0309E-05 & 6.8761E-05\\\hline 
\multirow{2}{*}{16}& 5.8605E-06 & 1.4026E-05 & 1.5501E-05 & 6.6467E-05 & 3.4149E-05 & 3.7829E-05\\     
                   & 1.9834E-05 & 1.9260E-05 & 9.7937E-05 & 9.4949E-05 & 6.6285E-05 & 5.2376E-05\\\hline
\multirow{2}{*}{32}& 6.7145E-06 & 2.0145E-05 & 2.0351E-05 & 1.2619E-04 & 3.9973E-05 & 8.4991E-05\\  
                   & 2.5714E-05 & 2.5005E-05 & 1.7685E-04 & 1.7140E-04 & 1.1430E-04 & 1.1152E-04\\\hline
\multirow{2}{*}{64}& 6.3275E-06 & 1.8390E-05 & 1.9405E-05 & 1.4991E-04 & 4.3354E-05 & 1.6584E-04\\   
                   & 2.6754E-05 & 2.3775E-05 & 1.8260E-04 & 1.7894E-04 & 2.2211E-04 & 2.1671E-04\\\hline
\multirow{2}{*}{128}& 6.3162E-06 & 1.7655E-05 & 1.9982E-05 & 1.9293E-04 & 4.5322E-05 & 2.5209E-04\\ 
                   & 2.7195E-05 & 2.1459E-05 & 2.3671E-04 & 2.3119E-04 & 3.2535E-04 & 3.1700E-04\\\hline
\multirow{2}{*}{256}& 6.5595E-06 & 3.5085E-05 & 2.4035E-05 & 5.9986E-04 & 4.5112E-05 & 2.3879E-04\\   
                   & 4.5202E-05 & 4.4046E-05 & 8.6567E-04 & 8.3841E-04 & 3.1927E-04 & 2.9358E-04\\\hline
\multirow{2}{*}{512}& 6.7797E-06 & 8.1866E-05 & 2.9931E-05 & 2.1139E-03 & 4.6091E-05 & 4.4815E-04\\   
                   & 1.1361E-04 & 1.0915E-04 & 3.1081E-03 & 2.9789E-03 & 5.7114E-04 & 5.5602E-04\\\hline
\multirow{2}{*}{1024}& 6.8055E-06 & 9.0098E-05 & 3.1472E-05 & 3.0128E-03 & 4.7830E-05 & 1.1992E-03\\  
                   & 1.1366E-04 & 1.0352E-04 & 3.9165E-03 & 3.8048E-03 & 1.6865E-03 & 1.6425E-03\\\hline
\end{tabular}
\end{table}
 
\begin{table}[htp]
\caption{{\small Numerical $\ell^2$-error $\varepsilon$ for $h^{-1}=2^n$ based on \eqref{eq:3D error final} for
$\mathbf  f$ reconstruction with $(m,k)=(5,5)$ and $\beta=1$. For each setup the left column shows the results for 
$\mathbf f$ consisting of eigenvectors, while in the middle and right columns are the errors for $\mathbf f$ taken as 
1000 random eigenvector combinations. In each box averaged error (top) and the 
maximal error (bottom) are reported.}}
\label{table:errorIDCheck}
\centering
\def\arraystretch{0.9}
\begin{tabular}{|c|ccc|ccc|}
\hline
\multirow{3}{*}{$h^{-1}$} & \multicolumn{3}{|c|}{$\calA(\calA^{-0.25}(\calA^{-0.75}\bff))$} & 
\multicolumn{3}{|c|}{$\calA(\calA^{-0.5}(\calA^{-0.5}\bff))$}\\ \cline{2-7}
& \multicolumn{3}{|c|}{$E_{0.75}+E_{0.25}-E_{0.75}E_{0.25}=${\bf 2.7448}E-03} & 
\multicolumn{3}{|c|}{$2 E_{0.5}-E^2_{0.5}=${\bf 5.3784}E-04}\\ \cline{2-7}
 & $\{\bPsi_i) \}_{i=1}^N$ & rand1000  & rand1000 & $\{\bPsi_i) \}_{i=1}^N$ & rand1000 & rand1000\\ \hline
\multirow{2}{*}{16} & 4.7354E-04 & 7.5219E-05 & 7.5147E-05 & 1.0598E-04 & 8.7507E-06 & 8.6706E-06\\
                   & 2.4855E-03 & 8.1696E-05 & 8.2122E-05 & 4.8149E-04 & 1.3993E-05 & 1.5610E-05\\\hline
\multirow{2}{*}{32} & 4.7350E-04 & 2.4231E-05 & 2.4240E-05 & 1.0689E-04 & 4.0082E-06 & 3.9951E-06\\
                   & 2.6751E-03 & 2.5452E-05 & 2.5517E-05 & 5.1127E-04 & 6.8881E-06 & 6.7313E-06\\\hline
\multirow{2}{*}{64} & 4.7775E-04 & 3.4729E-06 & 3.4744E-06 & 1.0850E-04 & 6.9353E-06 & 6.9349E-06\\
                   & 2.7268E-03 & 4.0236E-06 & 3.9443E-06 & 5.3095E-04 & 7.3211E-06 & 7.2493E-06\\\hline
\multirow{2}{*}{128} & 4.7955E-04 & 6.0366E-07 & 6.0448E-07 & 1.0878E-04 & 6.2346E-06 & 6.2351E-06\\
                   & 2.7402E-03 & 8.7980E-07 & 9.4323E-07 & 5.3609E-04 & 6.2777E-06 & 6.2911E-06\\\hline
\multirow{2}{*}{256} & 4.8034E-04 & 5.7357E-07 & 5.7343E-07 & 1.0891E-04 & 2.5050E-06 & 2.5049E-06\\
                   & 2.7437E-03 & 5.8741E-07 & 5.8982E-07 & 5.3740E-04 & 2.5149E-06 & 2.5144E-06\\\hline
\multirow{2}{*}{512} & 4.8074E-04 & 1.2493E-06 & 1.2493E-06 & 1.0895E-04 & 6.7062E-07 & 6.7059E-07\\
                   & 2.7445E-03 & 1.2496E-06 & 1.2496E-06 & 5.3773E-04 & 6.7463E-07 & 6.7458E-07\\\hline
\multirow{2}{*}{1024}& 4.8094E-04 & 2.0673E-07 & 2.0677E-07 & 1.0898E-04 & 6.9352E-07 & 6.9357E-07\\
                   & 2.7447E-03 & 2.0885E-07 & 2.0924E-07 & 5.3781E-04 & 6.9625E-07 & 6.9637E-07\\\hline
\multirow{2}{*}{2048}& 4.8104E-04 & 4.1620E-07 & 4.1620E-07 & 1.0899E-04 & 2.8010E-07 & 2.8010E-07\\
                   & 2.7448E-03 & 4.1646E-07 & 4.1649E-07 & 5.3783E-04 & 2.8034E-07 & 2.8027E-07\\\hline
\multirow{2}{*}{4096}& 4.8109E-04 & 5.3275E-07 & 5.3275E-07 & 1.0900E-04 & 3.6167E-08 & 3.6185E-08\\
                   & 2.7448E-03 & 5.3281E-07 & 5.3282E-07 & 5.3784E-04 & 3.7534E-08 & 3.7333E-08\\\hline
\end{tabular}
\caption{{\small Numerical $\ell^2$-error $\varepsilon$ for $h^{-1}=2^n$ based on \eqref{eq:3D error final} for
$\mathbf  f$ reconstruction with $(m,k)=(7,7)$ and $\beta=1$. For each setup the left column shows the results for 
$\mathbf f$ consisting of eigenvectors, while in the middle and right columns are the errors for $\mathbf f$ taken as 
1000 random eigenvector combinations. In each box averaged error (top) and the 
maximal error (bottom) are reported.}}
\label{table:App5} 
\centering
\def\arraystretch{0.9}
\begin{tabular}{|c|ccc|ccc|}
\hline
\multirow{3}{*}{$h^{-1}$} & \multicolumn{3}{|c|}{$\calA(\calA^{-0.25}(\calA^{-0.75}\bff))$} & 
\multicolumn{3}{|c|}{$\calA(\calA^{-0.5}(\calA^{-0.5}\bff))$}\\ \cline{2-7}
& \multicolumn{3}{|c|}{$E_{0.75}+E_{0.25}-E_{0.75}E_{0.25}=${\bf 7.9291}E-04} & 
\multicolumn{3}{|c|}{$2 E_{0.5}- E^2_{0.5}=${\bf 9.2071}E-05}\\ \cline{2-7}
 & $\{\bPsi_i) \}_{i=1}^N$ & rand1000  & rand1000 & $\{\bPsi_i) \}_{i=1}^N$ & rand1000 & rand1000\\ \hline
\multirow{2}{*}{16}   & 2.7428E-04 & 1.9574E-05 & 1.9581E-05 & 3.6147E-05 & 3.6776E-06 & 3.6796E-06\\
                      & 6.9284E-04 & 2.2126E-05 & 2.1900E-05 & 8.3254E-05 & 3.9151E-06 & 3.9819E-06\\\hline
\multirow{2}{*}{32}   & 2.7502E-04 & 5.8770E-06 & 5.8751E-06 & 3.7098E-05 & 3.6259E-06 & 3.6258E-06\\
                      & 7.6589E-04 & 6.5626E-06 & 6.3826E-06 & 8.5885E-05 & 3.7467E-06 & 3.7298E-06\\\hline
\multirow{2}{*}{64}   & 2.7683E-04 & 1.3628E-06 & 1.3604E-06 & 3.7178E-05 & 2.1428E-06 & 2.1423E-06\\
                      & 7.8591E-04 & 1.5488E-06 & 1.5196E-06 & 9.0463E-05 & 2.1745E-06 & 2.1682E-06\\\hline
\multirow{2}{*}{128}  & 2.7801E-04 & 2.2851E-07 & 2.3003E-07 & 3.7255E-05 & 1.0873E-06 & 1.0873E-06\\
                      & 7.9113E-04 & 3.2669E-07 & 3.2927E-07 & 9.1662E-05 & 1.0939E-06 & 1.0921E-06\\\hline
\multirow{2}{*}{256}  & 2.7853E-04 & 3.4135E-07 & 3.4141E-07 & 3.7286E-05 & 4.9867E-07 & 4.9869E-07\\
                      & 7.9246E-04 & 3.5032E-07 & 3.5252E-07 & 9.1968E-05 & 4.9986E-07 & 5.0030E-07\\\hline
\multirow{2}{*}{512}  & 2.7875E-04 & 3.0366E-07 & 3.0365E-07 & 3.7300E-05 & 1.4905E-07 & 1.4905E-07\\
                      & 7.9280E-04 & 3.0476E-07 & 3.0452E-07 & 9.2045E-05 & 1.4975E-07 & 1.4971E-07\\\hline
\multirow{2}{*}{1024} & 2.7886E-04 & 1.2014E-07 & 1.2014E-07 & 3.7307E-05 & 4.3477E-08 & 4.3476E-08\\
                      & 7.9288E-04 & 1.2029E-07 & 1.2030E-07 & 9.2065E-05 & 4.4327E-08 & 4.4182E-08\\\hline
\multirow{2}{*}{2048} & 2.7892E-04 & 7.9001E-08 & 7.9001E-08 & 3.7312E-05 & 7.2197E-08 & 7.2198E-08\\
                      & 7.9290E-04 & 7.9031E-08 & 7.9026E-08 & 9.2069E-05 & 7.2315E-08 & 7.2295E-08\\\hline
\multirow{2}{*}{4096} & 2.7895E-04 & 5.2692E-09 & 5.2567E-09 & 3.7314E-05 & 1.1877E-08 & 1.1875E-08\\
                      & 7.9291E-04 & 5.8000E-09 & 5.6822E-09 & 9.2071E-05 & 1.1928E-08 & 1.1916E-08\\\hline
\end{tabular}
\end{table}

\end{document}